\documentclass[11pt,a4paper]{article}
\usepackage[a4paper]{geometry}
\usepackage[T1]{fontenc}
\usepackage{amsfonts,amsthm,amssymb,amsmath,mathdots, bbm,mathabx,mathrsfs}
\usepackage[linecolor=white,backgroundcolor=white,bordercolor=white,textsize=tiny]{todonotes}
\usepackage{amssymb, amscd}
\usepackage{todonotes}
\usepackage{color,graphicx}
\usepackage[all,cmtip]{xy}
\usepackage{lmodern} 
\usepackage[colorlinks]{hyperref}
\usepackage{epstopdf}
\usepackage{pdfpages}
\usepackage{pdfsync}
\usepackage{epsfig}
\usepackage[latin1]{inputenc}
\usepackage{subfigure}
\usepackage{cite}
\hypersetup{
linkcolor=blue,
citecolor=blue,
}

\usepackage{silence}
\numberwithin{equation}{section}

\title{Freeness over the diagonal and outliers detection in deformed random matrices with a variance profile}

\author{J\'er\'emie Bigot \hspace{0.2cm} \&  \hspace{0.1cm} Camille Male  \\
\\  Institut de Math\'ematiques de Bordeaux et CNRS  (UMR 5251)   \\ Universit\'e de Bordeaux }

\newtheorem{Th}{Theorem}[section]

\newtheorem{Prop}[Th]{Proposition}

\newtheorem{Lem}[Th]{Lemma}

\newtheorem{Cor}[Th]{Corollary}

\theoremstyle{remark}

\newtheorem{Rk}[Th]{Remark}

\renewcommand\leq\leqslant

\renewcommand\geq\geqslant

\def\Tr{\mathrm{Tr}}

\def\esp{\mathbb E}

\def\Pr{\mathbb P}

\def\var{\mathbb V\textrm{ar}}

\def\etc{,\ldots ,}
\def\one{\mathbbm{1}}

\def\MN{\mathrm{M}_N(\mathbb C) }
\def\HN{\mathrm{H}_N(\mathbb C) }
\def\DN{\mathrm{D}_N(\mathbb C) }

\def\i{\mathbf i}

\def\DD{\mathcal D}

\def\R{\mathbb R}
\def\C{\mathbb C}

\def\eps{\varepsilon}

\def\eq{\begin{eqnarray*}}
\def\qe{\end{eqnarray*}}
\def\eqa{\begin{eqnarray}}
\def\qea{\end{eqnarray}}

\newcommand\at[2]{\left.#1\right|_{#2}}

\def\diag{\mathop{\rm diag}\nolimits}%

\def\mbf{\mathbf}
\def\mcal{\mathcal}

\def\mbb{\mathbb}

\def\mrm{\mathrm}

\def\bor{\begin{color}{orange}}
\def\eor{\end{color}}

\def\bte{\begin{color}{teal}}
\def\ete{\end{color}}

\newcommand{\bG}{\boldsymbol{G}}
\newcommand{\bone}{\mathbf{1}}    
\begin{document}
\maketitle

\begin{abstract}
We study the eigenvalue distribution of a GUE matrix with a variance profile that is  perturbed by  an additive random matrix  that may possess spikes. Our approach is guided by Voiculescu's notion of freeness with amalgamation over the diagonal and by the notion of deterministic equivalent. This allows to derive a fixed point equation  to approximate the spectral distribution of certain deformed GUE matrices with a variance profile and to characterize the location of potential outliers in such models in a non-asymptotic setting. We also consider the singular values distribution of a rectangular Gaussian random matrix with a variance profile in a similar setting of additive perturbation. We discuss the application of this approach to the study of low-rank matrix denoising models in the presence of heteroscedastic noise, that is when the amount of variance in the observed data matrix may change from entry to entry. Numerical experiments are used to illustrate our results. 

\end{abstract}

\noindent \emph{Keywords:}  Deformed random matrix, Variance profile, Outlier detection, Free probability, Freeness with amalgamation, Operator-valued Stieltjes transform, Gaussian spiked model,  Low-rank model.

\tableofcontents

\section{Introduction}

We first introduce some questions related to the estimation of the eigenvalue distribution of the sum of a GUE matrix with a variance profile and a deterministic matrix that may possess spikes. Then, we present the main theoretical contributions of the paper and its organization.

\subsection{Aim of the article}
\subsubsection{Deformed Hermitian random matrices}

We recall that for any integer $N\geq 1$, a $N \times N$ GUE matrix is a Hermitian random matrix $ \big(\frac{x_{i,j}}{\sqrt N}\big)_{i,j=1\etc N}$, such that  the sub-diagonal entries are independent,   the variables $x_{i,i}$, $i=1\etc N$, are centered real Gaussian variables with variance one,
and  the non-diagonal entries $x_{i,j}$ (with $i\neq j$) are centered complex Gaussian variables with variance one. Let $\Gamma_N = \big(\gamma_N^2(i,j)\big)_{i,j=1\etc N}$ be a symmetric matrix with non negative entries. Then, the random matrix
\begin{equation}
	X_N := \left( \gamma_N(i,j)\frac{x_{i,j}}{\sqrt N}\right)_{i,j=1\etc N} \label{def:GUEvarprof}
\end{equation}
is called a \emph{GUE matrix of size $N$ with variance profile} $\Gamma_N$. Furthermore, we consider a deterministic Hermitian matrix $Y_N$. The random matrix 
	\eqa\label{Model1}
		H_N  := X_N + Y_N
	\qea
is then an (additively) \emph{deformed random matrix}. We shall also study the situation where $Y_N$  has \emph{spikes} that may generate a finite number of eigenvalues (called \emph{outliers}) that detach from the rest of the spectrum. 

There are two dual ways to interpret  model \eqref{Model1} in RMT. From the point of view of mathematical physics, the random matrix $X_N$ models the Hamiltonian of a system, and the deformation $Y_N$ is a perturbation coming from an external source. A variance profile for $X_N$ can model impurities of the system. From the signal processing point of view, $Y_N$ models a signal and $X_N$ an additive noise which deforms the data. 

 In this paper, we present results and heuristics to approximate the eigenvalue distribution of a deformed random matrix by a deterministic function, called classically a \emph{deterministic equivalent}. Our approach involves tools in RMT and free probability, some of them having been created quite recently. They may have applications for the numerical analysis of the spectral distribution  of data organized in a matrix form such as noisy images. 

The substance of our mathematical arguments is a legacy of the investigations of Haagerup and Thorbj{\o}rnsen who developed a strategy in \cite{HT05} for the study of block GUE matrices, involving classical techniques of Gaussian calculus and complex analysis. In their seminal paper, such a strategy served as an ingredient to reinforce the link between RMT and the abstract theory of $\mcal C^*$-algebras of free groups. The analysis in \cite{HT05} and in our work are both based on Voiculescu's equations of subordination with amalgamation. Adapting Haagerup and Thorbj{\o}rnsen's technique, we want to spread their method to more applied contexts that we believe to be of interest in high-dimensional statistics and signal processing. We notify the reader that the aim of this work is not to obtain sharp estimations (that is with optimal rates of convergence) of the spectral distribution of deformed random matrix with a variance profile, and we refer to \cite{EKS19} for recent results on this topic.

\subsubsection{Information plus noise model}

If  model \eqref{Model1} is quite natural from the mathematical point of view, it is more interesting for application in statistics to study rectangular matrices with no symmetry. We define a standard Gaussian matrix as a $N \times M$ rectangular random matrix  $\big(\frac{x_{i,j}}{\sqrt M} \big)_{i,j}$ with $x_{i,j}$ independent centered complex Gaussian entries with variance one for all $i,j$. Let $\Gamma_{N,M} = \big(\gamma_{N,M}^2(i,j)\big)_{i,j}$ be a $N \times M$ matrix with non negative entries. Then, the random matrix 
	$$X_{N,M} = \left( \gamma_{N,M}(i,j)\frac{x_{i,j}}{\sqrt M}\right)_{\substack{ i=1\etc N\\j=1\etc M}}.$$
is called a \emph{Gaussian matrix with variance profile} $\Gamma_{N,M}$. Let $Y_{N,M}$ be a deterministic matrix of size $N\times M$.  The random matrix 
	\eqa\label{Model3}
		H_{N,M}  =  X_{N,M} + Y_{N,M}
	\qea
is called an \emph{information plus noise model with a variance profil}.

When $Y_{N,M}$ is a finite rank matrix, then the rectangular information plus noise model \eqref{Model3} corresponds to the  low-rank matrix denoising problem which arises in various applications, where it is of interest to estimate a $N$ by $M$ signal matrix $Y_{N,M}$ from noisy data.  When the variance profile $\Gamma_{N,M}$ is a matrix with equal entries, the problem of estimating the low-rank matrix $Y_{N,M}$ has been extensively studied in statistics and machine learning \cite{MR3054091,MR3105401,donoho2014,MR3200641} using spectral estimators constructed  from the singular value decomposition of $H_{N,M}$. These works build upon well understood results of the asymptotic behavior (as $\min(N,M) \to \infty$) of the singular values of $H_{N,M}$ in the Gaussian spiked population model \cite{Benaych-GeorgesN12,MR2322123}. 

Hence, low-rank matrix estimation is  well understood when the additive noise is Gaussian with \emph{homoscedastic variance} (that is a  constant variance profile). However, in many applications the  noise can be highly \emph{heteroscedastic} meaning that the amount of variance in the observed data matrix may significantly change from entry to entry. Examples can be found in photon imaging \cite{SalmonHDW14}, network traffic analysis \cite{Bazerque2013} or genomics for microbiome studies \cite{Cao17}. In such applications, the observations are count data that are modeled by Poisson or multinomial distributions which leads to heteroscedasticity. The literature on statistical inference from high-dimension matrices with heteroscedastic noise has thus recently been growing \cite{Bigot17,liu2018,MAL2016,Zhang18,Fan19}. 

As discussed in \cite{Zhang18}, the Gaussian model \eqref{Model3} with a variance profile can serve as a prototype for various applications involving low-rank matrix denoising in the presence of heteroscedasticity. The analysis of the outliers in model \eqref{Model3} with a non-constant variance profile is also of interest in neural networks to model synaptic matrix \cite{Rajan06} and in signal processing for radio communications \cite{Chapon14}. Nevertheless, to the best of our knowledge, the characterization of potential outliers in model  \eqref{Model3} when the noise matrix $X_{N,M}$ has a variance profile that is not constant has not received very much attention so far.

We now address the two questions on deformed random matrices we investigate.

\subsubsection{Global behavior of the eigenvalues of an additive perturbation} \label{subsec:global}

The first issue is  the global behavior of the eigenvalues of $H_N$. In particular, under general assumptions on $\Gamma_N$ and $Y_N$, we shall consider the question of how  approximating the empirical spectral distribution (e.s.d.) of $H_N$:
		$$\mu_{H_N} := \frac 1 N \sum_{i=1}^N \delta_{\lambda_i(H_N)},$$
where $\delta$ is the Dirac mass  and $\lambda_i(H_N)$ the $i$-th eigenvalue of $H_N$. As commonly done in random matrix theory (RMT), we use the Stieltjes transform $g_{H_N}$ of the e.s.d.\ of $H_N$, namely the map defined by the trace of the resolvent $(\lambda \mbb I_N - H_N)^{-1}$ of $H_N$ that is
	\eqa\label{Def:Stieltjes}
		g_{H_N} (\lambda) :=  \frac 1 N \Tr \big[ (\lambda \mbb I_N - H_N)^{-1}\big].
	\qea
In above formula, $\lambda$ belongs to the set $\mbb C^+$ of complex number with positive imaginary part. When the context is clear, we write $\lambda - H_N$ for the matrix $\lambda \mbb I_N-H_N$. This transform has numerous properties and applications, we refer to Benaych-Georges and Knowles lecture notes \cite[Section 2.2]{BK16} for an introduction. Note that for $\lambda = t+\mbf i \eta$ (with $t \in \R$ and $\eta > 0$) then $\frac{1}\pi \Im m \, g_{H_N}(\lambda) = (\mu_{H_N}* \rho_\eta)(t)$ where $*$ denotes the convolution product and $\rho_\eta(t) = \frac{\eta}{t^2+ \eta^2}$ is an approximate delta function (Cauchy kernel) as $\eta \rightarrow 0^+$. Hence for $\eta$ chosen sufficiently small, the function
 \begin{equation} \label{eq:Stieltjesinv}
 t \mapsto -\frac{1}\pi \Im m \left( g_{H_N}(t+\mbf i \eta) \right)
\end{equation}
is a good approximation of $\mu_{H_N}$ over $\R$ by a random analytic function. In what follows, the smooth density defined by \eqref{eq:Stieltjesinv} will be referred to as the {\it inverse Stieltjes transform}  of $g_{H_N}$.

Then, we first address the question of the construction of a deterministic analytic function $g_{H_N}^\square:\mbb C^+ \to \mbb C^-$, which only depends  on $N$, on the variance profile $\Gamma_N$ and  on $Y_N$, and that approximates $g_{H_N}$ with high probability. 

\subsubsection{Outliers location in finite rank perturbation models} \label{subsec:outlierdetection}

The second issue is the effect of a low rank additive perturbation. For notational convenience, we introduce another deterministic Hermitian matrix $Z_N$ of finite rank $k$ (not depending on $N$). We can thus write $Z_N =  U_{N,k} \Theta_k U_{N,k}^*$, where $\Theta_k$ is a $k \times k$ diagonal matrix with nonzero diagonal entries and $U_{N,k}$ is a $N\times k$ matrice whose columns are  orthonormal vectors. The eigenvalues of $Z_N$ are commonly called the \emph{spikes}. The matrix 
	\eqa\label{Model2}
		H'_N  = X_N + Y_N + Z_N = H_N  + U_{N,k} \Theta_k U_{N,k}^*
	\qea
is a \emph{finite rank deformation} of $H_N$. For such models, we consider the problem of how locating in the spectrum of $H_N'$ the eigenvalues coming from $Z_N$  that detach from the spectrum of $H_N$.

To formulate more precisely this problem, we say that, for a given level of precision $\eps>0$, an \emph{outlier} is a real eigenvalue $t$ of $H_N'$ which is not in an $\eps$-neighborhood of the spectrum of $H_N$. Following  \cite{Benaych-GeorgesN12}, let us first \emph{factorize the spectrum of} $H_N$ in $H'_N$, in the following sense. For any complex $\lambda$ which is not in the spectrum of $H_N$, we consider the decomposition
	$$ ( \lambda    - H'_N ) = (\lambda   - H_N) \times \alpha_N(\lambda),$$
where 
	$$\alpha_N(\lambda) = \mbb I_N -  (\lambda   - H_N)^{-1}Z_N.$$
An outlier is then a real number $t \in \R$ away from the spectrum of $H_N$ such that $\det\big(\alpha_N(t) \big) = 0$. Therefore, noticing that the determinant of $\alpha_N(\lambda)$ is equal to the determinant of the $k \times k$ matrix 
	\eqa\label{Eq:Beta}
		\beta_k(\lambda) :=  \mbb I_k - U_{N,k}^*(\lambda   - X_N-Y_N)^{-1}  U_{N,k} \Theta_k,
	\qea
it follows that, to  localize the outliers of $H_N'$, it  is sufficient to compute the real numbers $t \in \R$ such that
\begin{equation}
\det\big(\beta_k(t) \big) = 0. \label{eq:detbeta}
\end{equation}
Then, the second question that we address in this paper is the construction of a deterministic  matrix-valued function $\beta_k^\square$, which depends on $N$, on the variance profile $\Gamma_N$, and on the matrices $Y_N,Z_N$, and that approximates $\beta_k$. In this manner, the zeros of $\beta_k^\square$ that are away from the spectrum of $H_N$ shall indeed be close to the outliers of $H_N'$.

\subsection{Main statements}\label{Sec:MainStat}

We now provide the formal statements of our main results and the method to answer the two questions raised above. We let $\mrm D_N(\mbb C^+)$  (resp. $\mrm D_N(\mbb C)^-$) denote the set of diagonal matrices $\Lambda=\big(\Lambda(i,j)\big)_{i,j}$ of size $N$ with diagonal entries having positive imaginary (resp. negative) parts. For any matrix $A_N=(a_{i,j})_{i,j}$, we denote by $\Delta(A_N)$ the diagonal matrix whose diagonal entries are those of $A_N$. 

The operator-valued Stieltjes transform $G_{A_N}$ of a Hermitian matrix $A_N$ is the map 
\begin{equation} \label{eq:opStieltjes}
	\begin{array}{cccc} 
		G_{A_N}: & \DN^+ & \to & \DN^-\\
		& \Lambda & \mapsto & \Delta\big[ (\Lambda - A_N)^{-1} \big].\end{array} 
\end{equation}
For any $\Lambda \in \mrm D_N(\mbb C^+)$, we also introduce the mapping
\begin{equation}
		\mcal R_N(\Lambda) =   \underset{i=1\etc N}{\mrm{diag}} \Big( \sum_{j=1}^N \frac {\gamma_N^2(i,j)}N \Lambda(j,j)\Big). \label{eq:defRNbis}
\end{equation}
that is a key tool in our analysis. The map $\mcal R_N$ may also be written as
\begin{equation}
\mcal R_N(\Lambda) = \mrm{deg}\big(  \frac{\Gamma_N}N  \Lambda\big) =\esp [  X_N \ \Lambda \ X_N  ], \label{eq:RNdeg}
\end{equation}
where $\mrm{deg}(A)$ for a matrix $A$ is the diagonal matrix whose $k$-diagonal element is the sum of the entries of the $k$-row of $A$. We now state our main result on the construction of a deterministic equivalent of the operator-valued Stieltjes transform of deformed random matrices and its finite sample properties of approximation.
   
\begin{Th}\label{MainTh} There exists a unique function $G_{H_N}^\square : \DN^+ \to \mrm D_N(\mbb C)^-$,  analytic in each variable, that solves of the following fixed point equation
	\eqa\label{FixPtEq1} 
		  G_{H_N}^\square({ \Lambda}) =\Delta\bigg [ \Big( \Lambda   -\mcal R_N\big(  G_{H_N}^\square({ \Lambda})\big)  - Y_N \Big)^{ -1}\bigg],
	\qea
for any $\Lambda \in  \DN^+$. Let $\gamma^2_{\mrm {max}} = \max_{i,j}  \gamma_N^2(i,j)$, let  $0 < \delta < 1$, and consider $\Lambda \in \DN^+$ satisfying
\begin{equation}
\Im m \, \Lambda \geq  \gamma_{\mrm{max}}  \left(\frac{2}{N(1-\delta)}\right)^{1/5} \mbb I_N. \label{eq:condlambda0}
\end{equation}
Then, for any $d > 1$, setting
	\eq
		\eps_N(d) &: =&  \sqrt{2} \gamma_{\max}   \sqrt{\frac{d \log(N)}N}  \| (\Im m \, \Lambda)^{-1}\| ^2  \\
		&&+   \Big( 1 +  \frac{\gamma^2_{\mrm {max}}}{\delta}   \| (\Im m \, \Lambda)^{-1} \|^2 \Big)  \frac{2 \gamma_{\mrm{max}}^3\| (\Im m \, \Lambda)^{-1}\|^4}N,
	\qe
we have, for $N \geq 1$,
\begin{equation}
\mbb P\Big(  \big\| G_{H_N}(\Lambda) - G_{H_N}^\square(\Lambda) \big\| \geq \eps_N(d)  \Big) \leq 4 N^{1-d}, \label{eq:MainThStieltjes}
\end{equation}
where $\| \cdot \|$ denotes the operator norm of a matrix.
	
If moreover $Y_N$ is diagonal, for any $\Lambda \in \DN^+$ satisfying
\begin{equation}
\Im m \, \Lambda \geq\gamma_{\mrm{max}} N^{-1/4} (1-\delta)^{-1/6} \mbb I_N \label{eq:condlambda2}
\end{equation}
then \eqref{eq:MainThStieltjes} holds with 
	\eq
		\eps_N(d) &:= &  \sqrt{2} \gamma_{\max}   \sqrt{\frac{d \log(N)}N}  \| (\Im m \, \Lambda)^{-1}\| ^2  \\
		&&+   \big( 1 +  \frac{\gamma^2_{\mrm {max}}}{\delta}   \| (\Im m \, \Lambda)^{-1} \|^2 \big) \gamma_{\mrm{max}}^4 \frac{ \| (\Im m \, \Lambda)^{-1}\|^5}{N^{\frac 3 2}}.
	\qe
\end{Th}
The solution $G_{H_N}^\square$ of the fixed point equation is referred to as the {\it deterministic equivalent} of the operator-valued Stieltjes transform of $H_N$. It has another description given in Section \ref{sec:large}, as the limit of a functional on large random matrices. One may remark that the results of Theorem \ref{MainTh} hold without any assumption on the Hermitian matrix $Y_N$ and the variance profile $\Gamma_N$. In particular no bound from below for the entries $\Gamma_N$ is involved to derive the concentration inequality \eqref{eq:MainThStieltjes}. If we are given sequences of matrices with growing dimension $N$, then this estimate is also meaningful when $\gamma_{\max}$ grows slowly with $N$.

The proof of Theorem \ref{MainTh} is divided into three steps that are detailed in Section \ref{Sec:AsympSubProp}, Section \ref{Sec:FixedPtAnalysis} and Section \ref{Sec:resolventAnalysis}. We then deduce the following methods to answer the two issues raised in Section \ref{subsec:global} and Section \ref{subsec:outlierdetection}.

\paragraph{Approximation of the global behavior of the e.s.d.\ of $H_N$ by a smooth density.}  The proof of Theorem 1.1.\ with a slightly different argument of concentration implies the following approximation for the Stieltjes transform of $H_N$ (recall that $\gamma^2_{\mrm {max}} = \max_{i,j}  \gamma_N^2(i,j)$).

\begin{Cor}\label{MainCor}
For any $\lambda \in \mbb C^+$ such that
\begin{equation}
\Im m \, \lambda \geq  \gamma_{\mrm{max}}  \left(\frac{2}{N(1-\delta)}\right)^{1/5} , \label{eq:condlambda2bis}
\end{equation}
denoting for $d >0$
$$
\tilde{\eps}_N(d) :=    \frac{\sqrt{2} \gamma_{\max} \sqrt{2d \log(N)}}{  | \Im m \, \lambda|^2 N}  +   \Big( 1 +  \frac{\gamma^2_{\mrm {max}}}{\delta | \Im m \, \lambda|^{2}}    \Big) \frac{2 \gamma_{\mrm{max}}^3}{ N   (\Im m \, \lambda)^4},
$$
and $g_{H_N}^\square(\lambda) := \frac 1 N \Tr G_{H_N}^\square(\lambda \mbb I_N)$, it follows that
\begin{equation}
\Pr \big( \big| g_{H_N} (\lambda)   - g_{H_N}^\square (\lambda) \big| \geq \tilde{\eps}_N(d) \big) \leq N^{-d}, \label{eq:MainThStieltjes2}
\end{equation}
where $G_{H_N}^\square(\lambda \mbb I_N)$ is the unique solution of \eqref{FixPtEq1} for the scalar matrix $\Lambda = \lambda \mbb I_N$. 
If moreover $Y_N$ is diagonal and $\Im m \, \lambda \geq \gamma_{\mrm{max}} N^{-1/4} (1-\delta)^{-1/6}$, then \eqref{eq:MainThStieltjes2} holds with 
$$
\tilde{\eps}_N(d) :=    \frac{\sqrt{2} \gamma_{\max} \sqrt{2d \log(N)}}{  | \Im m \, \lambda|^2} N^{-1} +   \Big( 1 +  \frac{\gamma^2_{\mrm {max}}}{\delta | \Im m \, \lambda|^{2}}    \Big) \frac{ \gamma_{\mrm{max}}^4}{ N^{3/2}   (\Im m \, \lambda)^5},
$$
\end{Cor}

 The concentration inequality \eqref{eq:MainThStieltjes2} means that the  deterministic equivalent  $g_{H_N}^\square(\lambda) = \frac 1 N \Tr \, G_{H_N}^\square(\lambda \mbb I_N)$ is a good approximation of   $g_{H_N}(\lambda)$ when the imaginary part of $\lambda$ is allowed to decay to zero as the dimension $N$ grows at a rate given by the right hand side of Inequality \eqref{eq:condlambda2bis}. Hence, by the   inverse Stieltjes transform formula \eqref{eq:Stieltjesinv}, the map $t \mapsto  \frac{1}\pi \Im m \left( g_{H_N}^\square(t+\mbf i\eta) \right)$ is  a good approximation for the e.s.d.\ of $H_N$ when $\eta$ is small. However, compared to existing results in the RMT literature, the lower bound \eqref{eq:condlambda2bis} on $\Im m \, \lambda$ is not optimal, and thus Corollary \ref{MainCor} can only be interpreted as a {\it weak local law}. More generally, the lower bound condition for $\Im m \, \Lambda$ in Theorem  \ref{MainTh}  means that all diagonal entries of $\Lambda$ must have an imaginary part that is sufficiently large as shown by the Conditions \eqref{eq:condlambda0} and \eqref{eq:condlambda2}. Therefore, Theorem  \ref{MainTh} may  be interpreted as an \emph{operator-valued} weak local law for the operator-valued Stieltjes transform of $H_N$. In Section \ref{sec:Dyson}, we  discuss  more precisely the connections between our work and existing results on local laws for random matrices with a variance profile.

\paragraph{Outliers localization in the case where $Y_N$ is diagonal.} From its definition, the natural way to construct a deterministic equivalent for $\beta_k(\lambda)$ is to replace $(\lambda \mbb I_N - X_N-Y_N)^{-1}$ in expression  \eqref{Eq:Beta} by an appropriate deterministic estimate. When $Y_N$ is diagonal, then for any $\Lambda \in \DN^+$,  the matrix $\esp\big[ (\Lambda - X_N-Y_N)^{-1}\big]$ is also diagonal (thanks to Corollary \ref{Cor:ExpOutDiag}). Using a concentration inequality, one has  that $(\Lambda - X_N-Y_N)^{-1}$ is thus close to a diagonal matrix, and so we can replace this generalized resolvent with the deterministic equivalent $G^\square_{X_N+Y_N}(\Lambda)$ of the operator-valued Stieltjes transform. 

\begin{Cor}\label{Cor:Beta1} We assume that $Y_N$ is diagonal and define the  deterministic matrix valued-function function
	\begin{equation}
\beta_k^\square(\lambda) =   \mbb I_k - U_{N,k}^*G_{H_N}^\square( \lambda \mbb I_N)  U_{N,k} \Theta_k, \mbox{ for } \lambda \in \mbb C^+.\label{eq:defbetasquare}
	\end{equation}
Then, for any $\lambda$ such that $\Im m \, \lambda \geq \gamma_{\mrm{max}} N^{-1/4} (1-\delta)^{-1/6}$, then $\beta_k^\square(\lambda)$ is a deterministic equivalent of $\beta_k(\lambda)$, in the  sense that
\begin{equation}
\Pr \left( \| \beta_k(\lambda) - \beta_k^\square(\lambda) \|  \geq \| \Theta_k \| \eps'_N(d)  \right) \leq 4 k^2 N^{-d}, \label{Goal}
\end{equation}
where, for any $d > 0$,
$$
\eps'_N(d) :=  \sqrt{2} k \gamma_{\max}  \frac{\sqrt{d \log(N)}}{  | \Im m \, \lambda|^2} N^{-1/2} + \Big( 1 +  \frac{\gamma^2_{\mrm {max}}}{\delta | \Im m \, \lambda|^{2}}    \Big) \frac{ \gamma_{\mrm{max}}^4}{ N^{3/2}   (\Im m \, \lambda)^5}.
$$
\end{Cor}
The proof of  Inequality  \eqref{Goal} can be found in Section \ref{Sec:resolventAnalysis}.

\paragraph{Outliers localization, general case.} In general, if $Y_N$ is not diagonal then the expectation of the generalized resolvent, that is $\esp\big[ (\Lambda - X_N-Y_N)^{-1}\big]$, is no longer a diagonal matrix. Hence it is no longer correct to approximate $\beta_k(\lambda)$ by replacing $(\Lambda - X_N-Y_N)^{-1}$ with $G_{H_N}^\square( \lambda \mbb I_N)$ in Equation \eqref{Eq:Beta}. Yet, it is proved in Section \ref{Sec:AsympSubProp} that the generalized resolvent  $(\Lambda - X_N-Y_N)^{-1}$ can be approximated by the following deterministic matrix
\begin{equation}
(\Lambda - X_N-Y_N)^{-1} \approx   \big(\Omega_{X_N,Y_N}(\Lambda)-Y_N \big)^{-1} , \label{eq:approxresolvent}
\end{equation}
with  $\Omega_{X_N,Y_N}(\Lambda) :=  \Lambda - \mcal R_N\big(  \esp\big[G_{X_N+Y_N}(\Lambda)\big]  \big)$. It appears that the matrix $\Omega_{X_N,Y_N}(\Lambda)$ can be interpreted as an  approximate operator-valued subordination function, and we refer to Sections \ref{sec:freeapproach} and \ref{Sec:LitteratureRMTvp} for further details and discussion on this heuristic.

\begin{Cor}\label{Cor:Beta2}
We define the deterministic matrix valued-function function
\begin{equation}
\tilde{\beta}_k^\square(\lambda) =   \mbb I_k - U_{N,k}^* \big(\Omega_{H_N}^\square(\lambda \mbb I_N)-Y_N \big)^{-1}   U_{N,k} \Theta_k, \mbox{ for } \lambda \in \mbb C^+,\label{eq:defbetasquaregen}
\end{equation}
where
$$
\Omega_{H_N}^\square(\Lambda) := \Lambda - \mcal R_N\big(  G_{H_N}^\square(\Lambda)  \big),
$$
with $G_{H_N}^\square(\Lambda)$ solution of the  fixed point Equation \eqref{FixPtEq1}. 
Then, for any $\lambda$ such that  $\Im m \, \lambda \geq  \gamma_{\mrm{max}}  \left(\frac{2}{N(1-\delta)}\right)^{1/5}$, then $\beta_k^\square(\lambda)$ is a deterministic equivalent of $\beta_k(\lambda)$, in the  sense that
\begin{equation}
\Pr \left( \| \beta_k(\lambda) - \tilde{\beta}_k^\square(\lambda) \|  \geq \| \Theta_k \| \bar{\eps}_N(d)  \right) \leq 4 k^2 N^{-d}, \label{Goal2}
\end{equation}
where, for any $d > 0$,
$$
\bar{\eps}_N(d) :=  \sqrt{2} k \gamma_{\max}  \frac{\sqrt{d \log(N)}}{  | \Im m \, \lambda|^2} N^{-1/2} + \frac{2 \gamma_{\mrm{max}}^3}{ N   (\Im m \, \lambda)^4} +  \Big( 1 +  \frac{\gamma^2_{\mrm {max}}}{\delta | \Im m \, \lambda|^{2}}    \Big) \frac{2 \gamma_{\mrm{max}}^5}{ N   (\Im m \, \lambda)^6}.
$$
\end{Cor}

\begin{Rk} 
A key argument developed to obtain the proof of Inequality \eqref{Goal2} is the derivation of an upper bound on the difference in operator norm between  $\esp\big[  (\lambda \mbb I_N - X_N-Y_N)^{-1} \big]$  and its approximation by $\big(\Omega_{H_N}^\square(\lambda \mbb I_N)-Y_N \big)^{-1}$. In the general case, when $\lambda$ is held fixed, this error term decays at the rate $N^{-1}$ whereas, in the case where $Y_N$ is diagonal, the operator norm of the difference between  $\esp\big[  (\lambda \mbb I_N - X_N-Y_N)^{-1} \big]$  and $G_{H_N}^\square( \lambda \mbb I_N)$ decays at the rate $N^{-3/2}$. Note that an integrable decay such as $N^{-(1+\eta)}$, $\eta>0$, implies the convergence of the spectrum also called strong convergence \cite{Mal11}: almost surely for $N$ large enough the spectrum of $X_N+Y_N$ belongs to a small neighborhood of the support of the measure whose Stieltjes transform is $g_{H_N}^\square$. Under mild assumptions, the strong convergence of $H_N$ is proved in \cite[Corollary 2.3]{EKS19} with $Y_N$ possibly non diagonal.
\end{Rk}

\begin{Rk} In general, even when $Y_N=0$ the locations of outliers possibly depend  on the eigenvectors of $Z_N$, and there is no longer a canonical way to associate an outlier to a specific spike as it is the case when the variance profile has constant entries. We shall discuss this specific case in the literature review proposed in Section  \ref{sec:review}, and  this property will be illustrated by numerical experiments in Section \ref{sec:num}.
\end{Rk}

\subsection{Organization of the paper}

In Section \ref{sec:review} we relate the approach followed in this paper to the existing literature in RMT and free probability on deformed models and outlier detection. Then, we report the results of numerical experiments in Section \ref{sec:num} to shed some light on the benefits of our approach to localize potential outliers in information plus noise models with a variance profile. We present the strategy  of the proof of Theorem \ref{MainTh} and its organization  in Section \ref{sec:organization}. The details of the main steps of the proof are then gathered in Section \ref{Sec:AsympSubProp}, Section \ref{Sec:FixedPtAnalysis} and Section \ref{Sec:resolventAnalysis}.

\paragraph{Acknowledgements.} J\'er\'emie Bigot is a member of Institut Universitaire de France (IUF), and this work has been carried out with financial support from the IUF. Camille Male received the support of the Simons CRM Scholar-in-Residence Program. We would like to also thank László Erdös for his remarks on a previous version of this paper and for pointing out recent references.

\subsection{Publicly available source code}

For the sake of reproducible research, Python scripts available at the following address
\url{https://www.math.u-bordeaux.fr/~jbigot/Site/Publications\_files/ScriptsSpikesVarianceProfile.zip}
allow to reproduce the numerical experiments carried out in this paper.

\section{Related literature and methods} \label{sec:review}

\subsection{Additive perturbations and the standard spiked population model}

We first discuss the case when all entries of the variance profile equal one, that is $X_N$ is a {\it standard GUE matrix}. The celebrated Wigner's Theorem \cite{Wig58} states that the empirical spectral distribution of $X_N$ converges to the semicircular distribution $\mu_x:= (2\pi)^{-1}\sqrt{ 4 - t^2}\one_{[-2,2]}\mrm d t$. Pastur studies in \cite{Pas72} the global behavior of a GUE matrix with additive perturbation.  Assuming that the e.s.d.\ $\mu_{Y_N}$ of $Y_N$ converges to a measure $\mu_y$, then almost surely the e.s.d.\ of $X_N+Y_N$ converges to a measure denoted $\mu_{x+y}$. In general, there is no explicit description of $\mu_{x+y}$, but the limiting Stieltjes transform $g_{x+y}$ of $X_N+Y_N$ satisfies the so-called Pastur's equation, which is expressed in terms of the Stieltjes transform $g_y$ of $\mu_y$ as follows: for all $\lambda$ in $\mbb C^+$, we have
	\eqa\label{Eq:PasturEq}
		g_{x+y}(\lambda) = g_y\big( \lambda - g_{x+y}(\lambda) \big).
	\qea
Note that $Y_N=0$ implies $\mu_y=\delta_0$ and the above equation reads $g_x(\lambda) = \big(\lambda - g_x(\lambda) \big)^{-1}$.

In practical applications, we are given a matrix $Y_N$ of fixed size and not a sequence whose e.s.d.\ converges to a certain distribution $\mu_y$. The deterministic equivalent method, for Pastur's equation \eqref{Eq:PasturEq}, consists in replacing $\mu_y$ by the true e.s.d.\ $\mu_{Y_N}$. In this setting, there is a unique analytic map $g^{\square}_{X_N+Y_N}:\mbb C^+ \to \mbb C^-$, solution of the fixed-point equation
	\eqa\label{Eq:EqDetPastur}
		g(\lambda) = g_{Y_N}\big( \lambda - g(\lambda) \big), \ \forall \lambda \in \mbb C^+,
	\qea
for a map $g : \mbb C^+ \to \mbb C^-$. Note that there is a convenient abuse of notation in the sense that $g^{\square}_{X_N+Y_N}$ depends only on $Y_N$. We say that $g^{\square}_{X_N+Y_N}$ is a deterministic equivalent of $g_{X_N+Y_N}$. The interest is that $g^{\square}_{X_N+Y_N}$ is a good approximation of  $g_{x+y}$ of the Stieltjes transform that we can approximate numerically thanks to the fixed-point equation \eqref{Eq:EqDetPastur}.

A fundamental result of phase transition for finite rank deformed random matrices was discovered by Ben Arous, Baik and P\'ech\'e \cite{baik2005} in the slightly different model of Gaussian matrices without symmetry. Called in short BBP-transition, the analogue result for GUE matrices \cite{Pec06} states that, in the large $N$ limit, a spike $\theta$ of $Z_N$ will create an outlier $\sigma$ in $X_N+Z_N$ only if $\theta>1$ in which case $\sigma = \theta+ \theta^{-1}$. For related results, we also refer to \cite{FK81, FP07, CDF09,CDFF11}.

\begin{Rk}  When all entries of the variance profile equal one, we emphasis that the existence and the position of an outlier is independent of the eigenvector of $Z_N$ associated to the spike $\theta$. 
\end{Rk}

\subsection{The free probability approach} \label{sec:freeapproach}

More generally, the issues discussed above have also been considered when $X_N$ is a unitary invariant random matrix. In this context, Voiculescu's notion of asymptotic freeness \cite{VDN92} implies a generalization of Pastur's equation. Assume that $\mu_{X_N}$ and $\mu_{Y_N}$ have limiting e.s.d.\ $\mu_{x}$ and $\mu_{y}$ respectively. Recall that, denoting $g_x$ the Stieltjes transform of $\mu_x$,  the $\mcal R$-transform $\mcal R_x$ of $x$ is the analytic map satisfying
	\eqa\label{Def:RTransf}
		g_x(\lambda) = \Big( \lambda - \mcal R_x\big( g_x(\lambda) \big) \Big)^{-1}.
	\qea
Then $\mu_{X_N+Y_N}$ converges to a measure $\mu_{x+y}=\mu_{x}\boxplus \mu_{y}$, called the free convolution of $\mu_x$ and $\mu_y$. This limit is characterized by the so-called subordination property:  $\forall \lambda \in \mbb C^+$,
	\eqa\label{Eq:VoicEq}
		g_{x+y}(\lambda) = g_y\Big( \lambda - \mcal R_x\big( g_{x+y}(\lambda) \big) \Big).
	\qea
The map $\mcal R_x$ is linear if and only if $\mu_x$ is a centered semicircular distribution. 
The method of deterministic equivalent can be extended to this case using Voiculescu's equation \eqref{Eq:VoicEq} instead of Pastur's one. The difficulty in general is to compute the $\mcal R$-transform $\mcal R_x$, or to replace it with a good approximation.

For finite rank deformation, an important discovery was made in the early decade by Capitaine \cite{Cap13}. We recall the heuristic presented in \cite{belinschi2017} and refer to this paper for the mathematical arguments, without defining the notions of free probability. In the context of Voiculescu's problem, the limit $x$ of $X_N$ and $y$ of $Y_N$ are modeled in the free von Neumann algebra generated by two self-adjoint variables $x$ and $y$ with distribution $\mu_x$ and $\mu_y$ respectively. Let $E_y$ be the projection on the von Neumann algebra generated by $y$. Then Biane proved \cite{Bia98} that there exists an analytic map $\omega_{x,y}$ defined outside the spectrum of $x+y$, called the subordination function, such that
	\eqa\label{Eq:VonNeuSub}
		\mrm E_y\big[ (\lambda - x - y )^{-1}\big] = \big( \omega_{x,y}(\lambda) - y \big)^{-1}, \ \forall \lambda \in \mbb C^+.
	\qea
The above equality means that the projection of the resolvent of $x+y$ equals the resolvent of $y$ evaluated at the subordination function. Taking the trace in the identity yields the relation for Stieltjes transforms
	\eq
		g_{x+y}(\lambda) = g_y\big( \omega_{x,y}(\lambda)\big), \ \forall \lambda \in \mbb C^+.
	\qe

The reader not familiar with free probability language can still translate this result into a heuristic for unitary invariant matrices by replacing the condition expectation $E_y$ by the classical expectation $\esp$ and the variables by the matrices. One method to introduce an approximate subordination function consists in setting 
	$$\omega_{X_N,Y_N}(\lambda) := \big(\esp \big[ (\lambda - X_N - Y_N )^{-1}\big] \big)^{-1} + Y_N,$$
so that we have $\esp \big[ (\lambda - X_N - Y_N )^{-1}\big] = \big( \omega_{X_N,Y_N}(\lambda)  - Y_N \big)^{-1}$, similar to \eqref{Eq:VonNeuSub}.

The matrix $\omega_{X_N,Y_N}$ is not a scalar, but this property is true for the compression involved in the outlier detection problem, in the following sense. If $X_N$ is unitarily invariant, then $\esp \big[ (\lambda - X_N - Y_N )^{-1}\big]$ belongs to the unital algebra generated by $Y_N$. Let now assume that $Z_N =U_{N,k}\Theta U_{N,k}^*$ whose eigenvectors are orthogonal to those of $Y_N$. This property ensures that the spikes are eigenvalues of $Y_N+Z_N$. It implies that for any matrix $A$ in the algebra generated by $Y_N$, its compression $A\mapsto U_{N,k}^* A U_{N,k}$ is a scalar matrix. Hence in particular 
	$$U_{N,k}^* \omega_{X_N,Y_N}(\lambda) U_{N,k}= \big( U_{N,k}^*\esp \big[ (\lambda - X_N - Y_N )^{-1}\big]U_{N,k}\big)^{-1} + 0,$$
is a scalar matrix, and it is actually a good approximation for $\omega_{x,y}(\lambda)\mbb I_k$. 

Concentration properties of unitarily invariant matrices implies that the function $\beta_k:\lambda\mapsto \mbb I_k - U_{N,k}^*(\lambda   - X_N-Y_N)^{-1}  U_{N,k} \Theta_k$ defined in \eqref{Eq:Beta} is close to its expectation, and so $\mrm{det}(\beta_k(\lambda))$ is well approximated by
$$\prod_{i=1}^k \Big(1 - \frac{\theta_i}{\omega_{X_N,Y_N}(\lambda)}\Big),$$
where $\theta_1\etc \theta_k$ denote the eigenvalues of $Z_N$. Hence, we retrieve the fundamental Capitaine's relation  \cite{Cap13} between spikes and outliers, namely in the large $N$ limit, the locations of the outliers belong to the pre-image by $\omega_{x,y}$ of the spikes.

\begin{Rk}
When $X_N$ is unitary invariant random matrix and assuming that the eigenvectors of $Z_N$ are orthogonal to those of $Y_N$, the locations of outliers in $X_N+Y_N+Z_N$ depend individually on the eigenvalues of $Z_N$. They do not depend on its eigenvectors, and the outliers generated by a spike $\theta$ do not depend on the other spikes. 
\end{Rk}

\subsection{Random matrices with a variance profile in free probability}\label{Sec:LitteratureRMTvp}

The asymptotic of GUE matrix with a variance profile is characterized by Shlyakhtenko in \cite{SHL96} in the multi-matrix setting of operator-valued free probability over the diagonal. Assuming the variance profiles are of the form $\gamma_N^2(i,j) = \gamma^2\big( \frac i N, \frac j N\big)$, for some bounded real-valued function $\gamma:[0,1]^2\to \mbb R^+$, Shlyakhtenko proved that independent GUE matrices with variance profiles are \emph{asymptotically free with amalgamation over the diagonal}. 

In particular, this implies that the e.s.d.\ of $X_N$ converges almost surely, and that the limit is characterized by an integral operator with kernel $\gamma^2$. The approach is in the lineage of previous works on band matrices, see \cite{KLH91,BMP91,Pas92}, for which it was observed that to derive the limiting e.s.d.\ of a random matrix one can derive a system of linear equations for the diagonal of the resolvent $(\lambda \mbb I_N - X_N)^{-1}$ of the matrix. An interest in Shlyakhtenko's approach is the use of the notion of operator-valued free probability, which (in particular) nicely generalizes  Pastur's equation.

More formally, Shlyakhtenko considers in \cite{SHL96} the operator-valued Stieltjes transform $G_{X_N}$ of $X_N$, which the a map between sets of diagonal matrices defined by \eqref{eq:opStieltjes}. For any bounded function $\Lambda:[0,1]\to \mbb C^+$, let $\Lambda_N \in \mrm D_N(\mbb C)^+$ defined by $\Lambda_N(i,i) = \Lambda\big(\frac i N\big)$. Then, for any $\Lambda:[0,1]\to \mbb C^+$, the diagonal matrix $G_{X_N} (\Lambda_N)$, seen as a piece-wise constant function on $[0,1]$, converges to a function $G_x(\Lambda)$ in $L^\infty\big([0,1], \mbb C^-\big)$. The functional map $G_x$ is characterized the identity
	\eqa
		G_x(\Lambda) = \Big( \Lambda - \mcal R_x\big( G_x(\Lambda) \big) \Big)^{-1},
	\qea
where for any $\Lambda$ in $L^\infty\big([0,1], \mbb C\big)$, 
	\eqa\label{Eq:ShlR}
		\mcal R_x(\Lambda) := \int_0^1 \gamma(\, \cdot \, ,y) \Lambda(y)\mrm d y.
	\qea
Since $\mcal R_x$ is linear, we say that the abstract limit $x$ of $X_N$ (which lives in a von Neumann algebra)  is a semicircular variable with amalgamation over the diagonal. Note that the mapping $\mcal R_N$ defined by \eqref{eq:defRNbis}  is a discretization of the functional map $R_x$.

Recently, the \emph{traffic method} yields to the observation that asymptotic freeness over the diagonal was a generic rule for large permutation invariant random matrices with a variance profile \cite{ACDGM17}. We mention briefly a consequence of this result, referring to \cite{Mal20} for definitions.  Let $X_N = X_N' \circ (\gamma_{ij})_{i,j}$ be the entry-wise product of a permutation invariant random matrix $X_N'$ that converges in traffic distribution and of a matrix $(\gamma_{ij})_{i,j=1\etc N}$ that converges in graphons topology \cite[see second item in Corollary 2.19]{Mal20}. Let $Y_N$ be a matrix bounded in operator norm that converges in traffic distribution.  Then, under the additional assumption that $Y_N$ is permutation invariant, $X_N$ and $Y_N$ are asymptotically free over the diagonal. They converge to elements $x$ and $y$ of a von Neumann algebra endowed with a conditional expectation $\Delta$, and their operator-valued Stieltjes transforms satisfy
\begin{equation}
G_{x+y}(\Lambda) = G_{y}\Big( \Lambda - \mcal R_x\big( G_{x+y}(\Lambda) \big) \Big), \label{eq:opsubord}
\end{equation}
where $\mcal R_x$ is such that the above equation is valid for $y=0$. For more details about the operator-valued subordination property, we refer to  \cite{Voic00}.

A motivation of our work is to specify this statement, in a comprehensive way, when $X_N$ is a GUE matrix with a variance profile and to give an estimate for its operator-valued Stieltjes transform. Then, we explicit the associated deterministic equivalent and we show how to adapt Capitaine's approach for outlier detection. 

\subsection{The Dyson equations and local laws in RMT} \label{sec:Dyson}

In the RMT literature, Hermitian random matrices with centered entries but non-equal distribution are referred to as generalized Wigner matrices for which many asymptotic properties are now well understood in a  precise sense. For example, under the assumption that the variance profile is bi-stochastic (that is its rows and columns elements sump up to one), bulk universality at optimal spectral resolution for local spectral statistics have been established in \cite{erdos2012} and they are shown to converge to those of the GUE.  The case of a Wigner matrix with a variance profile that is not necessarily bi-stochastic has been studied in \cite{Ajanki2017}, and non-hermitian random matrices with a variance profile have been considered in \cite{cook2018,hachem2007,HACHEM2006649} using the notion of deterministic equivalent.

Under mild assumptions on the variance profile $\Gamma_N$, the e.s.d.\ of a generalized Wigner matrice converges to a limiting spectral measure for which there is generally no explicit formula. Currently, a classical method to approximate an asymptotic spectral measure is to solve a nonlinear system of  deterministic equations that are referred to as the Dyson equation \cite{Ajanki2017,Ajanki2019,alt2017,alt2018}. For each fixed dimension $N$,  the solution of this equation yields a deterministic equivalent of the resolvent of $H_N = X_N +Y_N$. These equations are also equivalent to the operator-valued equation of the subordination functions where the parameter is scalar, instead of being functional. For example, the  {\it vector Dyson equation} studied in details in \cite{Ajanki2017} corresponds to Equation \eqref{FixPtEq1} with $\Lambda = \lambda \mbb I_N$  and $Y_N$ a diagonal matrix. The  {\it matrix Dyson equation}, introduced in \cite{Ajanki2019} to study Hermitian random matrices with correlated entries and nonzero expectation, is the following nonlinear {\it matrix equation} formulated for an unknown matrix-valued function $A_N : \mbb C^+ \to \mbb C^{N \times N}$ 
\begin{equation} \label{eq:MatrixDyson}
\mbb I_N - \left(\lambda \mbb I_N - \mcal S_N(A_N(\lambda)) - Y_N  \right) A_N(\lambda)= 0
\end{equation}
for $\lambda \in \mbb C^+$, where $\mcal S_N$ is the mapping (see Equation (1.3) in \cite{Ajanki2019})
\begin{equation}
\mcal S_N(A) =\esp\big[  X_N \ A \ X_N\big] = \mcal R_N\big( \Delta(A)\big), \mbox{ for } A \in \mbb C^{N \times N}.  \label{eq:SN}
\end{equation}
Thus, provided that $A_N(\lambda)$ is invertible, Equation \eqref{eq:MatrixDyson} may be written as 
$$
A_N(\lambda) = \left( \lambda \mbb I_N - \mcal R_N(\Delta(A_N(\lambda))) - Y_N  \right)^{-1}.
$$
Hence, applying the operator $\Delta$ of both sides of the above equality yields the fixed point Equation \eqref{FixPtEq1} with scalar parameter  $\Lambda = \lambda \mbb I_N$.

The existence and stability of the solutions of the vector and matrix Dyson equation are studied in details in \cite{Ajanki2017} and \cite{Ajanki2019} respectively. These deterministic vector or matrix valued functions (parametrized by  $\lambda \in \mbb C^+$)  are used to prove {\it local laws} for the resolvent of $H_N$.   In RMT, the derivation of local laws  refers to results controlling the difference between the Stieltjes transform $g_{H_N} (\lambda)$ (or the resolvent $(\lambda \mbb I_N - H_N)^{-1}$) and a deterministic function when $\Im m \, \lambda$ is allowed to decay to zero at a rate depending on $N$. Deriving a local law for the  the Stieltjes transform is a generally a delicate problem that is more involved than proving a {\it global law} which refers to the convergence (e.g.\ in probability) of  $g_{H_N} (\lambda)$ for a {\it fixed value} of $\lambda  \in \mbb C^+$. The notion of local semicircular law for Wigner matrices, which constituted the central open question known as Wigner-Dyson-Mehta conjecture, was solved in 2011 independently by \cite{ESY11} and \cite{TV11}.  For detailed lecture notes on this notion, we  refer to \cite{BK16}.

In the case where $Y_N = 0$ and $X_N$ is a standard  Wigner matrix satisfying mild assumptions (with a constant variance profile), then the {\it optimal local law} for the Stieltjes transform and the resolvent of $H_N = X_N$ reads as the following concentration inequalities \cite[Theorem 2.6]{BK16}: for a fixed $\tau > 0$, define the complex domain
$$
\mcal C_{N}(\tau) = \left\{\lambda \in \mbb C *
 \; : \;  |\lambda| \leq \tau^{-1}, N^{-1 + \tau} \leq \Im m \, \lambda  \right\}.
$$
Then, denoting by $g_{sc}$ the  Stieltjes transform of the semicircular distribution $\mu_{sc}$, for any $\varepsilon > 0$ and $D > 0$, one has that 
$$
\mbb P\left( |g_{H_N} (\lambda) - g_{sc}(\lambda)| \geq  N^{\varepsilon} \psi^{(1)}_{N}(\lambda) \right) \leq N^{-D} \mbox{ with } \psi^{(1)}_{N}(\lambda) = \frac{1}{N \Im m \, \lambda},
$$
and, uniformly for $i,j =1 \ldots, N$,
$$
\mbb P\left( | (\lambda \mbb I_N - H_N)^{-1}(i,j) - g_{sc}(\lambda) \delta_{ij}| \geq  N^{\varepsilon} \psi^{(2)}_{N}(\lambda) \right) \leq N^{-D},$$
$$  \mbox{ with } \psi^{(2)}_{N}(\lambda) = \sqrt{\frac{\Im m (g_{sc}(\lambda))}{N \Im m \, \lambda}} + \frac{1}{N \Im m \, \lambda},
$$
for all $\lambda \in \mcal C_{N}(\tau)$ and all sufficiently large $N \geq N_{0}(\varepsilon,D)$. The above deviation inequalities are called {\it optimal local laws} as the rate of convergence  of the error terms $\psi^{(1)}_{N}(\lambda)$ and $\psi^{(2)}_{N}(\lambda)$ as $\lambda = \lambda_N$ tends to zero slower than $N^{-1}$ are known to be optimal for $t = \Re e(\lambda) \in  [-2,2]$ which is the support of $\mu_{sc}$ (see e.g.\  \cite[Section 2]{BK16}). Note that when $t$ is outside $[-2,2]$, faster error rate (as $\lambda \to 0$) may be obtained (see e.g.\ \cite[Theorem 10.3]{BK16}).

 The derivation of optimal local laws for Hermitian random matrices with an arbitrary variance profile using the deterministic solution of a Dyson equation has been largely investigated by László Erdös and his collaborators over the last decade, and for a recent overview we refer to Erd\"os's lecture notes \cite{Erdos19}. In \cite{Ajanki2017,Ajanki2019}, the authors make several assumptions to derive local laws, in particular they suppose that the entries of the variance profile $\Gamma_N$ are bounded away from below. In \cite{EKS19} this assumption is relaxed: a weak local law is obtained for the model $H_N = X_N + Y_N$ without lower bound assumptions on the variance profile. The generalized Wigner matrix $X_N$ can possibly have correlated entries that are not necessarily Gaussian random variables. More precisely, from \cite[Theorem 2.1]{EKS19} (see also the preliminary result \cite[Lemma B.1]{alt2019}) the following deviation inequalities hold: for a fixed $\delta > 0$ define the complex domain
$$
\mcal C_{N}^{out}(\delta) = \left\{\lambda \in \mbb C  \; : \;  |\lambda| \leq N^{C_0}, \mbox{dist}(\lambda,\mbox{supp } \mu^\square) \geq N^{-\delta} \right\}
$$
for some arbitrary $C_0 \geq 100$, where $\mu^\square$ is the probability measure associated to the Stieltjes transform $  \frac 1 N \Tr \big[ A_N^\square (\lambda) \big]$ where $A_N^\square$  is the solution of the matrix Dyson equation \eqref{eq:MatrixDyson} (see \cite[Section 2]{EKS19}). Then,  from \cite[Theorem 2.1]{EKS19} it follows that  for any $\varepsilon > 0$, there exists $\delta > 0$ such that for any $D > 0$ 
\begin{equation}  
\mbb P\left(|\langle u , ((\lambda \mbb I_N - H_N)^{-1}-A_N^\square (\lambda)) v\rangle|  \geq  \frac{N^{\varepsilon}}{(1+|\lambda|)^2 \sqrt{N}} \mbox{ in } \mcal C_{N}^{out}(\delta) \right) \leq C N^{-D}, \label{ineqErdos1a}
\end{equation}
for all unit vectors $u,v \in \mbb C^N$, and
\begin{equation}
\mbb P\left( |g_{H_N} (\lambda) - \frac 1 N \Tr \big[ A_N^\square (\lambda) \big] | \geq  \frac{N^{\varepsilon}}{(1+|\lambda|)^2 N}  \mbox{ in } \mcal C_{N}^{out}(\delta)  \right) \leq C N^{-D}, \label{ineqErdos1b}
\end{equation}
where $C$ is a constant depending on $\varepsilon, D$ and other quantities not depending on $N$. Inequalities \eqref{ineqErdos1a} and \eqref{ineqErdos1b} hold under general assumptions on the distributions of the entries of $X_N$ and the additive deformation $Y_N$ (we refer to \cite[Section 2.2]{EKS19} for further details) but {\it without assuming} that the entries of $\Gamma_N$ are lower  bounded. In \cite{EKS19} this result is referred to as a  local law outside the support of $\mu^\square$ as the error terms in  Inequalities \eqref{ineqErdos1a} and \eqref{ineqErdos1b} do not decay at an optimal rate when $t = \Re e(\lambda)$ belong to $\mbox{supp } \mu^\square$. To obtain an optimal local law when $t \in \mbox{supp } \mu^\square$ (that is in the bulk of the spectrum) a flatness condition on the variance profile  is added in \cite[Section 2.2]{EKS19} which implies that the entries of $\Gamma_N$ are bounded from below by positive constant. Under this supplementary condition, the following deviation inequalities hold:  for fixed $\delta, \gamma > 0$ define the complex domain
$$
\mcal C_{N}^{out}(\delta, \gamma) = \left\{\lambda = t+\mbf i \eta  \; : \;  |\lambda| \leq N^{C_0}, |\eta| \geq N^{-1+\gamma}, \rho(x) + \mbox{dist}(x,\mbox{supp } \mu^\square) \geq N^{-\delta} \right\},
$$
where $\rho$ is the density of $\mu$.  Then,  from \cite[Theorem 2.2]{EKS19} it follows that  for any $\gamma, \varepsilon > 0$, there exists $\delta > 0$ such that for any $D > 0$ 
\begin{equation}  
\mbb P\left(|\langle u , ((\lambda \mbb I_N - H_N)^{-1}-A_N^\square (\lambda)) v\rangle|  \geq  \frac{N^{\varepsilon}}{\sqrt{N \eta}} \mbox{ in } \mcal C_{N}^{out}(\delta) \right) \leq C N^{-D}, \label{ineqErdos2a}
\end{equation}
for all unit vectors $u,v \in \mbb C^N$, and
\begin{equation}
\mbb P\left( |g_{H_N} (\lambda) - \frac 1 N \Tr \big[ A_N^\square (\lambda) \big] | \geq  \frac{N^{\varepsilon}}{ N \eta}  \mbox{ in } \mcal C_{N}^{out}(\delta)  \right) \leq C N^{-D}. \label{ineqErdos2b}
\end{equation}
The error terms in Inequalities  \eqref{ineqErdos2a} and \eqref{ineqErdos2b} decay at the optimal rate when $t \in \mbox{supp } \mu^\square$ and $\eta = \Im m(\lambda)$ becomes small, and this result is thus an optimal local law within the bulk of the spectrum. 

Comparing Inequalities \eqref{ineqErdos1a}-\eqref{ineqErdos2b} with the results stated in Theorem  \ref{MainTh} (when $\Lambda = \lambda \mbb I_N$ is scalar) and Corollary \ref{MainCor}, it follows that we have only obtained  a weak local law on the convergence of the operator-valued Stieltjes transform $G_{H_N} $ and the scalar Stieltjes transform $g_{H_N}$ to their deterministic equivalent, since the decay of the error terms in Inequalities \eqref{eq:MainThStieltjes} and \eqref{eq:MainThStieltjes2} as $\Im m(\lambda) \to 0$ is not optimal. Hence, when $\Lambda = \lambda \mbb I_N$ is scalar, the results of Theorem  \ref{MainTh} are clearly sub-optimal as better concentration bounds already exist in the literature that also go beyond the Gaussian assumption for the entries of $X_N$. Nevertheless, to the best of our knowledge, obtaining a weak local law for  the operator-valued Stieltjes transform $G_{H_N} $ is a novel result. Hence, it would be interesting to have a generalization of optimal local laws as in \cite{Ajanki2017,Ajanki2019,EKS19} but in an operator-valued sense, that is for $(\Lambda  - H_N)^{-1}$, when $\Lambda$ is not a scalar matrix as we consider in the current paper.

Finally, it should be mentioned that the results from Corollary \ref{Cor:Beta1} and Corollary \ref{Cor:Beta2} on the convergence of the deterministic equivalents $\beta_k^\square$  and $\tilde{\beta}_k^\square$ to the matrix-valued function $\beta_k$ are novel. Indeed, we are not aware of other works dealing with the issue of localizing outliers in the singular values distribution of the matrix $H_N' = X_N+Y_N+Z_N$ when $X_N$ has an arbitrary variance profile. The question of the optimality of the deviation inequalities \eqref{Goal} and \eqref{Goal2} is obviously left open.

\section{Applications and numerical illustrations} \label{sec:num}

In this section, we report results on numerical experiments on the localization of outliers in the rectangular information plus noise model \eqref{Model3} for various variance profiles and additive perturbations that may posses spikes generating outliers. 

\subsection{Information plus noise model with a variance profile} \label{sec:infonoise}

We introduce the model
\eqa\label{Model4}
H'_{N,M}  =  X_{N,M} + Y_{N,M} + Z_{N,M}.
\qea
where $Z_{N,M} = U_{N,k} \Theta_{k} V_{M,k}^{\ast}$ is a low-rank matrix $N \times M$ with $k$ spikes that are equal to the singular values of $Z_{N,M}$, where $\Theta_{k}$ is $k \times k$ diagonal matrix with positive diagonal entries and $U_{N,k}$ (resp.\ $V_{M,k}$) is the matrix whose columns are the left (resp.\ right) singular vectors of  $Z_{N,M}$.

The question of locating potential outliers in model \eqref{Model4} can be answered using the approach developed in this paper for the Hermitian setting.  To this end, we use the principle of  \emph{Hermitian dilation} \cite{paulsen2002completely} which corresponds to embed any rectangular matrix $A_{N,M}$ (with complex entries) of size $N \times M$ within a larger Hermitian block matrix by defining
\begin{equation}
\DD(A_{N,M}) = \left[\begin{array}{cc}0 & A_{N,M} \\A_{N,M}^* & 0\end{array}\right]. \label{Eq:Dilation}
\end{equation}
Note that if one denotes by $\sigma_1 \geq \ldots \geq \sigma_r > 0$ the singular values of $A_{N,M}$ assumed to be of rank $r$, then the spectrum of the Hermitian matrix $\DD(A_{N,M})$ is
$$
\left\{- \sigma_1 \leq \ldots \leq - \sigma_r \leq 0 \leq \sigma_r \leq \ldots  \leq \sigma_1 \right\}
$$
where the eigenvalue $0$ is of multiplicity $M+N-2r$. By applying Hermitian dilation to Equation \eqref{Model3}, we obtain that
$$
\DD(H'_{N,M}) =\DD(X_{N,M}) + \DD(Y_{N,M})+ \DD(Z_{N,M}), 
$$
which is a finite rank deformation of  the GUE model $\DD(H_{N,M}) =\DD(X_{N,M}) + \DD(Y_{N,M})$.  For $i\leq N$ and $j\geq N+1$ the entry $(i,j)$ of $\DD(X_{N,M})$ is  a centered complex Gaussian variable with variance $\frac{\gamma_{N,M}(i,j)}{M}$ satisfying $ \frac{x_{ij}}{\sqrt M}  = \sqrt{\frac{N+M}M}\big( \frac{x_{ij}}{\sqrt {N+M}} \big)$. Therefore, $\DD(X_{N,M})$ is a GUE matrix of size $N+M$ with variance profile $\frac{N+M}M \DD(\Gamma_{N,M})$, where the zero entries of $\DD(X_{N,M})$ are considered as centered Gaussian variables with variance equal to zero. The additive perturbation $\DD(Z_{N,M})$ is a matrix of rank $2k$ that can be written as
$$
\DD(Z_{N,M}) = W_{N+M,2k}  \left[\begin{array}{cc}\Theta_k & 0 \\ 0& - \Theta_k \end{array}\right] W_{N+M,2k}^*
$$
where
$$
W_{N+M,2k} =  \frac{1}{\sqrt{2}} \left[\begin{array}{cc} - U_{N,k}  & U_{N,k}  \\ -V_{M,k}& - V_{M,k} \end{array}\right]
$$
is a $(N+M) \times 2k$ matrix whose columns are orthonormal vectors. After introducing the Hermitian dilation of model \eqref{Model3}, one may thus consider the  function $G_{\DD(H)}^\square : \mrm D_{N+M}(\mbb C)^+ \to \mrm D_{N+M}(\mbb C)^-$,  analytic in each variable, that is the unique solution of the  fixed point equation
	\eqa\label{FixPtEq2}
		  G_{\DD(H)}^\square({ \Lambda}) =\Delta\bigg [ \Big( \Lambda   -\mcal R_{N+M}\big(  G_{\DD(H)}^\square({ \Lambda})\big)  - \DD(Y_{N,M}) \Big)^{ -1}\bigg].
	\qea
which holds for any $\Lambda \in  \mrm D_{N+M}(\mbb C)^+$. Following Equation \eqref{eq:RNdeg},  $\mcal R_{N+M}$  is the map  defined on $\mrm D_{N+M}(\mbb C)^+$ by
\begin{equation}
		\mcal R_{N+M}(\Lambda) = \mrm{deg}\Big(  \frac{1}{N+M} \times \frac{N+M}M \DD(\Gamma)  \Lambda \Big) = \mrm{deg}\Big( \frac{\DD(\Gamma)  }{M}  \Lambda  \Big). \label{eq:defRNM}
\end{equation}
Hence, after Hermitian dilation, one may follow the approach described in Section \ref{Sec:MainStat} to approximate of the global behavior of the empirical distribution of the singular values of $H_{N,M}$ and to localize potential outliers generated by the spikes of $Z_{N,M}$.

\begin{Rk}
In the RMT literature, the Hermitian dilation \eqref{Eq:Dilation} of a rectangular random matrices is classically referred to as Girko's Hermitization trick\cite{MR1887675,MR1887676}. As our study of a GUE matrix with variance profile  allows the setting where large blocks of  $\DD(\Gamma_{N,M})$ are equal to zero, treating the setting of the information plus noise model using  Girko's Hermitization  is an immediate application of our results in the Hermitian case.   As explained e.g.\ in \cite{alt2017,alt2018} this Hermitization trick may also be used beyond the Gaussian case. However, a direct application of the results in \cite{Ajanki2017,Ajanki2019} on the vector and matrix Dyson equation for the study of generalized Wigner matrices is not possible as a key assumption in these papers is that  the entries of the variance profile must be bounded from below. Hence, beyond the Gaussian case, the use of Girko's Hermitization requires a specific treatment in \cite{alt2017,alt2018}  and the introduction of a system of quadratic vector equations extensively studied in \cite{Oskari17} that relates the resolvent of $X_{N,M} X_{N,M}^{\ast}$ to the resolvent of $\DD(X_{N,M})$ via the equality
$
( \lambda^2 \mbb I_{N}- X_N X_N^{\ast} )^{-1} = A_{1,1}(\lambda) / \lambda,
$
where $A_{1,1}(\lambda)$ denotes the upper left $N \times N$ block of $(\lambda \mbb I_{N+M}  - \DD(X_{N,M}) )^{-1}$.
\end{Rk}

\subsection{Additive deformation of rank one}

Let us first consider the rectangular model \eqref{Model4} under the simplified setting $Y_{N,M} = 0$ and the low rank denoising model with an additive deformation of rank $k=1$, that is
\eqa\label{ModelLowRankOne}
	H'_{N,M}  = X_{N,M} + Z_{N,M}, \quad \mbox{where} \quad Z_{N,M} = \theta u_{N}  v^{\ast}_{M},
\qea
where  $u_{N} \in \R^{N}, v_{M} \in \R^{M}$ are unit vectors, $\theta > 0$ is a spike, and $X_{N,M}$ is a rectangular Gaussian matrix with a variance profile $\Gamma_{N,M} = \big(\gamma_{N,M}^2(i,j)\big)_{i,j}$. We assume that the variance profile satisfies
\begin{equation}
\frac{1}{N} \sum_{i=1}^{N} \sum_{j=1}^{M} \frac{\gamma_{N,M}^2(i,j)}{M} = 1, \label{eq:normcond}
\end{equation}
which ensures the normalization condition $\esp\big[ \frac 1{N} \Tr  X_{N,M} X_{N,M}^* \big] = 1$.
In all the numerical experiments we took $N=360$ and $M=400$. Following the discussion in Section \ref{Sec:MainStat} and the principle of Hermitian dilation described in Section \ref{sec:infonoise}, a potential outlier in model \eqref{ModelLowRankOne} may be found by searching for a positive real $\lambda$ such that
\begin{equation}
\det\big(\beta_{2}^\square(\lambda ) \big) = 0,  \label{eq:detbeta2}
\end{equation}
where
\begin{equation}
\beta_{2}^\square(\lambda) =   \mbb I_{2} - \frac{1}{2} \left[\begin{array}{cc} - u_{N}^*  & -v_{M}^*   \\  u_{N}^* & - v_{M}^* \end{array}\right] G_{\DD(H)}^\square( \lambda \mbb I_{N+M})  \left[\begin{array}{cc} - u_{N}  & u_{N}  \\ -v_{M}& - v_{M} \end{array}\right] \left[\begin{array}{cc}\theta & 0 \\ 0& - \theta \end{array}\right], \label{eq:defbeta2}
\end{equation}
and $G_{\DD(H)}^\square( \lambda \mbb I_{N+M})$ is the solution of the fixed point equation \eqref{FixPtEq2}, with $\DD(Y_{N,M}) = 0$.
\begin{Rk}
We choose to directly search for potential outliers by minimizing $\beta_{2}^\square$ over the set of reals $\lambda \in \R$ rather than over the set of complex values $(\lambda = t + \mbf i \eta)_{t \in \R}$ with a small and fixed value $\eta > 0$ as both approaches lead to the same numerical results.
\end{Rk}
A numerical approximation (for a given value of $\lambda$) is easily obtained by the following iterative procedure
\begin{equation}
 G_{n+1}^\square( \lambda) =\Delta\bigg [ \Big( \lambda \mbb I_{N+M}   -\mcal R_{N+M}\big(   G_{n}^\square( \lambda)\big) \Big)^{ -1}\bigg], \label{eq:vectorfixedpoint}
\end{equation}
that is stopped for $n$ sufficiently large or when the difference between two successive iterations is sufficiently small. Since $\lambda \mbb I_{N+M}   -\mcal R_{N+M}\big(   G_{n}^\square( \lambda)\big)$ is a diagonal matrix, the fixed point iteration corresponds to the numerical evaluation of the vector Dyson equation \cite{Ajanki2017}, and it simplifies to the vector equation
\begin{equation}
\bG_{n+1}^\square( \lambda) = \frac{1}{\lambda \bone_{N+M} -  \frac{\DD(\Gamma_{N,M})  }{M}  \bG_{n}^\square( \lambda) }. \label{eq:itervector}
\end{equation}
In practice, to find a potential solution to equation \eqref{eq:detbeta2}, we use a numerical optimization procedure to obtain a minimizer $\lambda = \bar{\lambda}(\theta)$ of the function $\lambda \mapsto \det\big(\beta_{2}^\square(\lambda ) \big)$ over $\R_+$. To this end, we have used Python's command \texttt{fminsearch} which is based on the Nelder-Mead simplex method. Then, if the value $\det\big(\beta_{2}^\square(\bar{\lambda}(\theta)) \big)$ is \emph{sufficiently close to zero}, we conclude that $\bar{\lambda}(\theta)$ is an outlier.

Finally, a smooth approximation of the  singular values distribution (s.v.d.) of $H_{N,M} = X_{N,M}$ at location $t \in \R$ may be obtained from the inverse Stieltjes transform formula \eqref{eq:Stieltjesinv}. Letting $g_{n}^\square(\lambda) = \frac{1}{N+M} \Tr \,G_{n}^\square( \lambda)$, we define $f_{n}(t) = - \frac{1}{\pi} \Im m \left( g_{n}^\square(t + \mbf i \eta) \right)$ for $\eta > 0$ small enough. Now, recall that  $g_{n}^\square(\lambda)$ is an approximation of the dilation matrix $\DD(H_{N,M}$) whose eigenvalue values are $0$ (with multiplicity $M-N$) and $\left\{- \sigma_1, \ldots , - \sigma_N, \sigma_N , \ldots  , \sigma_1 \right\}$ where $\sigma_N \leq \ldots  \leq \sigma_1$ are the singular values of $H_{N,M}$, and that the inverse Stieltjes transform \eqref{eq:Stieltjesinv} amounts to approximate a measure by a convolution with the Cauchy kernel  $t \mapsto \frac{\eta}{t^2 + \eta^2}$. Therefore, an  approximation of the s.v.d.\ of $H_{N,M}$ is given by the density
\begin{equation}
\tilde{f}_{n}(t) = \frac{2}{1-(M-N)/(M+N)} \left( f_{n}(t) - \frac{M-N}{M+N} \frac{\eta}{t^2 + \eta^2} \right) , \quad t \geq 0. \label{eq:approxdens}
\end{equation} 

\subsubsection{Constant variance profile} We propose to  validate this way of localizing outliers by first considering the standard case where the variance profile $\Gamma_{N,M}$ has constant entries equal to one. This setting corresponds the so-called \emph{Gaussian spike population model} for which the asymptotic behavior (as $\min(N,M) \to + \infty$) of the singular values of $H'_{N,M}$  is  well understood \cite{Benaych-GeorgesN12,MR2322123,MR3054091,loubaton2011} when the rank $k$ of the additive deformation $Z_{N,M}$ is held fixed. In the asymptotic framework where the sequence $M = M_{N} \geq N$ is such that $\lim_{N \to + \infty} \frac{N}{M} = c$ with $0 < c \leq 1$,  it is well known \cite{MR2567175} that the empirical distribution of the singular values of $X_{N,M}$  converges, as $N \to + \infty$, to the quarter circle distribution if $c = 1$,  or to its generalized version if $c < 1$,  called the Marchenko-Pastur distribution, which is supported on the compact interval $[c_{-},c_{+}]$ with
$
c_{\pm} = 1 \pm \sqrt{c}
$
where $c_{+}$ is the so-called bulk (right) edge. Now, if one denotes by $\sigma_1 \geq \ldots \geq \sigma_N > 0$ (recall that $N \leq M$) the singular values of  $H'_{N,M}$, then the following result holds (see e.g.\ Theorem 2.8 in  \cite{Benaych-GeorgesN12} and Proposition 9 in \cite{MR3054091}). Almost surely, one has that
\begin{align*}
\lim_{N \to + \infty} \sigma_{1} =
\left\{
\begin{array}{cc}
\lambda_c\left( \theta \right) & \mbox{ if }  \theta > c^{1/4}, \\
c_{+} & \mbox{ otherwise,}
\end{array}
\right.
\end{align*}
and
$
\lim_{N \to + \infty} \sigma_N = c_{-},
$
where
\begin{equation}
\lambda_c\left( \theta \right) = \sqrt{\frac{(1+\theta^{2})(c+\theta^{2})}{\theta^{2}}} \mbox{ for any } \theta > c^{1/4}. \label{eq:lambdatheta}
\end{equation}
The interpretation of this result, called the BBP transition after \cite{baik2005}, is as follows. If the spike $\theta$ in model \eqref{ModelLowRankOne}  is larger than $c^{1/4}$ then an outlier exists and it is asymptotically located at $\lambda_c\left( \theta \right) > c_{+}$. To the contrary, if $\theta \leq c^{1/4}$ then there exists no outlier as the largest singular value $H'_{N,M}$ converges to the bulk edge $c_{+}$.

\begin{figure}[!t]
\centering
\subfigure[]{\includegraphics[width=0.48\linewidth]{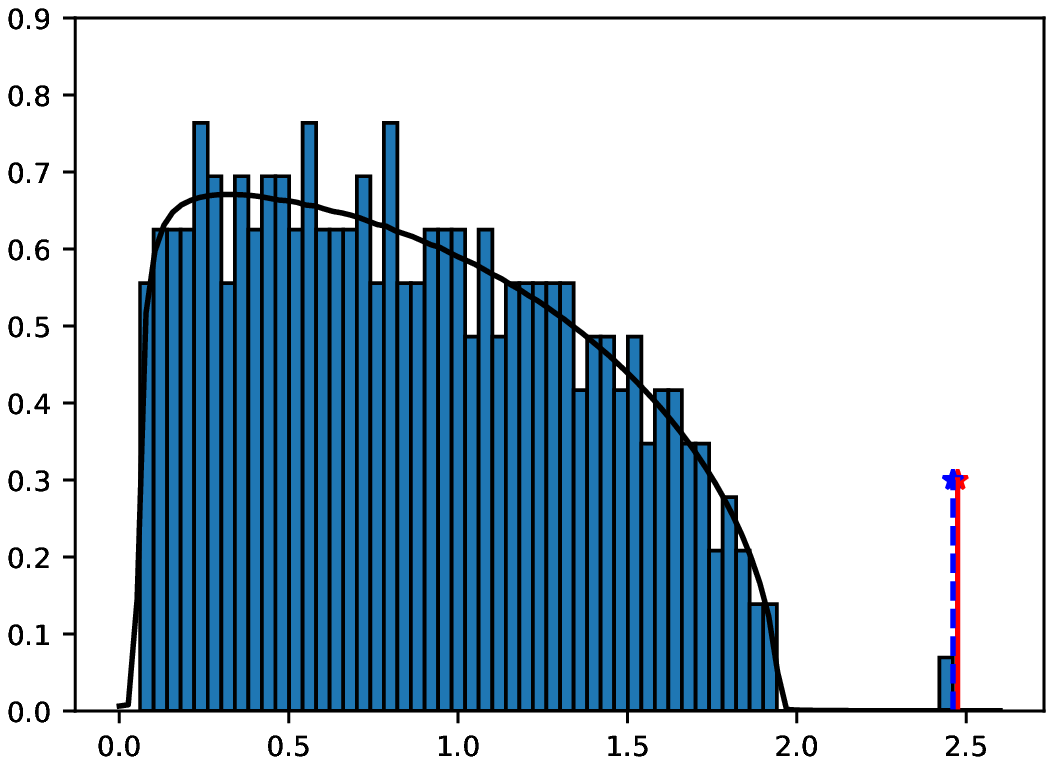}} \\

\subfigure[]{\includegraphics[width=0.4\linewidth]{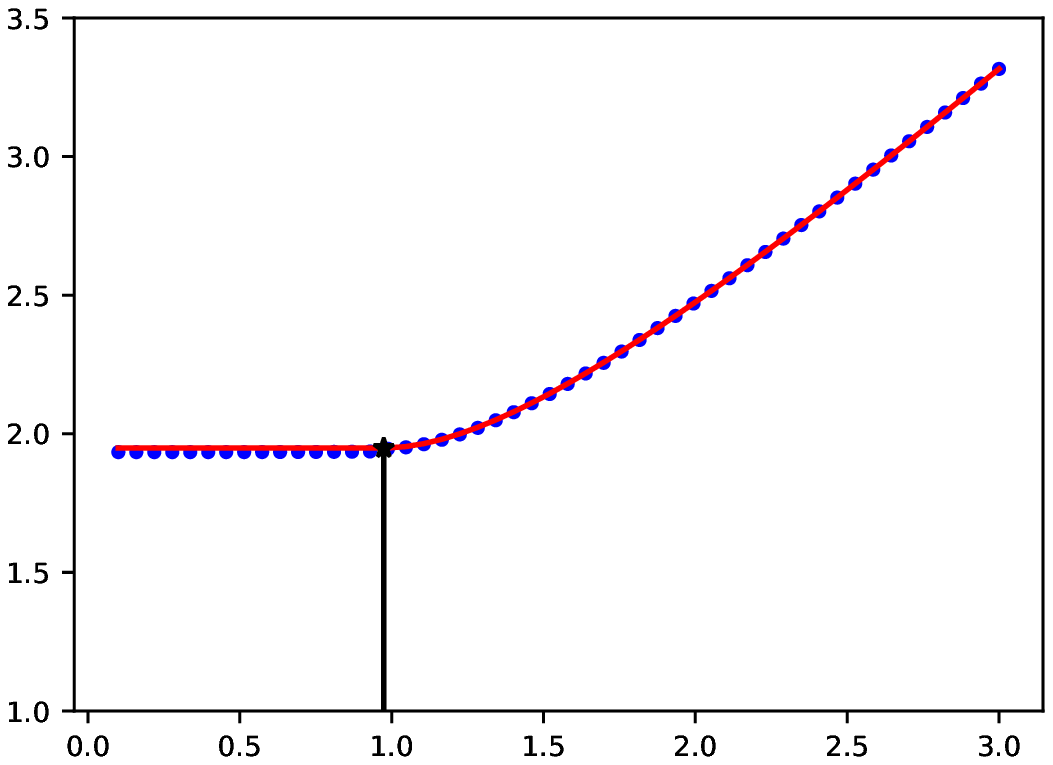}}
\subfigure[]{\includegraphics[width=0.4\linewidth]{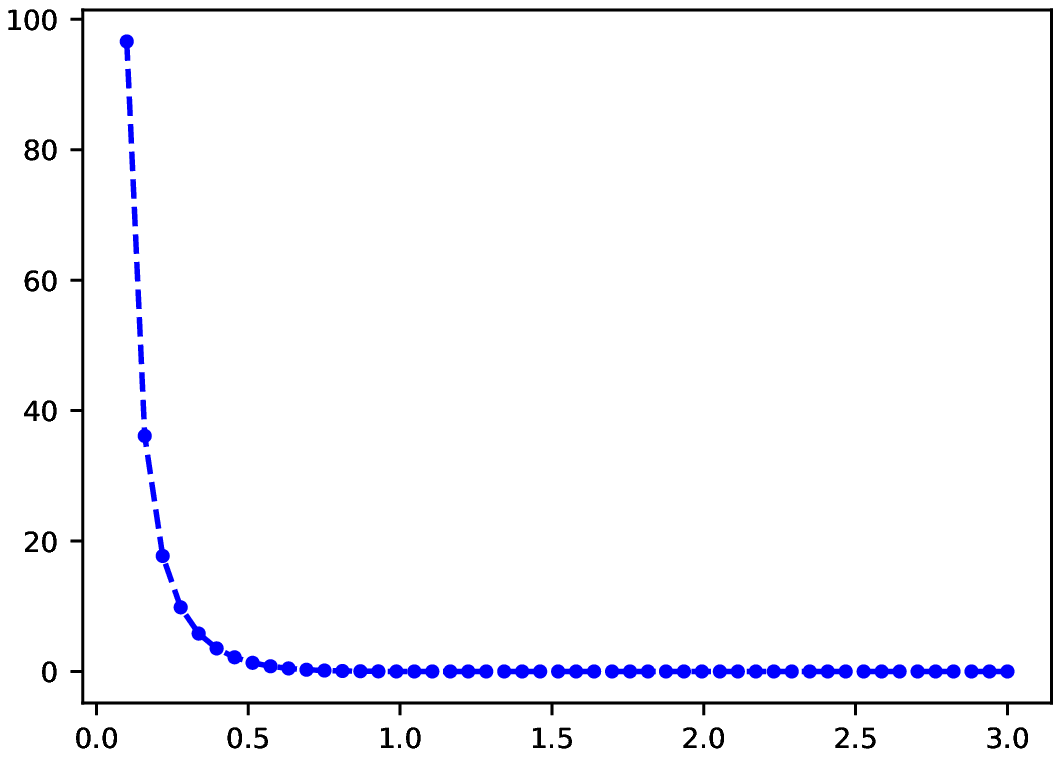}}\\
\caption{Constant variance profile. (a) Histogram of the singular values of one realization of $H'_{N,M}$ for $\theta = 2$. The black curve is the smooth approximation by $\tilde{f}_{n}$ of the singular values distribution of $X_{N,M}$. The red vertical line  denotes the value $\lambda_{N/M}\left(2\right) \approx 2.48$  which is  the approximation of the location of the outlier, while the blue vertical dashed line  denotes the location of the singular value of $H'_{N,M}$ which is the closet to $\lambda_{N/M}\left(2\right)$. (b) The red line is the curve $\theta \mapsto \max\left(1 + \sqrt{\frac{N}{M}} , \lambda_{N/M}\left( \theta \right)\right)$, and the blue dots are the points $(\theta,\bar{\lambda}(\theta))$ where $\bar{\lambda}(\theta)$ is found by numerical minimization of  $\lambda \mapsto \det\big(\beta_{2}^\square(\lambda ) \big)$ for $\theta$ ranging in a grid of 50 regularly spaced values in $[0,3]$. The black vertical line is located at $(\frac{N}{M})^{1/4}$ and its height is $1 + \sqrt{\frac{N}{M}}$. (c) The blue dashed line is the curve $\theta \mapsto \det\big(\beta_{2}^\square(\bar{\lambda}(\theta) ) \big)$.}
  \label{fig:ConstVarProfile}
\end{figure}

For a constant variance profile, an explicit solution of the equation $\det\big(\beta_{2}^\square(\lambda ) \big) = 0$ exists as stated below.

\begin{Lem}\label{Lem:lambdatheta2}  
Assume that $\Gamma_{N,M}(i,j)=1$ for any $i,j$. Then, the Equation \eqref{eq:detbeta2} admits a solution given by
\begin{equation}
\lambda_{N/M}\left( \theta \right) = \sqrt{\frac{(1+\theta^{2})( \frac{N}{M}+\theta^{2})}{\theta^{2}}} \mbox{ provided that } \theta > \left( \frac{N}{M}\right)^{1/4}. \label{eq:lambdatheta2}
\end{equation}
\end{Lem}

\begin{proof} 
Let us  first  determine the  solution $G_{\DD(H)}^\square( \lambda \mbb I_{N+M})$ of the fixed point equation \eqref{FixPtEq2} for a constant variance profile, namely $\Gamma_{N,M}(i,j)=1$ for any $i,j$.  Using the particular structure of the variance profile $\DD(\Gamma_{N,M})$ and the expression \eqref{eq:defRNM} of $R_{N+M}$, one  obtains by simple calculations that  $G_{\DD(H)}^\square( \lambda \mbb I_{N+M}) =  \left[\begin{array}{cc} g_N^\square(\lambda)  \mbb I_{N}  & 0 \\ 0 &  g_M^\square(\lambda) \mbb I_{M} \end{array}\right]$, where $g^\square_N, g^\square_M$ are  complex-valued functions satisfying 
$
g_N^\square(\lambda) = (\lambda - g_M^\square(\lambda))^{-1}
$
and
$$
g_M^\square(\lambda) = \frac{1+\lambda^2- \frac{N}{M} -\sqrt{(1+\lambda^2- \frac{N}{M} )^2-4 \lambda^2}}{2 \lambda}.
$$
Inserting this expression for $G_{\DD(H)}^\square( \lambda \mbb I_{N+M})$ into \eqref{eq:defbeta2},  one obtains that
\begin{eqnarray*}
\det\big(\beta_{2}^\square(\lambda ) \big) 
& = & 1 - \theta^2 g_N^\square(\lambda)g_M^\square(\lambda).
\end{eqnarray*}
Then, by simple calculations, it can be shown that the equation $\det\big(\beta_{2}^\square(\lambda ) \big) = 0$ admits a solution given by \eqref{eq:lambdatheta2} provided that  $\theta > \left( \frac{N}{M}\right)^{1/4}$.
\end{proof}

Note that the condition $\theta > \left( \frac{N}{M}\right)^{1/4}$ guarantees that $\lambda_{N/M}\left( \theta \right)  > 1 + \sqrt{\frac{N}{M}}$. Hence, we retrieve the expression \eqref{eq:lambdatheta} of the asymptotic location of an outlier in the Gaussian spike population model where the asymptotic ratio $c = \lim_{N \to + \infty} \frac{N}{M}$ is replaced with its non-asymptotic approximation $\frac{N}{M}$. As expected, we also remark that the localization $\lambda_{N/M}\left( \theta \right)$ of an outlier does not depend on the singular vectors $u_N$ and $v_N$ of the additive perturbation $Z_{N,M}$.

In Figure \ref{fig:ConstVarProfile}(a), we display the histogram of the singular values of one realization of $H'_{N,M}  = X_{N,M} +  \theta u_{N}  v^{\ast}_{M}$ with $\theta = 2$, where $u_N$ and $v_M$ are chosen to be unit vectors with constant entries. There is clearly an outlier outside the interval $[1-\sqrt{\frac{N}{M}}, 1+  \sqrt{\frac{N}{M}}] \approx [0.05,1.95]$. In Figure \ref{fig:ConstVarProfile}(a), we also plot the curve $x \mapsto \tilde{f}_n(x)$ which shows that the density defined by \eqref{eq:approxdens} is a very satisfactory approximation the distribution of the singular values  of $X_{N,M}$. In Figure \ref{fig:ConstVarProfile}(b), we  plot the curve $\theta \mapsto \max\left(1 + \sqrt{\frac{N}{M}} , \lambda_{N/M}\left( \theta \right)\right) $ for $\theta \in [0,3]$, which gives the location of outliers for any $\theta > \left( \frac{N}{M}\right)^{1/4} \approx 0.974$.

For a set of regularly spaced values of $\theta$ on $[0,3]$, we also report the results of the numerical procedure that we use to compute a minimizer $\bar{\lambda}(\theta)$  of the function $\lambda \mapsto \det\big(\beta_{2}^\square(\lambda ) \big)$ over $\R_+$. In Figure \ref{fig:ConstVarProfile}(c), we display the curve $\theta \mapsto \det\big(\beta_{2}^\square(\bar{\lambda}(\theta) ) \big)$. It can be seen that this curve is close to zero for $\theta > \left( \frac{N}{M}\right)^{1/4} $, and that it does not vanish for smallest values of $\theta$ which is in agreement with the fact that there is no outlier for $\theta \leq \left( \frac{N}{M}\right)^{1/4} $. In Figure \ref{fig:ConstVarProfile}(b), we also display the curve $\theta \mapsto  \bar{\lambda}(\theta)$ found by numerical minimization which coincides with $\theta \mapsto \lambda_{N/M}\left( \theta \right)$ for $\theta > \left( \frac{N}{M}\right)^{1/4} $. Interestingly, for all values of $\theta$ smaller than $\left( \frac{N}{M}\right)^{1/4} $ it appears that $\bar{\lambda}(\theta) =  1 + \sqrt{\frac{N}{M}}$, suggesting that $\lambda \mapsto \det\big(\beta_{2}^\square(\lambda))$ admits a minimizer at  the bulk edge.

\subsubsection{Piecewise constant variance profile} We now consider the following example of a piecewise constant variance profile
\begin{equation}
\Gamma_N = \left[\begin{array}{cc} \gamma_1  \mbb \one_{N/4}\one_{M/4}^*  & \gamma_2  \mbb \one_{N/4}\one_{3M/4}^* \\ \gamma_2  \mbb \one_{3N/4}\one_{M/4}^* & \gamma_1  \mbb \one_{3N/4}\one_{3M/4}^*  \end{array}\right], \label{eq:piecewise}
\end{equation}
where $\one_{q}$ denotes the vector of length $q$ with all entries equal to one, and $\gamma_1,\gamma_2$ are positive constant such that $\gamma_2 = 200 \times \gamma_1$.  Then, we compare two settings where either  $u_N$ and $v_M$ are unit vectors with constant entries, or $u_N$ (resp.\ $v_M$) is equal to the first vector $e_1^N$ (resp.\ $e_1^M$) of the canonical basis of $\R^N$ (resp. $\R^M$). 

\begin{figure}[!t]
\subfigure[$\theta = 0.39$]{\includegraphics[width=0.45\linewidth]{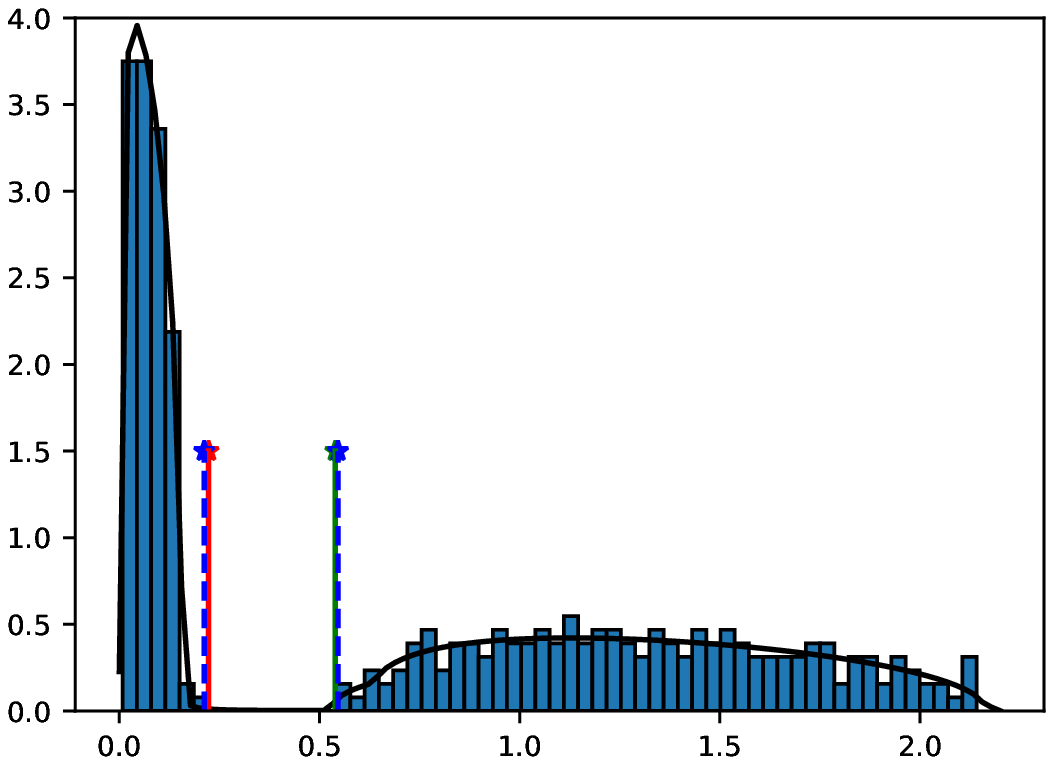}}
\subfigure[$\theta = 0.57$]{\includegraphics[width=0.45\linewidth]{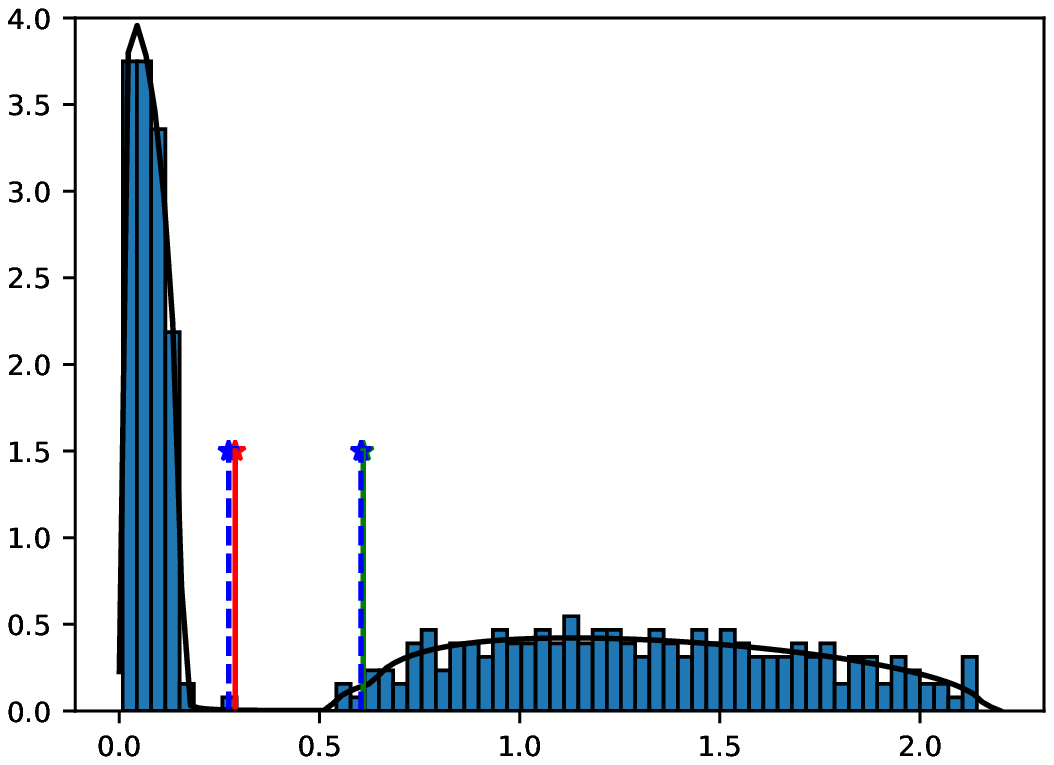}}

\subfigure[$\theta = 0.81$]{\includegraphics[width=0.45\linewidth]{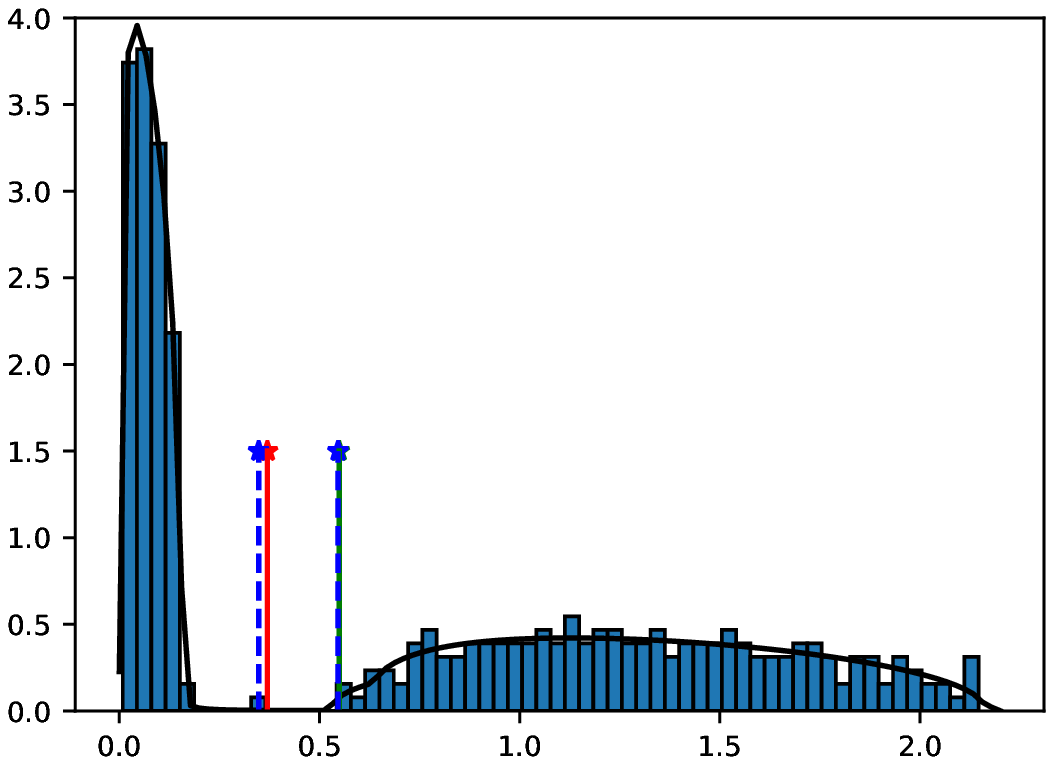}}
\subfigure[$\theta = 0.93$]{\includegraphics[width=0.45\linewidth]{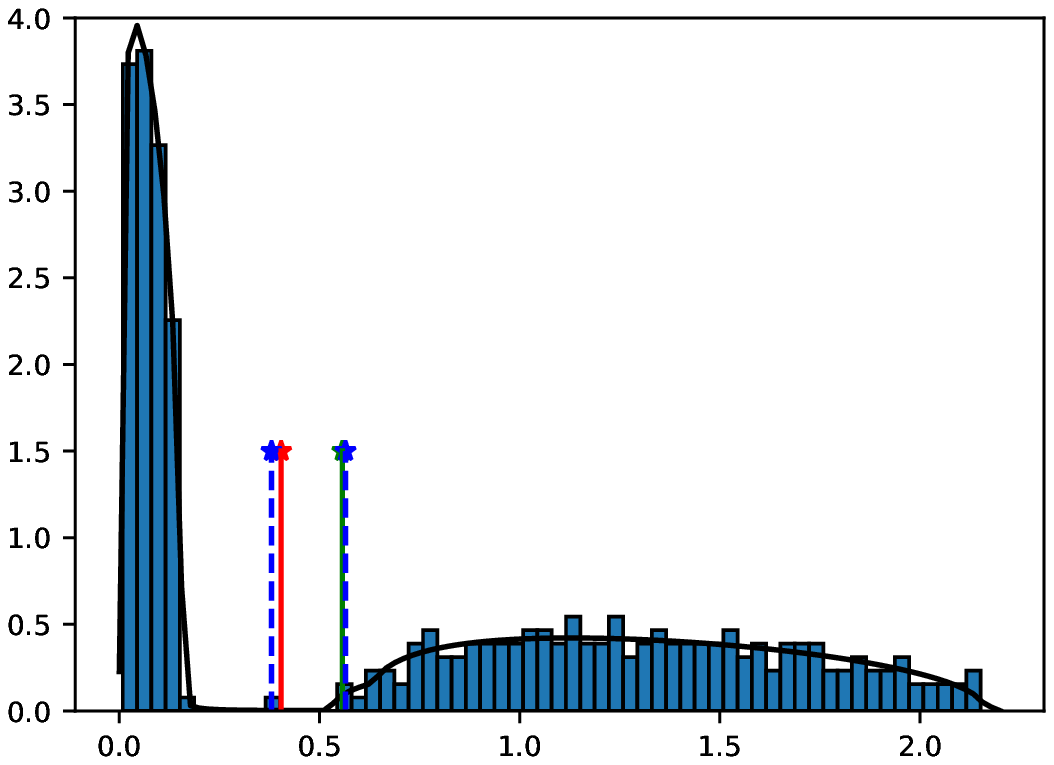}}

\caption{Piecewise constant variance profile. (c-f) Histograms of the singular values of one realization of $H'_{N,M} = X_{N,M} + \theta u_{N}  v^{\ast}_{M}$  and  smooth approximation  of the singular values distribution of $X_{N,M}$  (black curve) for $u_N = \frac{1}{\sqrt{N}} \bone_N$ and $v_M= \frac{1}{\sqrt{M}} \bone_M$. The red (resp.\ green) vertical line  denotes the value $\bar{\lambda}(\theta)$ when $u_N = \frac{1}{\sqrt{N}} \bone_N$ and $v_M= \frac{1}{\sqrt{M}} \bone_M$    (resp.\   $u_N = e_1^N$ and  $v_M = e_1^M$), while the blue vertical dashed line  denotes the location of the singular value of $H'_{N,M}$ which is the closet to $\bar{\lambda}(\theta)$.}
  \label{fig:PiecewiseConstVarProfile3_4}
\end{figure}

\begin{figure}[!t]
\centering
\subfigure[]{\includegraphics[width=0.4\linewidth]{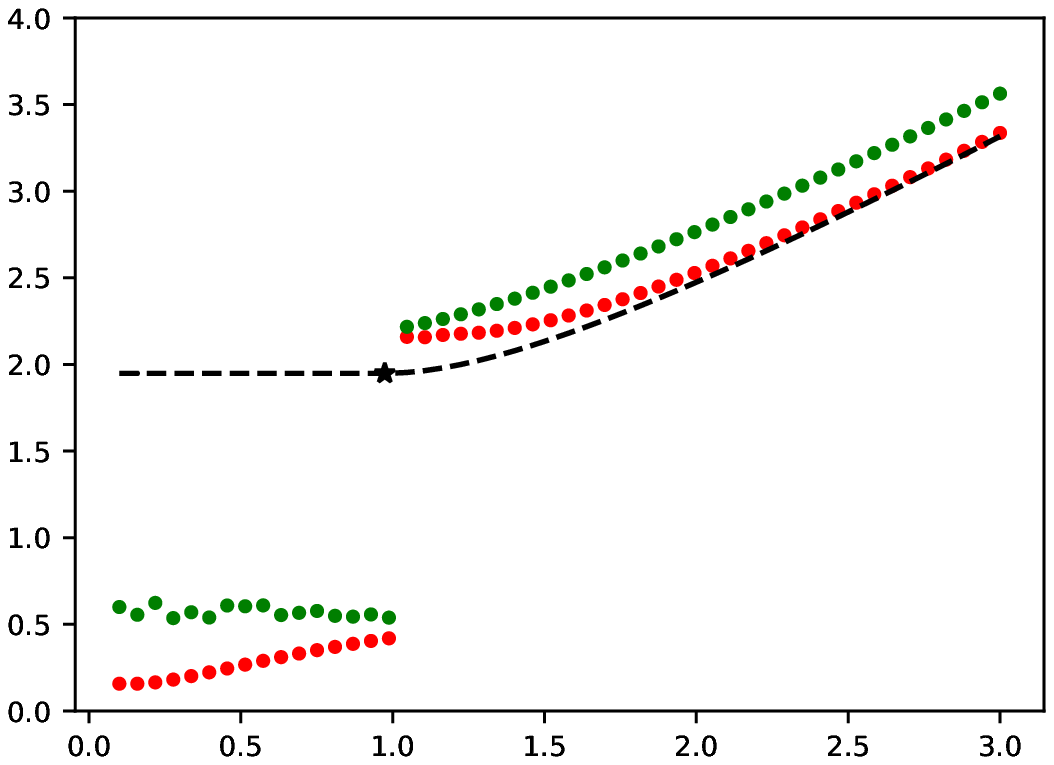}}\hspace{1.3cm}
\subfigure[]{\includegraphics[width=0.4\linewidth]{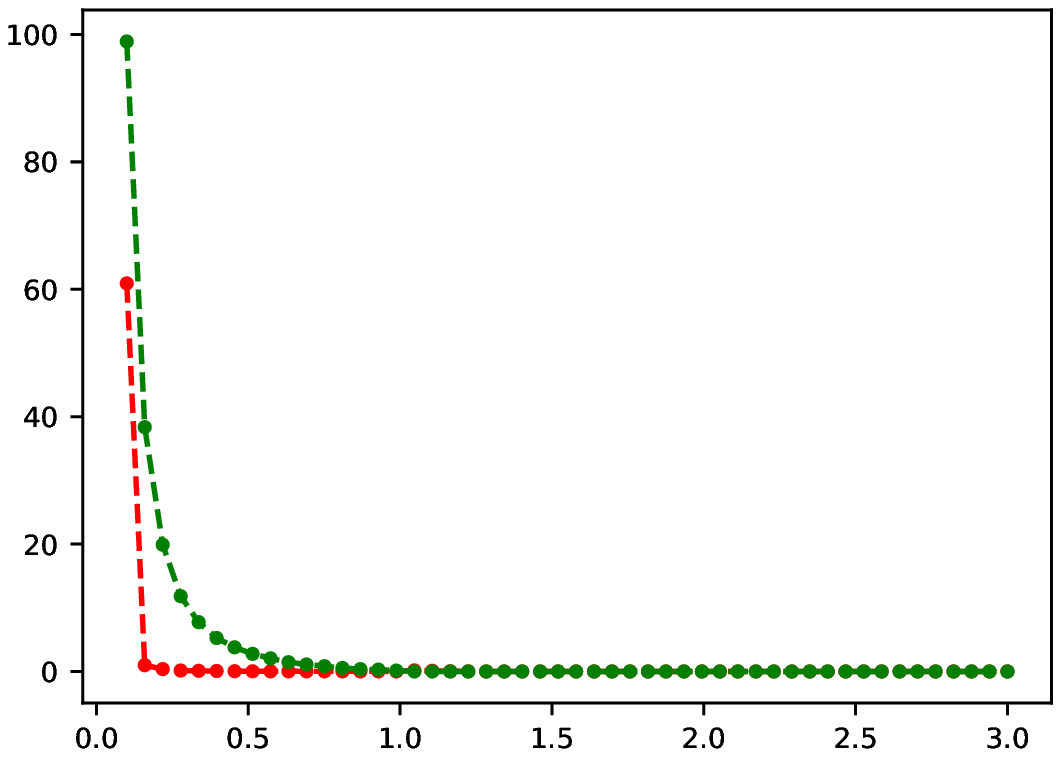}} 
\caption{Piecewise constant variance profile. (a) The dashed black line is the curve $\theta \mapsto \max\left(1 + \sqrt{\frac{N}{M}} , \lambda_{N/M}\left( \theta \right)\right)$, and the red (resp.\ green) dots are the points $(\theta,\bar{\lambda}(\theta))$ when $u_N = \frac{1}{\sqrt{N}} \bone_N$ and $v_M= \frac{1}{\sqrt{M}} \bone_M$    (resp.\   $u_N = e_1^N$ and  $v_M = e_1^M$). When $u_N$ and $v_M$ are unit vectors with constant entries, outliers are generated within the interval $[0.2,0.5]$ for spikes $\theta \in [0.28,1]$. (b) The dashed lines are the curves $\theta \mapsto \det\big(\beta_{2}^\square(\bar{\lambda}(\theta) ) \big)$ depending on the choice of $(u_N,v_M)$. }  \label{fig2:PiecewiseConstVarProfile3_4}
\end{figure}

In Figure \ref{fig:PiecewiseConstVarProfile3_4}, we display the histogram of the singular values of one realization of $H'_{N,M}  = X_{N,M} +  \theta u_{N}  v^{\ast}_{M}$ for  different values of $\theta < 1$ for $u_N = \frac{1}{\sqrt{N}} \bone_N$ and $v_M= \frac{1}{\sqrt{M}} \bone_M$. When  $u_N$ and $v_M$ are unit vectors with constant entries, a spike $\theta \in \{0.39,0.57,0.81,0.93\}$ clearly generates an outlier at $\lambda \in [0.2,0.5]$, while in the case $u_N = e_1^N$ and  $v_M = e_1^M$  there is no outlier for such values of the spike. For each setting, we also display  in Figure \ref{fig2:PiecewiseConstVarProfile3_4}   the curves $\theta \mapsto \bar{\lambda}(\theta)$ and $\theta \mapsto \det\big(\beta_{2}^\square(\bar{\lambda}(\theta) ) \big)$ where $\bar{\lambda}(\theta)$  is found by numerical minimization of $\lambda \mapsto \det\big(\beta_{2}^\square(\lambda ) \big)$ over $\R_+$.  In the case where $u_N$ and $v_M$ have constant entries, the value of $\det\big(\beta_{2}^\square(\bar{\lambda}(\theta) ) \big)$
is close to zero when $\theta  \in [0.28,1]$ which confirms the existence of outliers for  values of the spike within this interval. When $u_N = e_1^N$ and  $v_M = e_1^M$, one has clearly that $\det\big(\beta_{2}^\square(\bar{\lambda}(\theta) ) \big) \neq 0$ when $\theta  \in [0.28,1]$  and thus, for such spikes, there is no outlier. In both settings, when $\theta$ is sufficiently large (e.g.\ $\theta \geq 1$), there exists an outlier $\bar{\lambda}(\theta) > 1.5$ but with a location depending on the values of $u_N$ and $v_M$.

\subsubsection{Bernoulli variance profile} We now consider variance profiles whose entries may be equal to zero and are chosen randomly (and independently) as follows. Each entry of $\Gamma_{N,M}$  takes either the value zero with probability $1-p$ (for some $0 < p < 1$)  or a fixed positive value $\gamma^2 > 0$ with probability $p$. After randomly fixing the entries of  $\Gamma_{N,M}$ in this way, the value of $\gamma^2$ is chosen such that the normalisation condition \eqref{eq:normcond} is satisfied.  In Figure \ref{fig:BernouVarProfile}, we display the histogram of the singular values of one realization of $H'_{N,M}  = X_{N,M} +  \theta u_{N}  v^{\ast}_{M}$ for $\theta = 2$,   with  $u_N = \frac{1}{\sqrt{N}} \bone_N, v_M= \frac{1}{\sqrt{M}} \bone_M$, and for various  Bernoulli variance profiles by letting the  sampling probability $p$ ranging from $\frac{5}{M}$ to $\frac{40}{M}$. For  values of $p$ larger than $\frac{5}{M}$, the value $\bar{\lambda}(2)$ is an accurate approximation of the location of an outlier in the singular values distribution of $H'_{N,M}$. The shape of the  smooth approximation by $\tilde{f}_{n}$ of the singular values distribution of $X_{N,M}$  clearly depends on the value of $p$.

\begin{figure}[!t]
\centering%
\subfigure[$p = \frac{5}{M}, \gamma_{\max}^2 \approx 72$]{\includegraphics[width=0.40\linewidth]{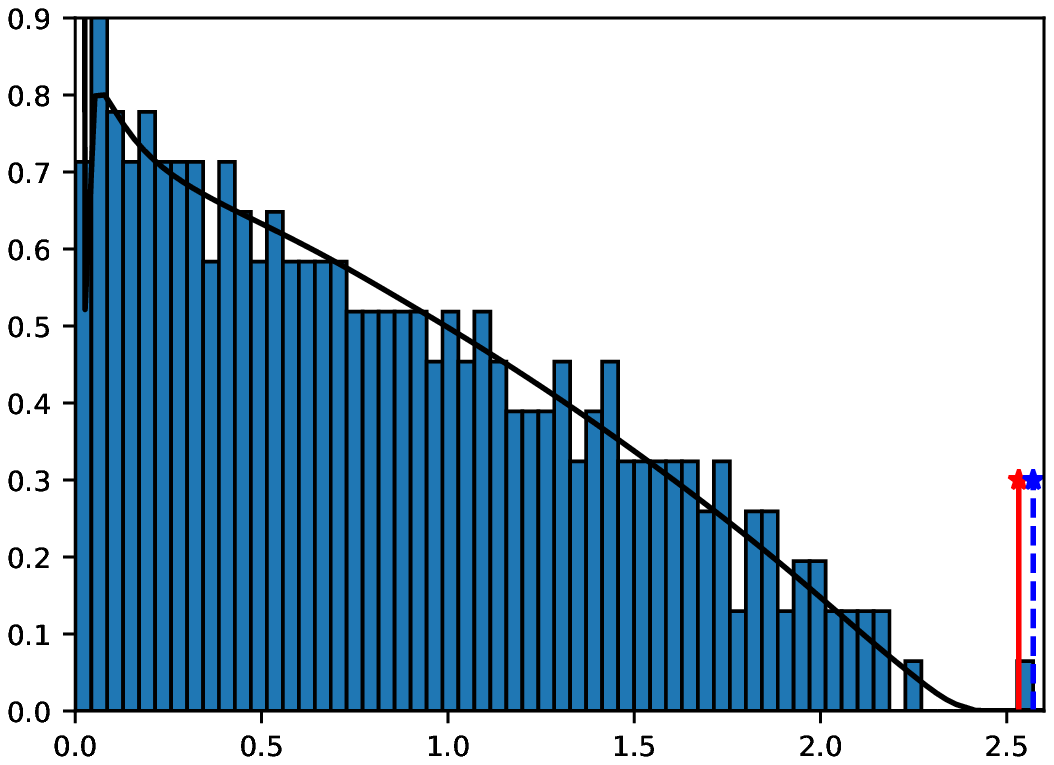}}
\subfigure[$p = \frac{10}{M}, \gamma_{\max}^2 \approx 36$]{\includegraphics[width=0.40\linewidth]{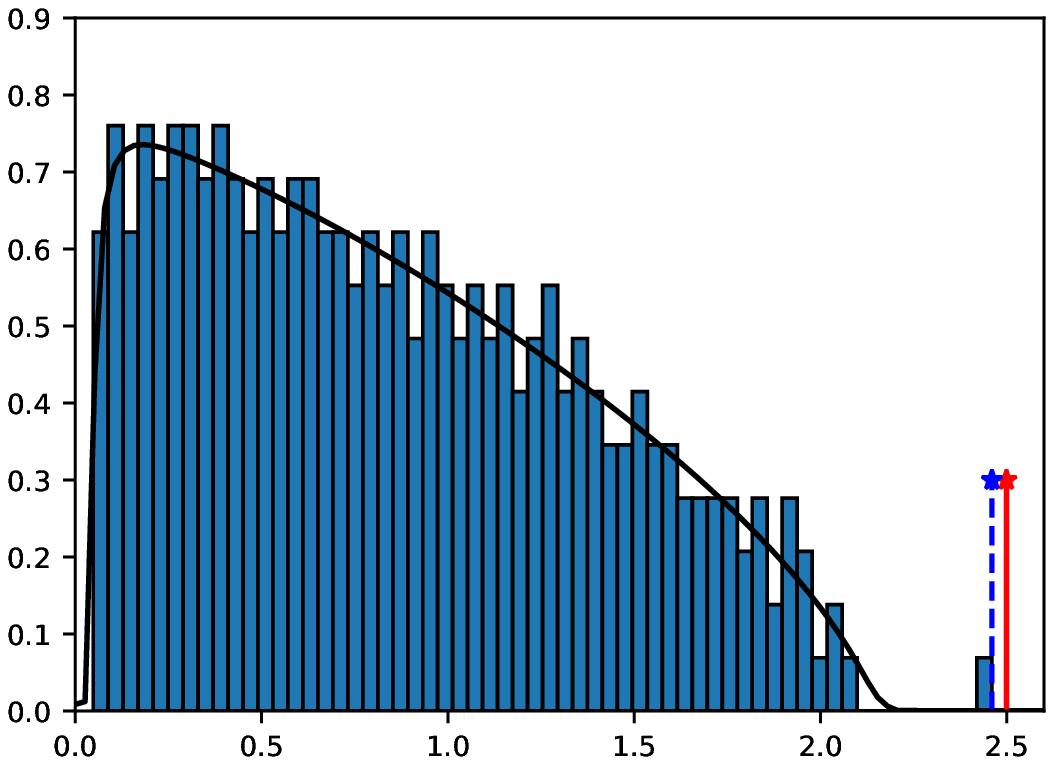}}

\subfigure[$p = \frac{20}{M}, \gamma_{\max}^2 \approx 18$]{\includegraphics[width=0.40\linewidth]{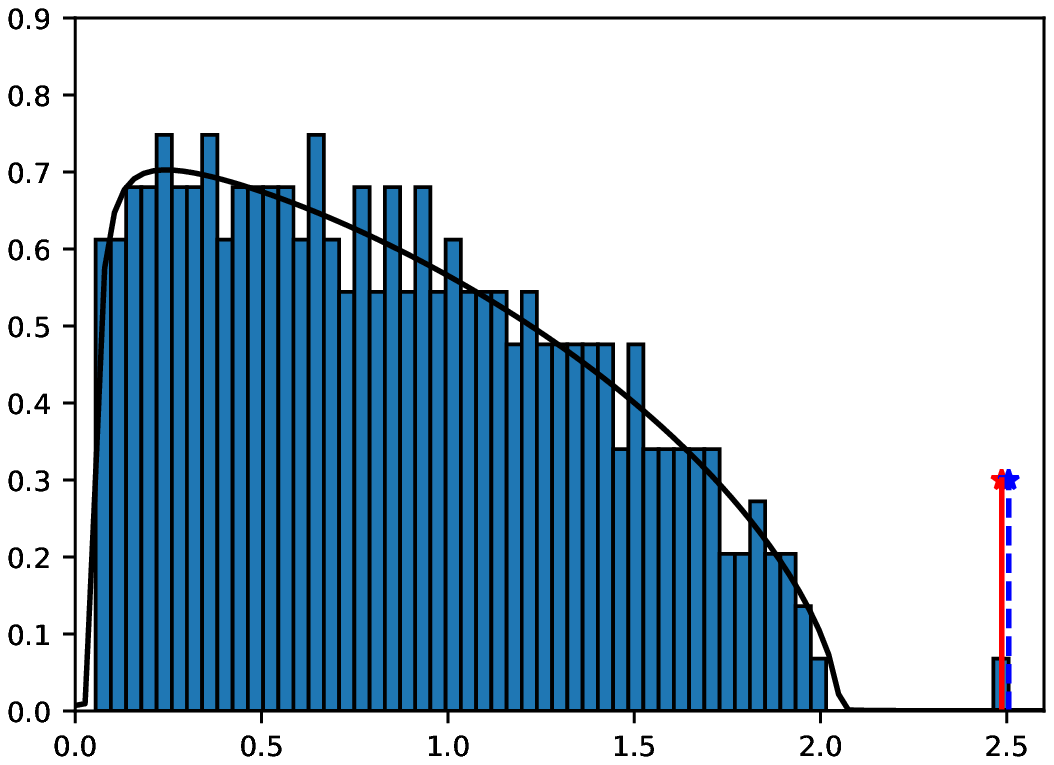}}
\subfigure[$p = \frac{40}{M}, \gamma_{\max}^2 \approx 9$]{\includegraphics[width=0.40\linewidth]{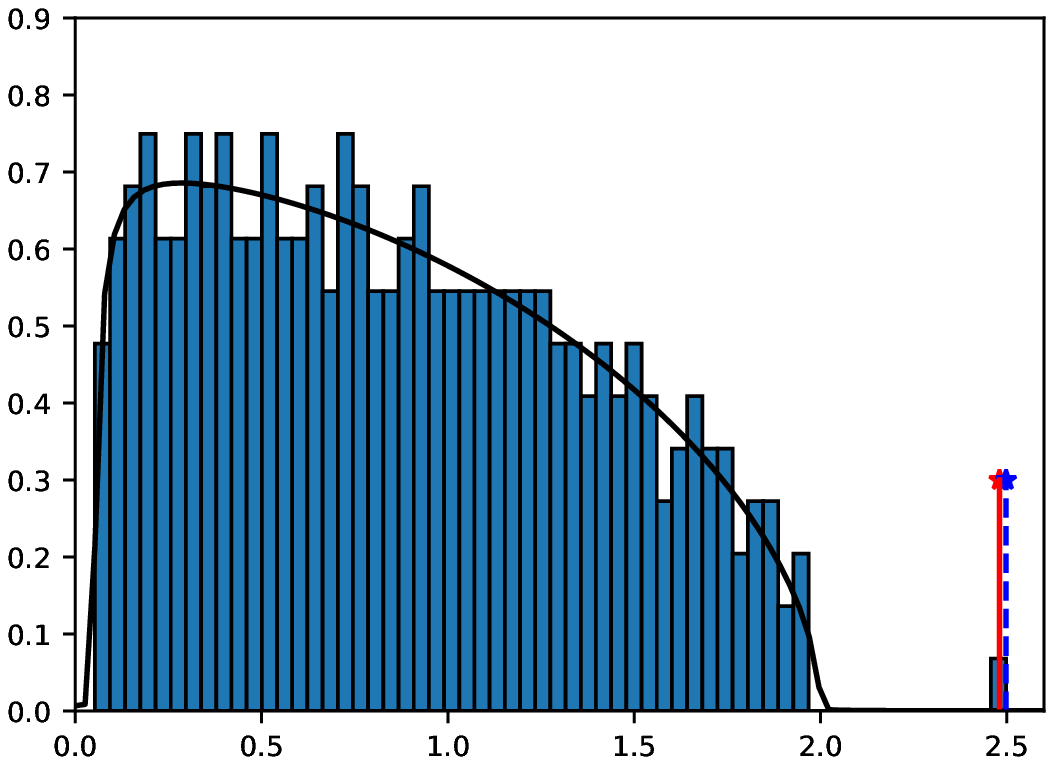}}

\caption{Bernoulli variance profile. Histogram of the singular values of one realization of $H'_{N,M}$ for $\theta = 2$ and with a Bernoulli variance profile for different values of the sampling probability $p$. In each figure, the black curve is the smooth approximation by $\tilde{f}_{n}$ of the singular values distribution of $X_{N,M}$, the red vertical line  denotes the value $\bar{\lambda}(2)$  which is  the approximation of the location of an outlier, while the blue vertical dashed line  denotes the location of the singular value of $H'_{N,M}$ which is the closet to $\bar{\lambda}(2)$.}
  \label{fig:BernouVarProfile}
\end{figure}

\subsubsection{Doubly stochastic variance profile}

 We finally assume that $N=M = 400$, and we consider the setting where $\frac{\Gamma_{N,M}}{M} $ is a doubly stochastic matrix. Under such an assumption, an explicit solution of the equation $\det\big(\beta_{2}^\square(\lambda ) \big) = 0$ exists as stated below.

\begin{Lem}\label{Lem:lambdatheta1}  
Assume that $\frac{\Gamma_{N,M}}{M} $ is a doubly stochastic matrix. Then, the Equation \eqref{eq:detbeta2} admits a solution given by
\begin{equation}
\lambda_{1}\left( \theta \right) = \frac{(1+\theta^{2})}{\theta} = \theta^{-1} + \theta \mbox{ provided that } \theta > 1, \label{eq:lambdatheta1}
\end{equation}
\end{Lem}

\begin{proof} 
Under the assumption that the variance profile is doubly stochastic,  it can be easily shown that the solution to \eqref{FixPtEq2} is a scalar matrix $$G_{\DD(H)}^\square( \lambda \mbb I_{N+M}) = g_N^\square(\lambda) \mbb I_{N+M},$$ where $g^\square_N$ is complex-valued function satisfying
$
g_N^\square(\lambda) = (\lambda - g_N^\square(\lambda))^{-1}
$
for $\lambda \in \mbb C^+$. Therefore, one has that  $g_N^\square(\lambda) = \frac{\lambda - \sqrt{\lambda^2 - 4}}{2}$ which is the Stieltjes transform of the semicircular law. Hence, using exactly the same calculations than those made for a constant variance profile to derive Lemma \ref{Lem:lambdatheta2}, one obtains that, for any unit vectors $u_N$ and $v_N$, the equation
$$\det\big(\beta_{2}^\square(\lambda ) \big) = 1 - \theta^2 \big( g_N^\square(\lambda) \big)^2= 0$$ admits a solution given by \eqref{eq:lambdatheta1} provided that $\theta > 1$.
\end{proof}

\begin{figure}[!t]
\centering%
\subfigure[$K = 1, \gamma_{\max}^2 = 400$]{\includegraphics[width=0.40\linewidth]{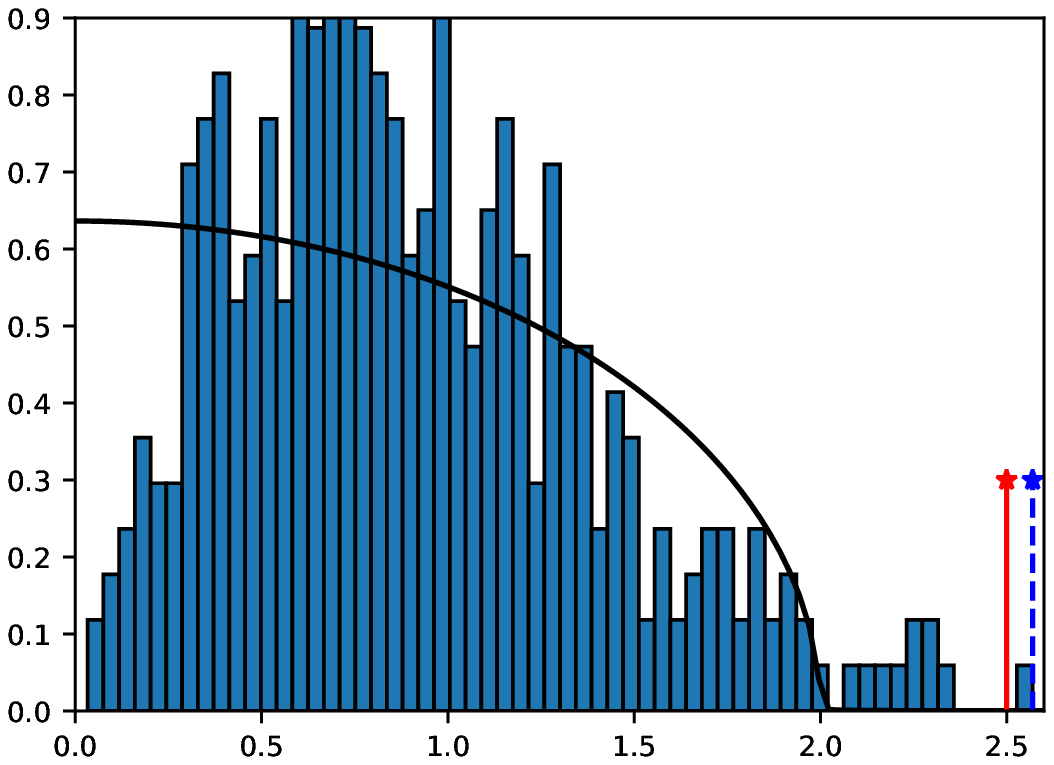}}
\subfigure[$K = 2, \gamma_{\max}^2 = 400$]{\includegraphics[width=0.40\linewidth]{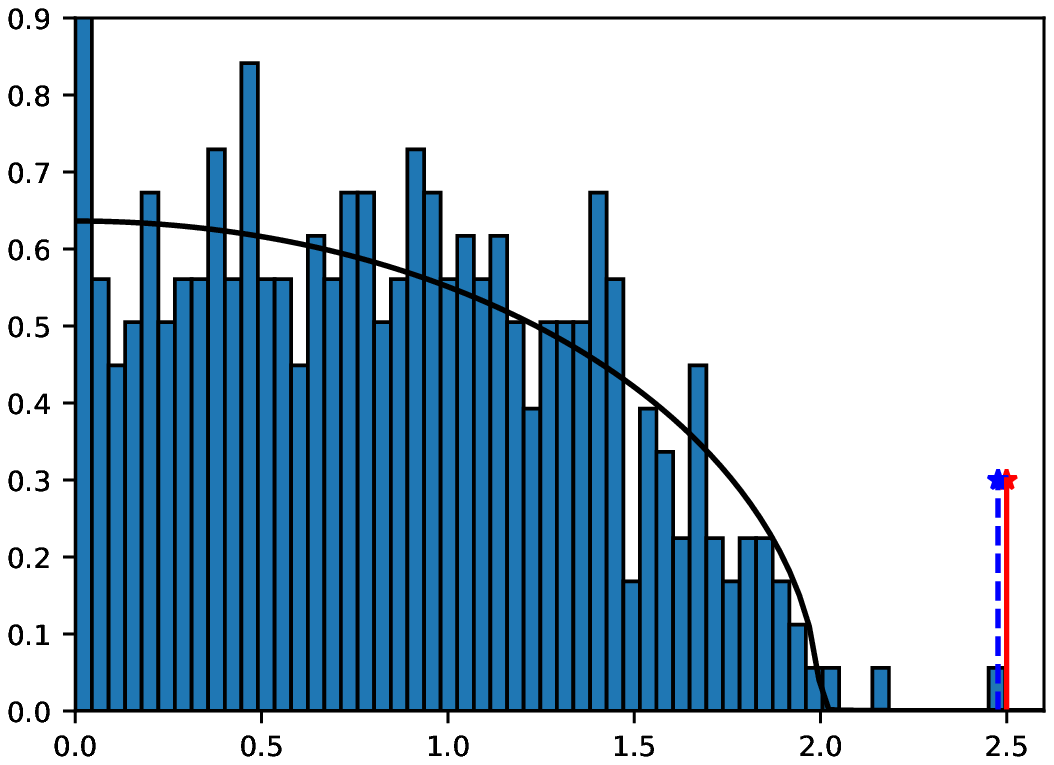}}

\subfigure[$K = 4, \gamma_{\max}^2 = 200$]{\includegraphics[width=0.40\linewidth]{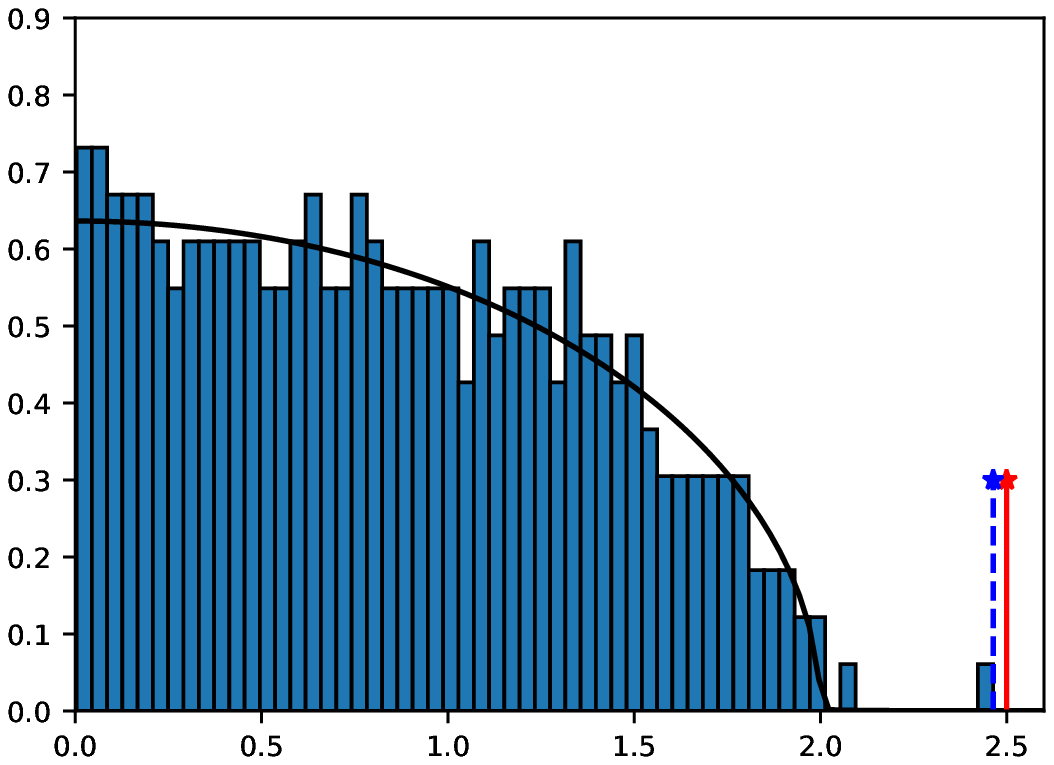}}
\subfigure[$K = 8, \gamma_{\max}^2 = 100$]{\includegraphics[width=0.40\linewidth]{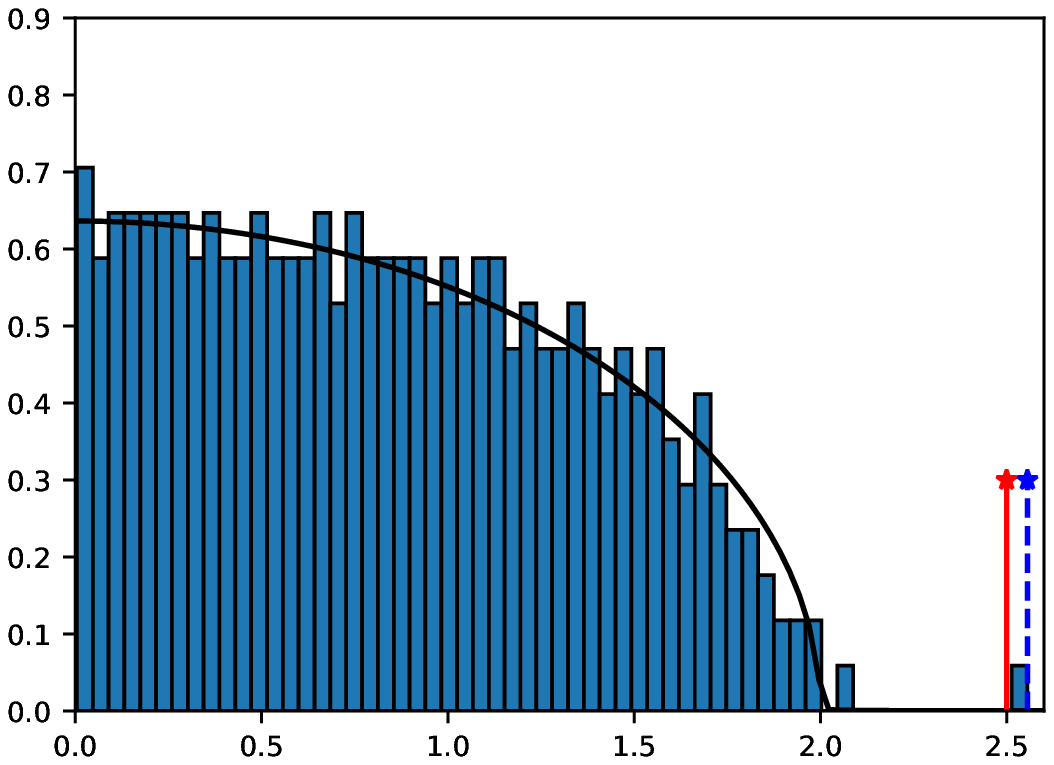}}

\caption{Doubly stochastic variance profile. Histogram of the singular values of one realization of $H'_{N,M}$ for $\theta = 2$ and with a variance profile given by \eqref{eq:varprofsto} for different values of $K$. In each figure, the black curve is the smooth approximation by $\tilde{f}_{n}$ of the singular values distribution of $X_{N,M}$, the red vertical line  denotes the value $\lambda_{1}\left(2\right) = 2.5$  which is  the approximation of the location of an outlier, while the blue vertical dashed line  denotes the location of the singular value of $H'_{N,M}$ which is the closet to $\lambda_{1}\left(2\right)$.}
  \label{fig:DoubleStoVarProfile}
\end{figure}

Interestingly, for a doubly stochastic variance profile, we obtain exactly the expression \eqref{eq:lambdatheta} of the asymptotic location of an outlier in the standard Gaussian spike population model for the ratio $c = 1$. In Figure \ref{fig:DoubleStoVarProfile}, we display the histogram of the singular values of one realization of $H'_{N,M}  = X_{N,M} +  \theta u_{N}  v^{\ast}_{M}$ with $\theta = 2$ and $u_N$ and $v_M$  chosen to be unit vectors with constant entries. The normalized variance profile $\frac{\Gamma_{N,M}}{M}$ of $X_{N,M}$ is chosen as follows
\begin{equation}
\frac{\Gamma_{N,M}}{M} =  \frac{1}{K}\sum_{k=1}^{M} P_{k}, \label{eq:varprofsto}
\end{equation}
where $P_{1},\ldots,P_{K}$ are permutation matrices (obtained by random permutations of the columns of the identity matrix $\mbb I_{N}$). For small values of $K$,  such variance profiles have many entries equal to zero.   In Figure \ref{fig:DoubleStoVarProfile},  we display the histogram of the singular values of one realization of $H'_{N,M}  = X_{N,M} +  \theta u_{N}  v^{\ast}_{M}$ with $\theta = 2$ for different values of $K = 1,2,4,8$. For $K \geq 4$, there is clearly an outlier located  approximately at $\lambda_{1}\left( \theta \right) = 2.5$. The quality of the smooth approximation by $\tilde{f}_{n}$ of the singular values distribution of $X_{N,M}$ also clearly depends on the value of $\gamma_{\max}^2$ (maximum value of the entries of the variance profile $\Gamma_{N,M}$) which is consistent with our theoretical results  in Theorem \ref{MainTh} on the control of the deviation between the deterministic equivalent $g_{H_N}^\square$ and the Stieltjes transform $g_{H_N}$.

\subsection{General deformed models}

Let us now consider  the general setting of the rectangular information plus noise model \eqref{Model4} with $Y_{N,M} \neq 0$.  

\subsubsection{Simulated model}

Taking again $N=360$ and $M=400$ we generate one realization from the model  $H'_{N,M} = X_{N,M} + Y_{N,M} + Z_{N,M}$ as follows. The matrix $X_{N,M}$ is a Gaussian matrix with piecewise constant variance profile given by \eqref{eq:piecewise}. Then, we generate a $N \times M$ matrix $W_{N,M}$ with i.i.d.\ real entries sampled from a Gaussian distribution with zero mean and variance $\tau^2 = \frac{1}{4M}$, that we write using singular value decomposition (SVD) as $W_{N,M} = U \Sigma V^\ast$. Denoting $(U_j)_{1 \leq j \leq  3}$ (resp.\ $(V_j)_{1 \leq j \leq  3}$) the left (resp.\ right) singular vectors associated to the three largest singular values $(\sigma_j)_{1 \leq j \leq 3}$ of $W_{N,M}$, we  finally define
$$
 Y_{N,M} =  W_{N,M} - \sum_{j=1}^{3} \sigma_j U_j V_j^{\ast} \quad \mbox{and} \quad Z_{N,M} = \sum_{j=1}^{3} \theta_j U_j V_j^{\ast} \; \mbox{ with $\theta_1=4$, $\theta_2 = 3$, $\theta_3 = 2$}.
$$
The histograms of the s.v.d.\ of $X_{N,M}$, $Y_{N,M}$ and $H'_{N,M}$ are displayed in Figure \ref{fig:simulatedgenmodel}. It can be seen that the $k=3$ spikes of $Z_{N,M}$ clearly generate 3 outliers. 

In Figure \ref{fig:simulatedgenmodel}(c), we also display the smooth approximation of the s.v.d.\ of the random matrix $H_{N,M}=  X_{N,M} + Y_{N,M}$ using the  inverse Stieltjes transform \eqref{eq:Stieltjesinv} of $g_{n}^\square(\lambda) = \frac{1}{N+M} \Tr \,G_{n}^\square( \lambda)$.  Since $Y_{N,M}$ is not a diagonal matrix, the deterministic equivalent $G_{n}^\square( \lambda)$ of the operator-valued Stieltjes transform of $H_{N,M}$  is obtained by iterating the following {\it matrix equation}  
\begin{equation}
 G_{n+1}^\square( \lambda) =\Delta\bigg [ \Big( \lambda \mbb I_{N+M}   -\mcal R_{N+M}\big(   G_{n}^\square( \lambda)\big) - Y_{N,M}   \Big)^{ -1}\bigg]. \label{eq:matrixfixedpoint}
\end{equation}

In Figure \ref{fig:simulatedgenmodel}(d), we also report the curves of the mapping $\lambda \mapsto \log(\beta_k^\square(\lambda))$, where $\beta_k^\square$ is the deterministic equivalent defined by \eqref{eq:defbetasquare},  and of the mapping $\lambda \mapsto \log(\tilde{\beta}_k^\square(\lambda))$ defined by \eqref{eq:defbetasquaregen}. As $Y_{N,M}$ is non diagonal, the second deterministic equivalent $\tilde{\beta}_k^\square$ gives  the right prediction of the locations of the outliers which is not the case for the first one $\beta_k^\square$.  This is confirmed by the numerical simulations reported in Figure \ref{fig:simulatedgenmodel}(c), where it can be seen that the zeros of the mapping  $\lambda \mapsto \log(\tilde{\beta}_k^\square(\lambda))$  correspond to the locations of the true outliers in the s.v.d.\ of $H'_{N,M}$.

\begin{figure}[!t]
\centering%

\subfigure[s.v.d.\ of $X_{N,M}$]{\includegraphics[width=0.45\linewidth]{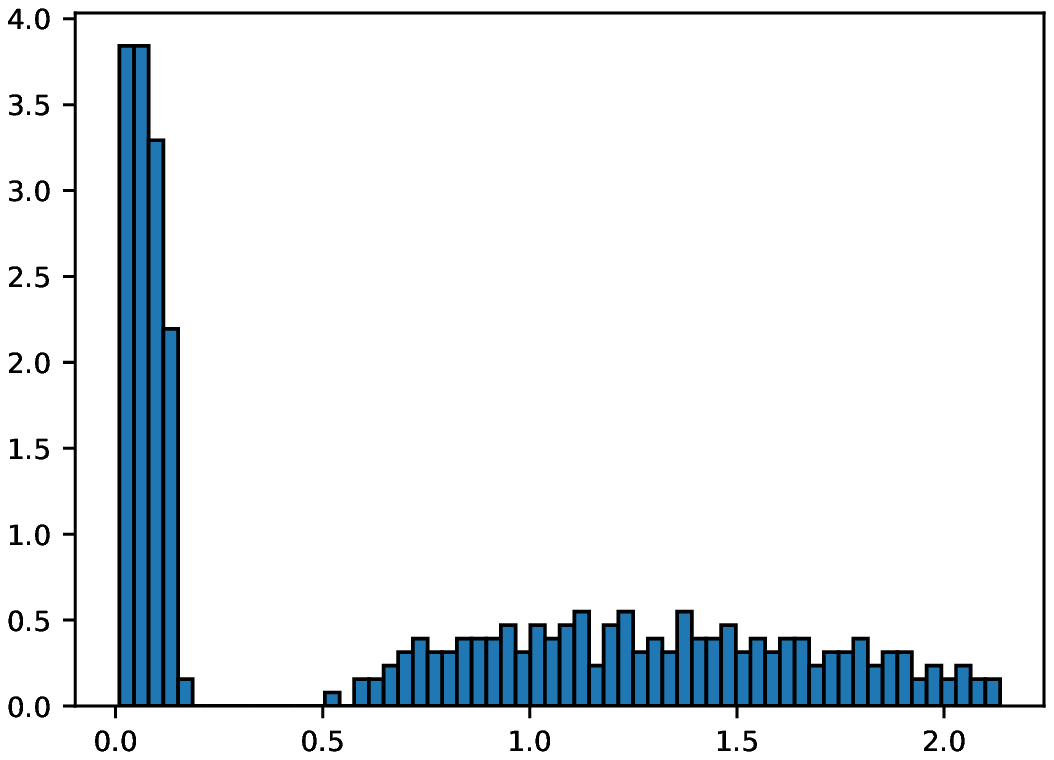}}
\subfigure[s.v.d.\ of $Y_{N,M}$]{\includegraphics[width=0.45\linewidth]{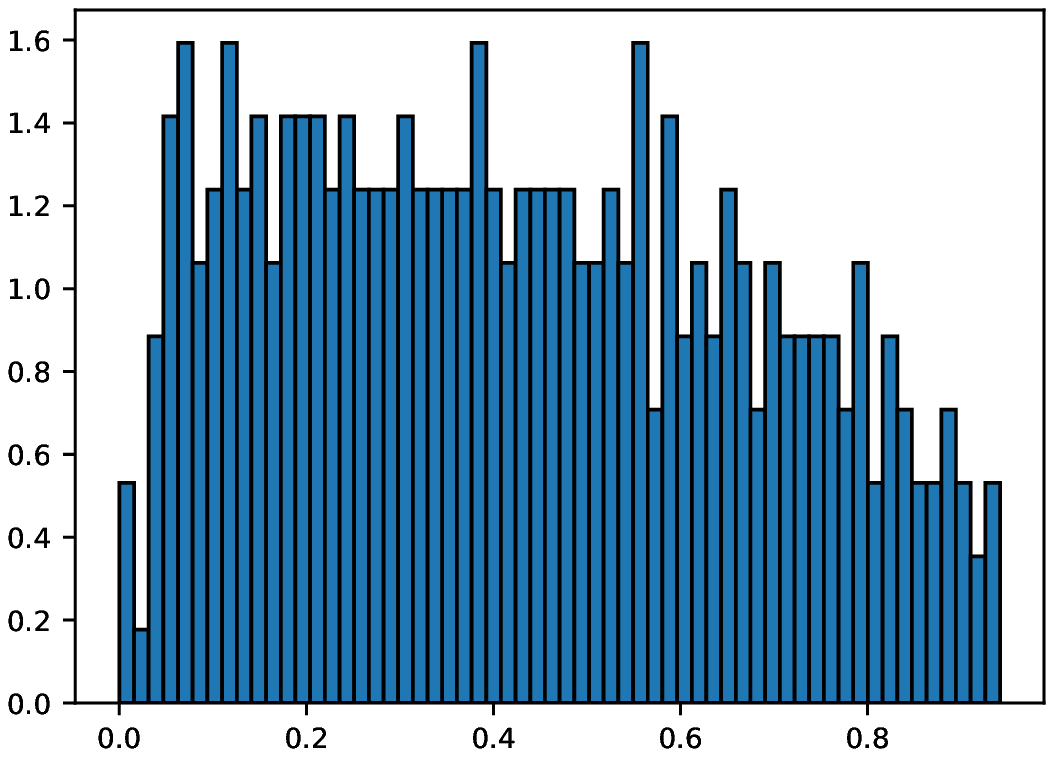}}

\subfigure[s.v.d.\ of $H'_{N,M}$]{\includegraphics[width=0.45\linewidth]{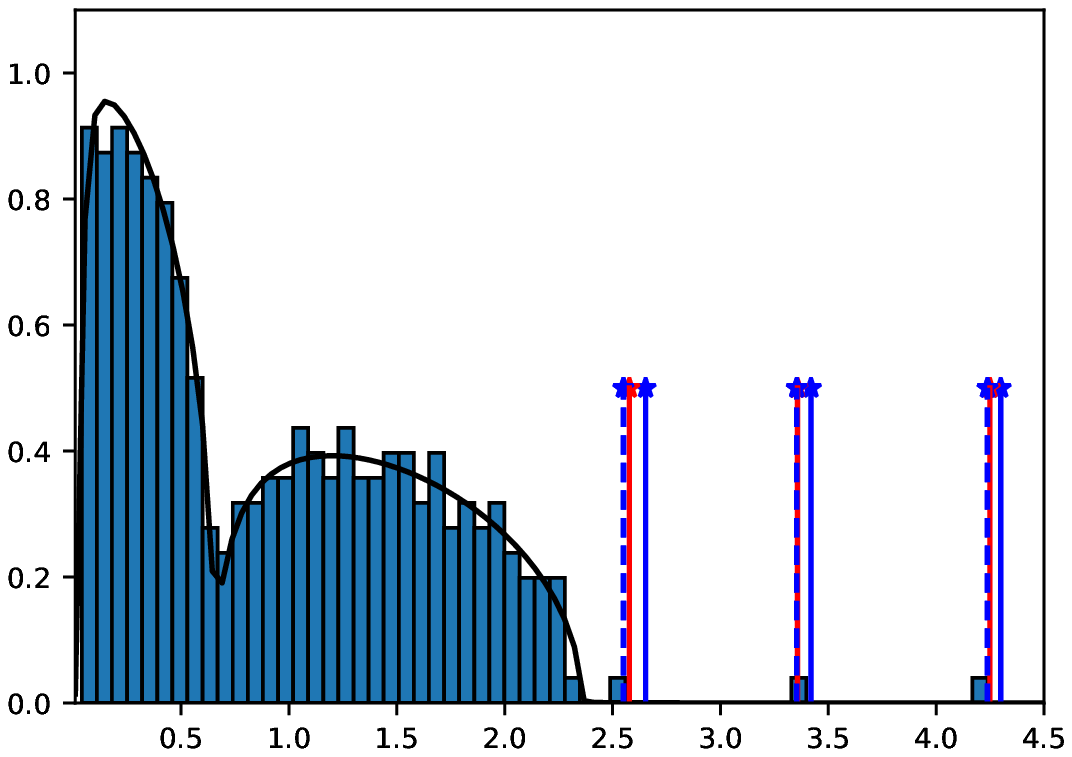}}
\subfigure[]{\includegraphics[width=0.45\linewidth]{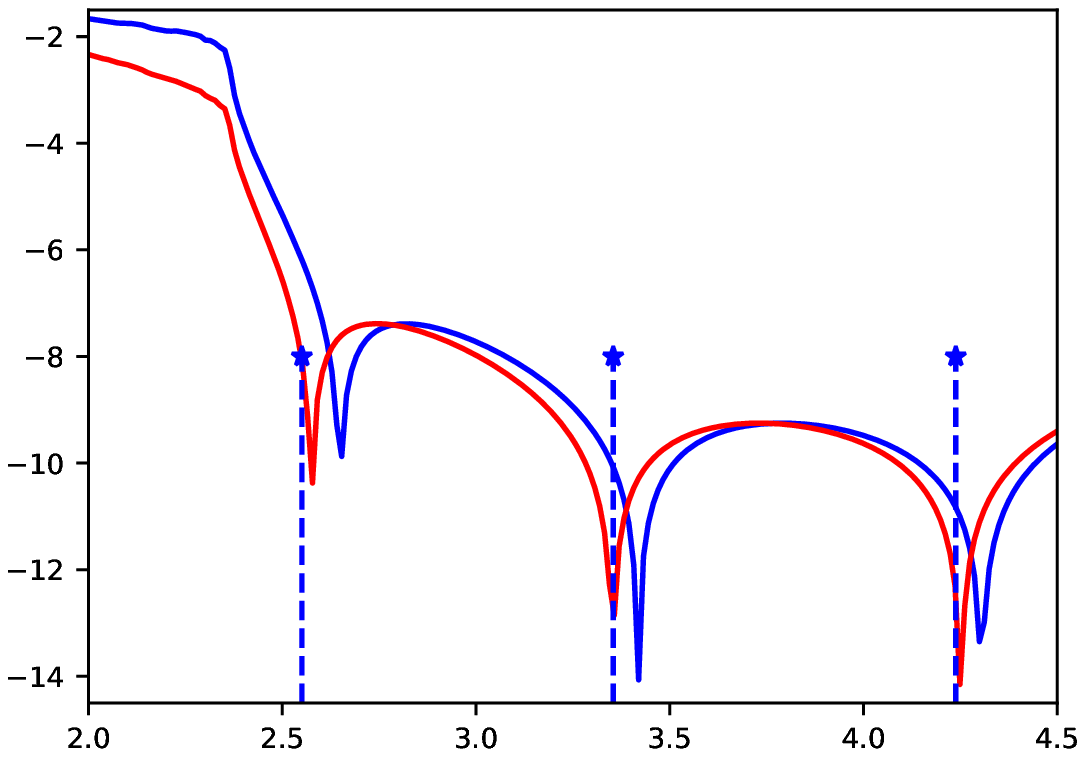}}

\caption{Deformed model with a piecewise constant variance profile. Histogram of the s.v.d.\ of (a)  $X_{N,M}$ and (b) $Y_{N,M}$. (c) Histogram of the singular values of $H'_{N,M} = X_{N,M} + Y_{N,M} + Z_{N,M}$ with $k = 3$ spikes. The black curve is the smooth approximation of the s.v.d.\ of $H_{N,M} = X_{N,M} + Y_{N,M}$. The red (resp.\ blue) vertical lines  denote the values $(\tilde{\lambda}_j)_{1 \leq j \leq 3}$ (resp.\  $(\lambda_j)_{1 \leq j \leq 3}$)  which are  the approximation of the locations of outliers given by the zeros of $\tilde{\beta}_{2k}^\square$ (resp.\ $\beta_{2k}^\square)$, (d) Graph of  $\lambda \mapsto \log\big(\det\big(\beta_{2k}^\square(\lambda) \big)\big)$ (blue curves)  and $\lambda \mapsto \log\big(\det\big(\tilde{\beta}_{2k}^\square(\lambda) \big)\big)$ (red curves). The blue vertical dashed lines  denote the locations of the true outliers in the s.v.d.\ of $H'_{N,M}$.
} \label{fig:simulatedgenmodel}
\end{figure}

\subsubsection{A model of noisy images with heteroscedasticity}

To conclude this section on numerical experiments, we study an example inspired by the problem of low-rank matrix denoising in image processing in the presence of Poisson noise. In this setting, one observes a $N \times M$ data matrix such that each $(i,j)$-th entry is independently sampled from a Poisson distribution with parameter $\kappa_{i,j} > 0$. Under such an assumption, the expectation and variance of each entry are thus equal to $\kappa_{i,j}$.

\begin{figure}[!t]
\centering%
\subfigure[]{\includegraphics[width=0.48\linewidth]{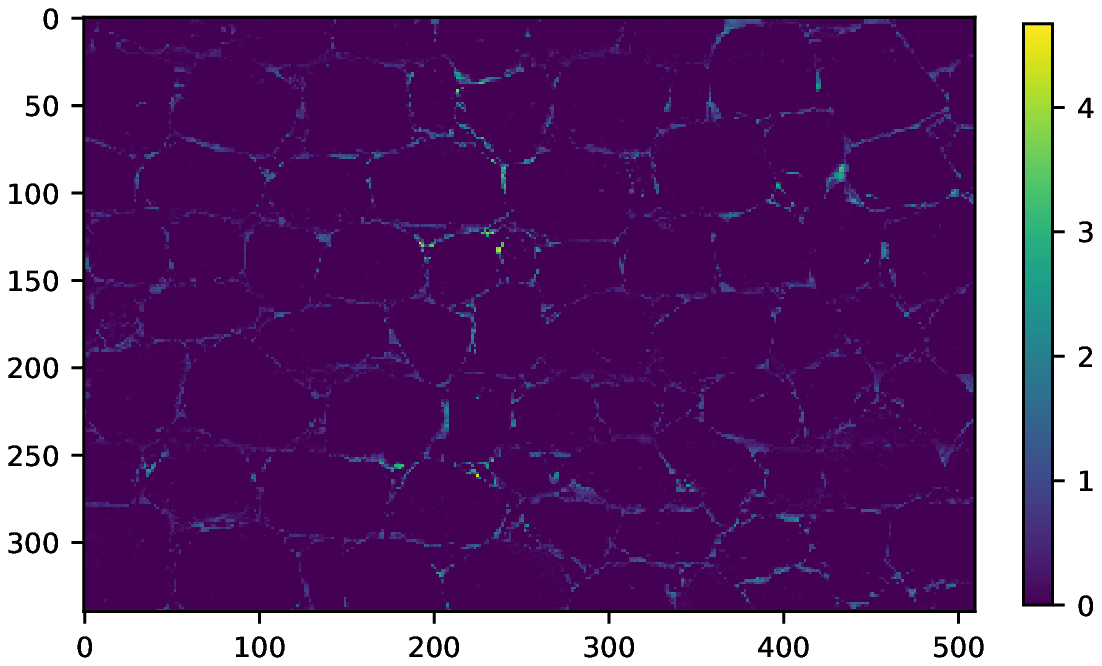}}
\subfigure[]{\includegraphics[width=0.48\linewidth]{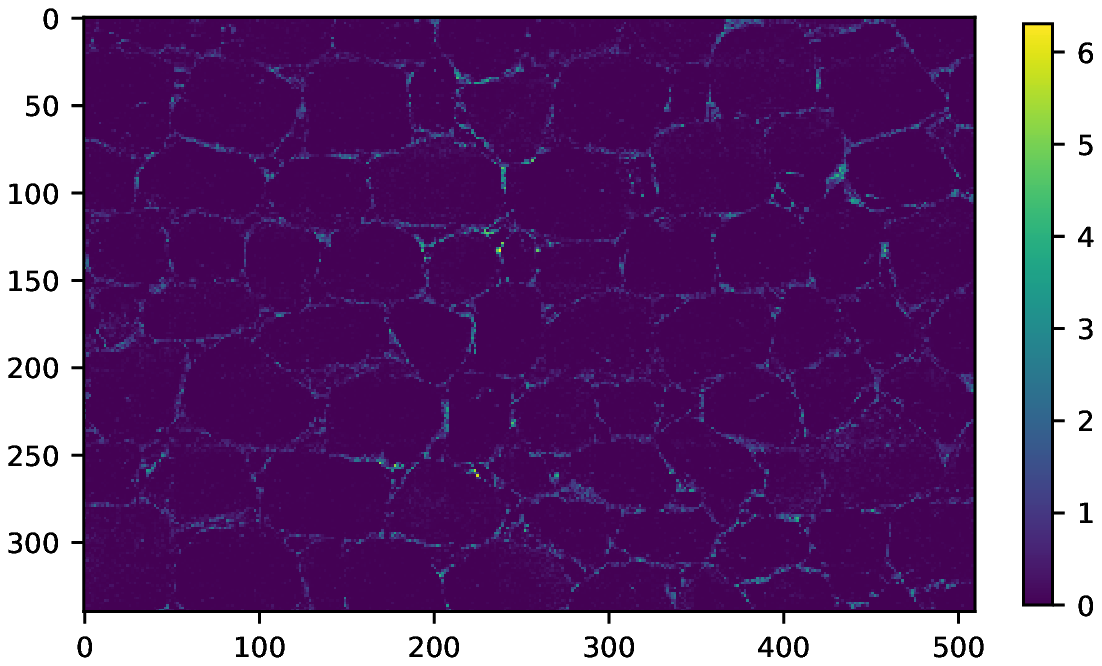}}

\subfigure[]{\includegraphics[width=0.48\linewidth]{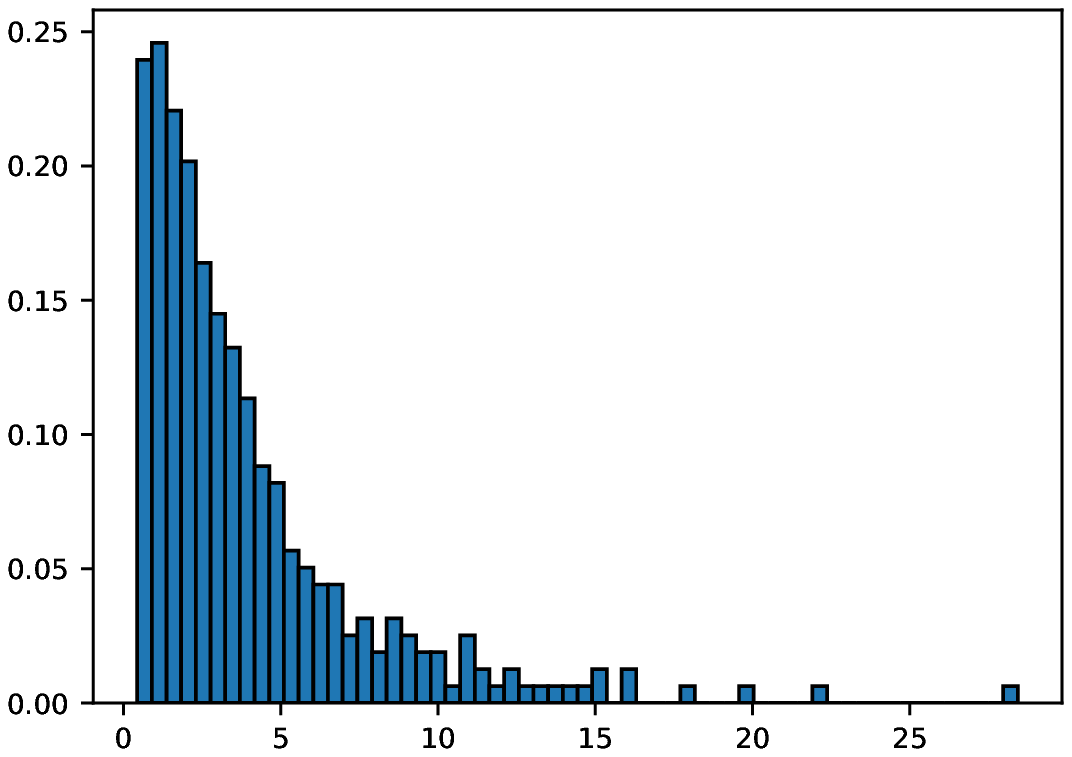}}

\caption{Gaussian setting with equal mean and variance.  (a)  Image of the values (in grayscale) of the entries of the $N \times M$ normalized variance profile $\frac{\Gamma_{N,M}}{M}$ with $N=340$ and $M=510$ (b) Image of the modulus of the entries of the matrix $H_{N,M}'$. (c)   Histogram of the singular values of the matrix $\Gamma_{N,M}$.}
  \label{fig:imgPoisson}
\end{figure}

Therefore, the following  rectangular information plus noise model 
\begin{equation}
H_{N,M}' = X_{N,M} + \frac{\Gamma_{N,M}}{M}, \label{eq:PoissonModel}
\end{equation}
where  $X_{N,M}$ is a rectangular Gaussian matrix with a variance profile $\Gamma_{N,M} = \big(\gamma_{N,M}^2(i,j)\big)_{i,j}$, may be viewed as a prototype for studying low-rank matrix denoising in the presence of Poisson noise (with $\kappa_{i,j} = \gamma_{N,M}^2(i,j)/M$), as considered e.g.\ in  \cite{Zhang18}. We shall refer to model \eqref{eq:PoissonModel} as the {\it Gaussian setting with equal mean and variance}.  In Figure \ref{fig:imgPoisson}, we display the image made of the entries of the normalized variance profile $\frac{\Gamma_{N,M}}{M}$ that is considered in these numerical experiments, as well as the histogram of the singular values of this matrix which is scaled so that it satisfies the normalization condition:
\begin{equation}
\frac{1}{N} \sum_{i=1}^{N} \sum_{j=1}^{M} \frac{\gamma_{N,M}^2(i,j)}{M} = 30. \label{eq:normcondPoisson}
\end{equation}
Now, let us consider the SVD  of the normalized variance profile
$\frac{\Gamma_{N,M}}{M} = U \Theta V^*$. For any $1 \leq k \leq \min(N,M)$, model \eqref{eq:PoissonModel} can  be written as
\begin{equation}
H_{N,M}' =  X_{N,M} + Y_{N,M}^{(k)} + Z_{N,M}^{(k)}, \quad \mbox{with}  \quad Z_{N,M}^{(k)} = U_{N,k} \Theta_k V_{M,k}^*, \label{eq:PoissonModelDecomp}
\end{equation}
where $ \Theta_k$ is $k \times k$ diagonal matrix whose elements are the $k$ largest singular values of $\frac{\Gamma_{N,M}}{M}$ and $ U_{N,k}$ (resp.\ $V_{M,k}$) is the matrix made of the associated left (resp.\ right) singular vectors, and
$$
Y_{N,M}^{(k)} =  \frac{\Gamma_{N,M}}{M} - U_{N,k} \Theta_k V_{M,k}^*.
$$
is the matrix obtained by keeping only the remaining smallest singular values in the SVD of the normalized variance profile.

In Figure \ref{fig:histoPoisson}, we display the histogram of the s.v.d.\ of $H'_{N,M}$ sampled from model \eqref{eq:PoissonModel}. We also report the smooth approximation of the singular value distributions of the random matrices $X_{N,M}$ and $H_{N,M}^{(k)} :=  X_{N,M} + Y_{N,M}^{(k)} $ for $k=1,2,3$ using the  inverse Stieltjes transform \eqref{eq:Stieltjesinv} of $g_{n}^\square(\lambda) = \frac{1}{N+M} \Tr \,G_{n}^\square( \lambda)$. As described previously, for the matrix $X_{N,M}$, the deterministic equivalent $G_{n}^\square( \lambda)$ of its operator-valued Stieltjes transform is obtained by iterating the vector Dyson equation \eqref{eq:itervector}. For the matrix $H_{N,M}^{(k)}$, such a deterministic equivalent  is obtained by iterating the matrix equation \eqref{eq:matrixfixedpoint}.

In Figure \ref{fig:histoPoisson}, we also report, for $k=1,2,3$, the values of the mapping  $\lambda \mapsto \det\big(\beta_{2k}^\square(\lambda)\big)$ and $\lambda \mapsto \det\big(\tilde{\beta}^\square_{2k}(\lambda)\big)$ for $15 \leq \lambda \leq 35$. Again, this illustrates the benefits of using $\tilde{\beta}_{2k}$ instead of $\beta_{2k}$ for outliers detection. Moreover, for all $1 \leq k \leq 3$ one predicts accurately $k$ outliers using the mapping $\lambda \mapsto \det\big(\tilde{\beta}^\square_{2k}(\lambda)\big)$.

\begin{figure}[!t]
\centering%
\subfigure[$k = 1$]{\includegraphics[width=0.32\linewidth]{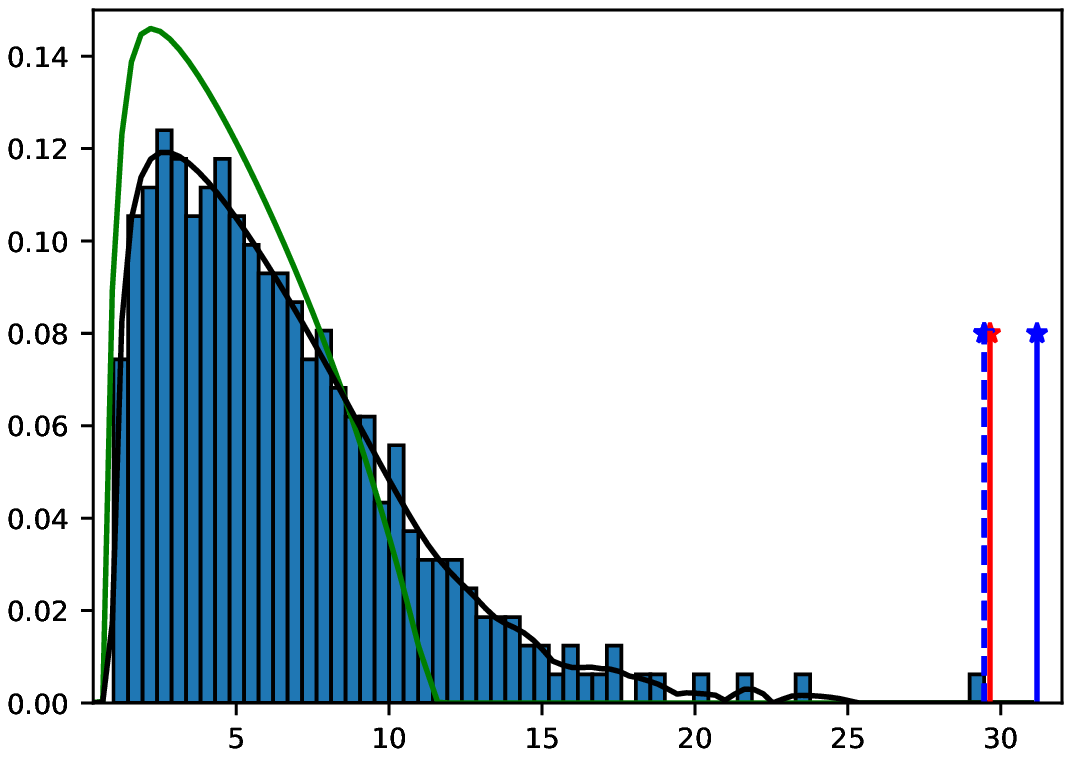}}
\subfigure[$k = 2$]{\includegraphics[width=0.32\linewidth]{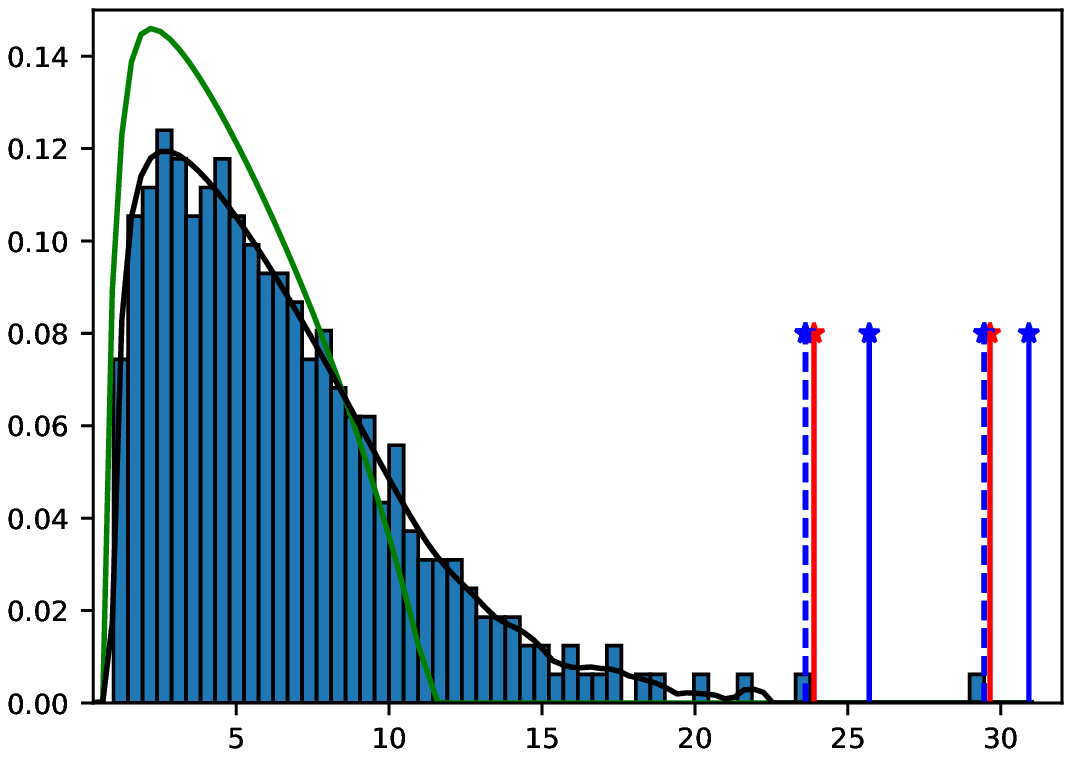}}
\subfigure[$k = 3$]{\includegraphics[width=0.32\linewidth]{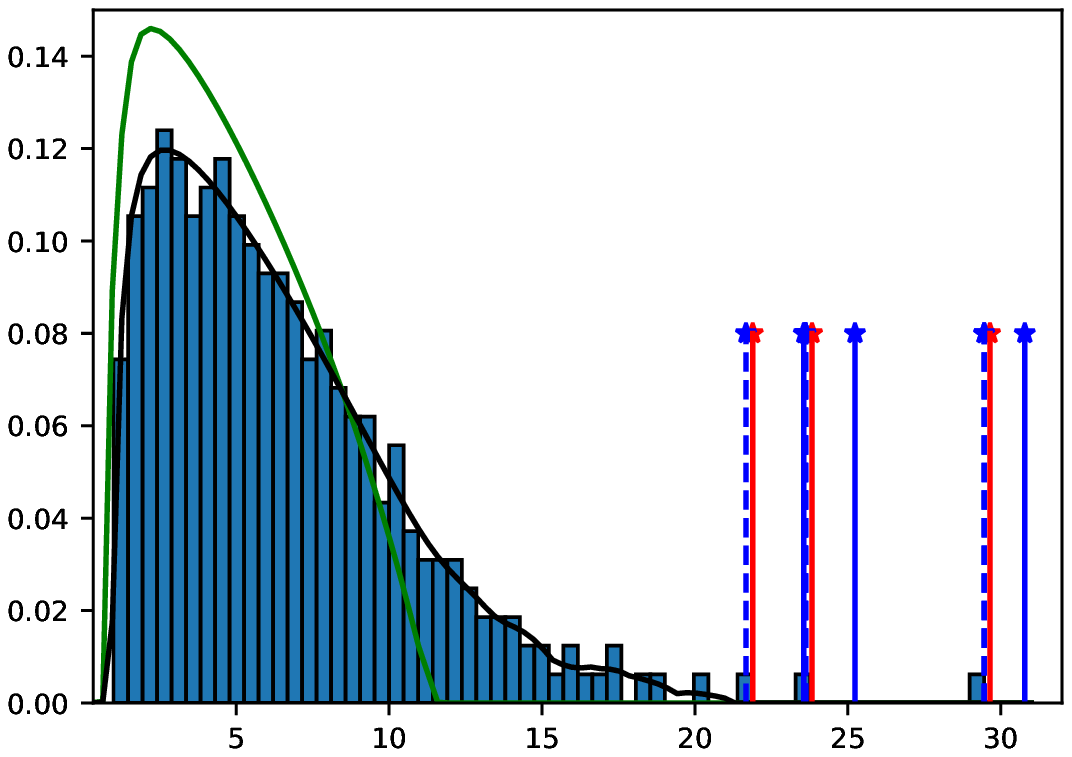}}

\subfigure[$k = 1$]{\includegraphics[width=0.32\linewidth]{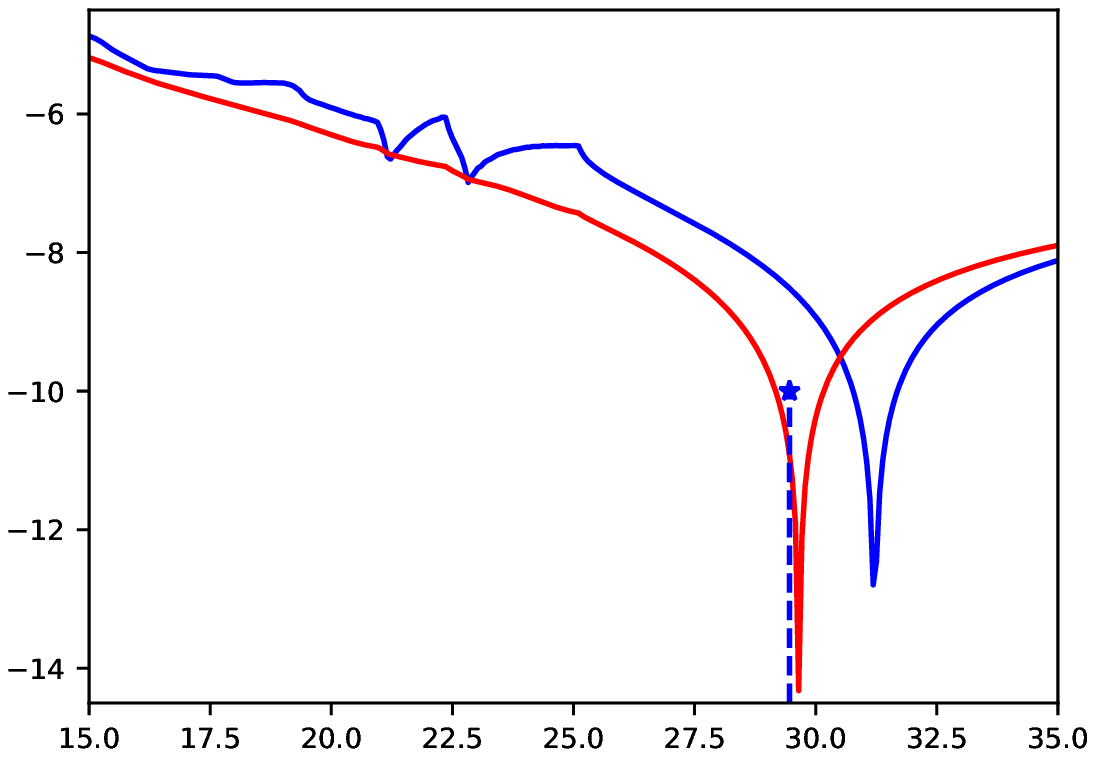}}
\subfigure[$k = 2$]{\includegraphics[width=0.32\linewidth]{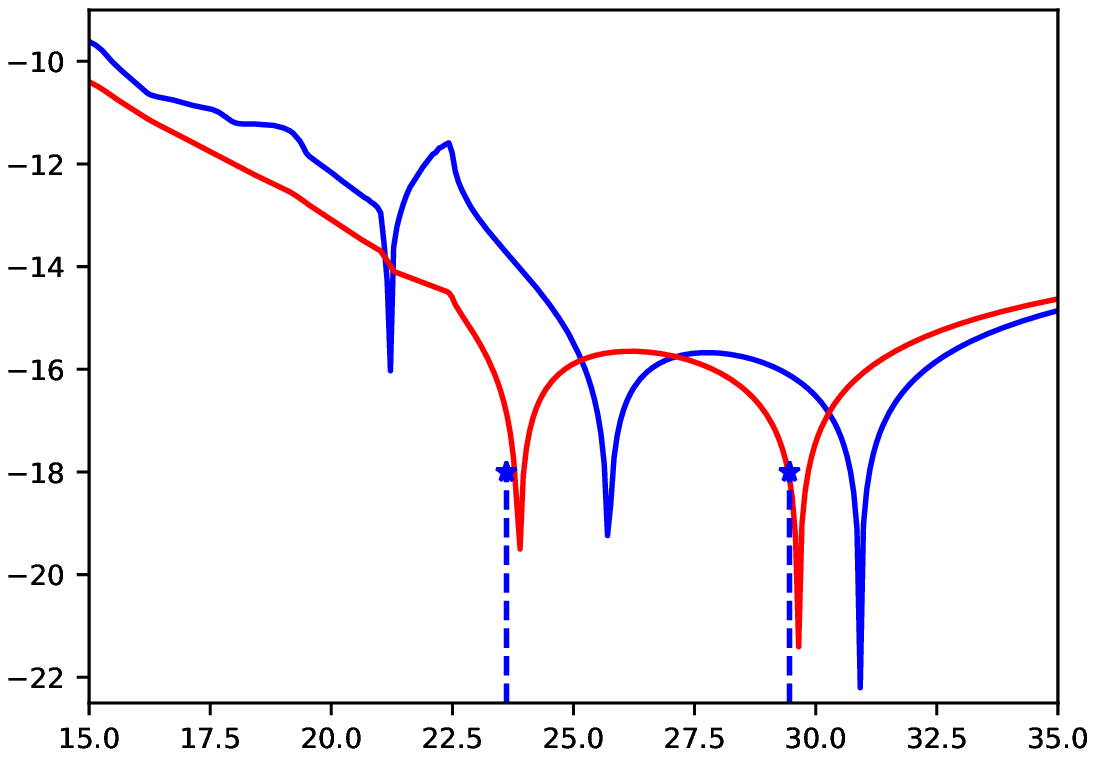}}
\subfigure[$k = 3$]{\includegraphics[width=0.32\linewidth]{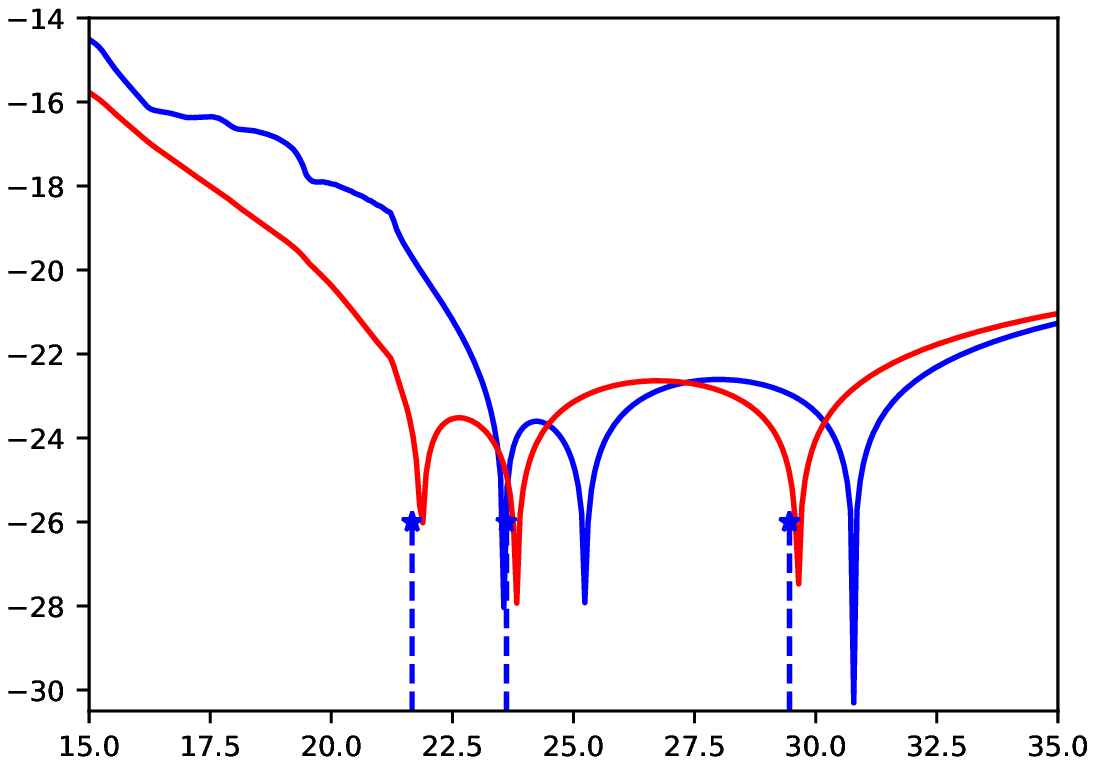}}

\caption{Gaussian setting with equal mean and variance. First row: histogram of the s.v.d.\ of $H'_{N,M} = X_{N,M} + Y_{N,M}^{(k)} + Z_{N,M}^{(k)} $ for $k$ spikes with $k = 1,2,3$. The black curves are the smooth approximations of the s.v.d.\ of $H^{(k)}_{N,M} = X_{N,M} + Y_{N,M}^{(k)} $ for different values of $k$, while the green curve is the smooth approximation of the s.v.d.\ of $X_{N,M}$. The red (resp.\ blue) vertical lines  denote the values which are  the approximation of the locations of outliers given by the zeros of $\tilde{\beta}_{2k}^\square$ (resp.\ $\beta_{2k}^\square)$. Second row: graph of  $\lambda \mapsto \log\big(\det\big(\beta_{2k}^\square(\lambda) \big)\big)$ (blue curves)  and $\lambda \mapsto \log\big(\det\big(\tilde{\beta}_{2k}^\square(\lambda) \big)\big)$ (red curves) for different values of $k$. The blue vertical dashed lines  denote the locations of the true outliers in the s.v.d.\ of $H'_{N,M}$}
  \label{fig:histoPoisson}
\end{figure}

\section{Organization of the proofs} \label{sec:organization}

 We first describe the mains steps to derive the proof of Theorem \ref{MainTh} using the notation of Section \ref{Sec:MainStat}. In Section \ref{Sec:AsympSubProp}, following the method and terminology of Haagerup and Thorbj{\o}rnsen in \cite{HT05}, we start by proving a \emph{Master equality} (see Lemma \ref{Lem:RecallASP}) involving  the expectation of the generalized resolvent $(\Lambda - X_N-Y_N)^{-1}$ that can be decomposed as follows
\eqa\label{Eq:AsympSubProp}
\esp\big[ (\Lambda - X_N-Y_N)^{-1} \big] = \big(\Omega_{X_N,Y_N}(\Lambda)-Y_N \big)^{-1} + F_N(\Lambda),
\qea
with
$$
\Omega_{X_N,Y_N}(\Lambda) =  \Lambda - \mcal R_N\big(  \esp\big[G_{X_N+Y_N}(\Lambda)\big]  \big),
$$
and $F_N(\Lambda) =  \big( \Omega_{H_N}(\Lambda) -Y_N \big)^{-1}  E_N$, where $E_N$ is the matrix of covariance between $(\Lambda - X_N-Y_N)^{-1} $ and $\mcal R_N\big( \esp[G_{X_N+Y_N}(\Lambda) ]\big)$ given explicitly in \eqref{Eq:DefEN}.  We deduce this result from an identity on the resolvent $(\Lambda - H_N)^{-1}$ that is a consequence of the {\it Gaussian integration by part formula} (see Lemma \ref{MasterEq} and Lemma \ref{lem:GaussIPP} below).

Following the heuristic of Section \ref{sec:freeapproach}, in particular \eqref{Eq:VonNeuSub}, the matrix $\Omega_{X_N,Y_N}(\Lambda)$ is a good candidate to approximate an operator-valued subordination function. Applying the operator $\Delta$ on both sides of equality \eqref{Eq:AsympSubProp} provides the following approximate equation for operator-valued Stieltjes transforms:
\begin{equation} \label{Eq:DiagAsympSubProp}
 \esp\big[G_{X_N+Y_N}(\Lambda)\big] = G_{Y_N}\Big(\Lambda - \mcal R_N\big(  \esp\big[G_{X_N+Y_N}(\Lambda)\big]  \big) \Big)  +  \Delta[F_N(\Lambda)].
\end{equation}
Equation \eqref{Eq:DiagAsympSubProp} tells that the expectation of $G_{X_N+Y_N}$ satisfies, up to additive error term, the fixed point equation \eqref{FixPtEq1} that defines its deterministic equivalent. Then, using \emph{Gaussian Poincar\'e inequality} (see Lemma \ref{prop:Poincare}), we obtain an upper bound on $\| E_N\|$ and $\| \Delta(E_N)\|$
 that we call \emph{Master inequalities} (see Lemma \ref{MasterIneq} below), following again the terminology of Haagerup and Thorbj{\o}rnsen \cite{HT05}. 
 It will be finally shown that 

\begin{equation} \label{eq:upperboundDeltaFN}
	\big\| \Delta[F_N(\Lambda)] \big\| \leq 
	 2\gamma_{\max}^3 N^{-1} \times  \| (\Im m \, \Lambda)^{-1}\|^4.
\end{equation}
 and if $Y_N$ is diagonal then 
 
\begin{equation} \label{eq:upperboundDeltaFN2}
 \big\| \Delta[F_N(\Lambda)] \big\|  
 \leq \gamma_{\max}^4 N^{-\frac 3 2} \times  \|(\Im m \, \Lambda)^{-1}\|^5.
\end{equation}

In Section \ref{Sec:FixedPtAnalysis} we prove the existence of the deterministic equivalent $G_{H_N}^\square$, solution of the equation as in \eqref{Eq:DiagAsympSubProp} taken for $\Delta[F_N(\Lambda)]=0$. We also prove  that the  upper bound \eqref{eq:upperboundDeltaFN} for $\big\| \Delta[F_N(\Lambda)] \big\| $  implies a bound for the difference $\big\| \esp\big[G_{H_N}(\Lambda)\big] - G^\square_{H_N}(\Lambda)\big\|$, thanks to an analysis of regularity of the fixed point problem \eqref{FixPtEq1}. More precisely, $G_{H_N}^\square$ is a good approximation of $G_{H_N}$ out of a small strip around the real axe as stated below.

\begin{Lem}\label{StabilityAnalysis} For all $\delta \in (0,1)$ and all $\Lambda$ such that $\Im m \, \Lambda \geq  \gamma_{\mrm{max}}  \left(\frac{2}{N(1-\delta)}\right)^{1/5} \mbb I_N$ we have
	\eq
		\Big\|\esp \big[ G_{H_N}(\Lambda) \big] - G_{H_N}^\square(\Lambda)\Big \| & \leq & \big( 1+ \gamma^2_{\mrm {max}}/\delta  \| (\Im m \, \Lambda)^{-1} \|^2  \big)  C_N,
	\qe
where $C_N$ is the bound in the r.h.s. of \eqref{eq:upperboundDeltaFN}. Moreover, if $Y_N$ is diagonal and $\Im m \, \Lambda \geq  \gamma_{\mrm{max}}  \left(\frac{2}{N(1-\delta)}\right)^{1/5} \mbb I_N$, then the above estimate holds with $C_N$ as in the  r.h.s. of \eqref{eq:upperboundDeltaFN2}.
\end{Lem}

The rest of the proof of Theorem \ref{MainTh} is finally based on the control of the difference between the generalized resolvent $(\Lambda   - H_N)^{-1}$ and the expectation of $G_{H_N}(\Lambda)$. As explained in Section \ref{Sec:resolventAnalysis}, the comparison between the random quantity $G_{H_N}(\Lambda)$ and the deterministic equivalent $G_{H_N}^\square(\Lambda)$ follows from the combination of Lemma \ref{StabilityAnalysis} and \emph{Gaussian concentration inequality} of Lipschitz functions allowing to show that the entries of the generalized resolvent $ (\Lambda - H_N)^{-1}$ are close to their expectation with high probability.  Finally, Section \ref{Sec:resolventAnalysis} ends with a mathematical justification of the convergence of the numerical method used to approximate  the solution of the fixed point equation \eqref{FixPtEq1}.

\section{The approximate subordination property}\label{Sec:AsympSubProp}

\subsection{Notation and preliminaries}\label{Sec:Prelim}

We let $[N]$ be the set of integers between $1$ and $N$. Then, we recall some basic properties of matrices with complex entries that we repeatedly use in the proof. For any matrix $A$ in $\MN$ we denote by $\| A \|$ its operator norm, namely
	\eq
		\| A \| = \underset{ \substack{ x \in \mbb C^N \\ \mrm{s.t. \, } \| x \|_2 = 1}}\sup \| Ax\|_2,
		& \mrm{where \ } \|x\|_2 = \big( \langle x , x  \rangle \big)^{\frac 12}, & \ \langle x , y  \rangle =  \sum_i \bar{x_i} y_i,
	\qe
we $\langle \, \cdot \, \rangle$ stands for the standard scalar product on $\mbb C^N$. We recall if $A$ is a Hermitian matrix then $\|A\|$ is the spectral radius of $A$, and in general it is the square-root of the spectral radius of $AA^*$. In particular it satisfies the $\mcal C^*$-norm condition $\| A\|^2 = \|A^*\|^2 = \|AA^*\|$. We denote by $\Re e\, A$ and $\Im m \, A$ the real and imaginary parts of $A$, which are Hermitian matrices defined by 
	\eq
	\Re e\, A =  \frac 1 {2}(A+A^*) , & \Im m \, A = \frac 1 {2\mbf i}(A-A^*), & \mbf i^2=-1.
	\qe
We write $A\geq 0$ (resp. $A>0$) whenever the matrix $A$ is Hermitian and positive (resp. positive definite), and $A\leq 0$ (resp. $A<0$) if $-A$ satisfies this property. 

\begin{Lem}\label{Lem:NormEstim} Let $A$ in $\MN$  such that $\Im m \, A>0$. Then $A$ is invertible and 
		\begin{equation}\label{ImEst}
			\big \| A^{-1} \big \| \leq \| (\Im m \, A)^{-1}\|.
		\end{equation}
\end{Lem}
\begin{proof}
 We follow the proof of Lemma 3.1 in \cite{HT05}. For any unit vector $x$ in $\mbb C^N$, since $\langle \Re e (A)x,x\rangle $ and $ \langle \Im m(A)x,x\rangle$ are real, we have
	\eq
		\| Ax\|_2& = &\|Ax\|_2 \|x\|_2\geq \big| \langle Ax,x \rangle \big| =  \big|\langle \Re e (A)x,x\rangle + \mbf i  \langle \Im m(A)x,x\rangle \big|\\
		& \geq  & \big|  \langle  \Im m (A)x,x\rangle \big| \geq \|  (\Im m\, A)^{-1}\|^{-1} \| x\|_2 =  \|  (\Im m\, A)^{-1}\|^{-1}.
	\qe
Hence $A$ is injective, so it is invertible. Since $1=\|x\|_2=\| A^{-1}A x\|_2\leq \|A^{-1}\| \times \|Ax\|_2$, we get from the previous lower bound that $\|A^{-1}\|  \leq  \|  (\Im m\, A)^{-1}\|$.
\end{proof}

For a diagonal matrix $\Lambda$, note that $\|\Lambda\| = \underset{i\in [N]} \max |\Lambda(i,i)|$ and $\Lambda\geq 0$ whenever  $ \Lambda(i,i)\geq0$ for any $i=1\etc N$. Recall that we defined $\Delta(A) = \underset{i\in [N]}{\mrm{diag}}\big( A(i,i) \big)$ for any matrix $A$.

\begin{Lem}\label{Lem:DeltaProp} For any $A$ in $\MN$, one has
		\eqa\label{DHEst}
			\big \| \Delta (A)\big \| \leq \| A\|.
		\qea
	Moreover, $A\leq 0$ implies $\Delta(A)\leq 0$ and $A< 0$ implies $\Delta(A)< 0$.
\end{Lem}

\begin{proof} We have $\| \Delta (A)\big \| = \underset{i\in [N]} \max  \big| A(i,i) \big| =\underset{i\in [N]} \max   \big| \langle e_i, A e_i \rangle \big| \leq \|A\|,$ where $(e_i)_{i\in [N]}$ denotes the canonical basis of $\mbb C^N$. Moreover, if $A\leq 0$ then $A=-BB^*$ for some Hermitian matrix $B$. Hence for any $i=1\etc N$,
	$\Delta(A)(i,i) =- \sum_{j=1}^N \big| B(i,j)\big|^2 \leq 0,$
and so we get $\Delta(A)\leq 0$. If moreover $A<0$ then $A$ is invertible so the $i$-th line of $B$ is nonzero for any $i=1\etc N$ and hence $\Delta(A)<0$.
\end{proof}

If $\Lambda$ is an invertible diagonal matrix, note also that 
$
		 \|\Lambda^{-1}\| = \Big( \underset{i\in [N]} \min |\Lambda(i,i)|\ \Big)^{-1}.
$
We use several times the following direct consequence of Lemmas \ref{Lem:NormEstim} and \ref{Lem:DeltaProp}.
  
  \begin{Lem}\label{Eq:PositivityLambdaPrime} For any Hermitian matrix $A$ and any diagonal matrix $\Lambda$ such that $\Im m \, \Lambda>0$, the matrix $(\Lambda - A)$ is invertible and 
	\eqa\label{eq:boundopnorm}
		\big \| (\Lambda - A)^{-1} \big \| \leq \|  (\Im m \, \Lambda)^{-1}\|.
	\qea
Hence the map $G_{A} : \Lambda \mapsto \Delta \big[ (\Lambda - A)^{-1}\big]$ is well defined on $ \mrm D_N(\mbb C^+)$, and it satisfies $\big\|G_A(\Lambda)\big\| \leq \| (\Im m \, \Lambda)^{-1}\|$ and $\Im m \, G_{A}(\Lambda) <0$ for all $\Lambda$. 
\end{Lem}

Recall that we defined $\mcal R_N (G)=  \underset{i\in [N]}{\mrm{diag}} \Big( \sum_{j=1}^N \frac {\gamma^2_{i,j}}N G_{j,j}\Big)$ for any diagonal matrix $G$. Throughout the paper, we denote $\gamma_{\mrm{max}} = \max_{i,j}\gamma_{i,j}$.

\begin{Lem}\label{Lem:RNProp} For any diagonal matrix $G$ such that $\Im m \, G<0$, one has $\mcal R_N(G)\leq 0$. Moreover, for any diagonal matrix $G$, the following operator norm bound holds
	$$\| \mcal R_N( G ) \| \leq \gamma_{\mrm{max}}^2 \| G\|.$$
	\end{Lem}

\begin{proof}
For any $G$ such that $\Im m \, G<0$, since the matrix $\Gamma_N$ has nonnegative entries, we have
	\eq
		  \Im m  \big( \mcal R_N    (G  ) \big) & = & \sum_{i\in [N]} \frac 1{2\i} \Big[ \sum_{j=1}^N \frac{\gamma_{i,j}^2}N G_{j,j}E_{i,i}  -  \big(\sum_{j=1}^N \frac{\gamma_{i,j}^2}N G_{j,j}E_{i,i} \big)^* \Big]\\
		 & = &  \sum_{i\in [N]}   \sum_{j=1}^N \frac{\gamma_{i,j}^2}N  \Im m (G_{j,j})E_{i,i},
	\qe
	where $E_{i,i} = e_{i} \otimes e_{i}$ denotes the matrix will entries equal to zero except the $(i,i)$-th one which is equal to one.
Hence the diagonal entries of $ \mcal R_N    (G  ) $ are indeed nonpositive. Moreover, for any diagonal matrix $G$ we have
$$
		\| \mcal R_N( G ) \| =  \underset{i\in [N]} \max \big| \mcal R_N( G )(i,i)\big| \leq \gamma_{\mrm{max}}^2    \frac 1 N \sum_{j=1}^N  |G(j,j)|.
$$
Since $|G(j,j)| \leq \|G\|$, we obtain the expected inequality.
\end{proof}

\subsection{The Master equality}

The notation  are as in Section \ref{Sec:MainStat}, in particular we denote $H_N=X_N+Y_N$. We start with the following observation.

\begin{Lem}\label{MasterEq} For any $\Lambda \in \mrm D_N(\mbb C^+)$, diagonal matrix with positive imaginary part, recalling that $
G_{H_N}(\Lambda) = \Delta \big[ (\Lambda- H_N)^{-1}\big],
$ we have the equality between the $N \times N$ matrices
\begin{equation}
\esp\big[ X_N (\Lambda- H_N)^{-1} \big] = \esp\Big[  \mcal R_N\big( G_{H_N}(\Lambda) \big)   (\Lambda- H_N)^{-1} \Big]. 
\label{eq:MasterEquality}
\end{equation}
\end{Lem}
Lemma \ref{MasterEq} is a consequence of the well-known Gaussian integration by part formula (see e.g.\ Lemma 3.3 in \cite{HT05}) that we recall below. 

\begin{Lem} \label{lem:GaussIPP} Let $f : \R^q \to \C$ be a continuously differentiable function, and $X_1,\ldots,X_q$ a sequence of independent centered real Gaussian variables with possibly different variances $\var(X_k) = \gamma_k^2$ for $1 \leq k \leq q$. Then, under the conditions that $f$ and its first order derivatives $ \frac{\partial}{\partial x_1}f, \ldots, \frac{\partial}{\partial x_q}f$ are polynomially bounded,  one has that
\begin{equation}
\esp \big[X_{k} f(X_1,\ldots,X_q) \big] = \gamma_k^2 \esp \left[ \frac{\partial}{\partial x_k}f(X_1,\ldots,X_q) \right]. \label{eq:GaussIPP}
\end{equation}
\end{Lem}

\begin{proof}[Proof of Lemma \ref{MasterEq}] Let us write the random matrix $X_N$ in an appropriate orthonormal basis to make its dependency on only real Gaussian variables  more explicit. Let $E_{i,j} = e_{i} \otimes e_{j}$ be the canonical basis of $N \times N$ matrices, and define

\eqa\label{Eq:BaseF}
F_{i,j} = \left\{ \begin{array}{ccc} E_{i,j} & \mrm{if} & i=j\\
						\frac{1}{\sqrt{2}} \left(E_{i,j} + E_{j,i}\right) & \mrm{if} & i>j\\
						\i \frac{1}{\sqrt{2}} \left(E_{i,j} - E_{j,i}\right) & \mrm{if} & i<j \end{array}\right.
\qea
Then, $X_N$ can be decomposed as
\begin{equation} \label{eq:decompZN}
X_N =\sum_{i,j=1}^Nx'_{i,j}F_{i,j},
\end{equation}
where the $x'_{i,j}$ are i.i.d. real Gaussian random variables, centered and such that $x_{i,j}$ has variance $\frac{\gamma_{i,j}^2}N$. This implies that 
\begin{eqnarray*}
\esp\left[ X_N (\Lambda-H_N)^{-1} \right]  & = &   \sum_{i , j=1}^N   F_{i,j} \esp\left[x'_{i,j}  (\Lambda-H_N)^{-1}\right].
\end{eqnarray*}
By the Gaussian integration by part \eqref{eq:GaussIPP}, we have
	$$\esp\left[x'_{i,j}  (\Lambda-H_N)^{-1}\right] = \frac{\gamma^2_{i,j}}{N} \esp \Big[ \at{\frac{d}{d\epsilon}}{\epsilon=0} (\Lambda-H_N-\epsilon F_{i,j})^{-1} \Big].$$
Now, recalling that $H_N=X_N+Y_N$, we can compute
\begin{equation} \label{eq:diff}
\at{\frac{d}{d\epsilon}}{\epsilon=0} (\Lambda-H_N-\epsilon A)^{-1} = (\Lambda-H_N)^{-1} A (\Lambda-H_N)^{-1} 
\end{equation}
for any direction $A \in \C^{N \times N}$, and get the following relation
\begin{eqnarray}\label{Eq:StepIPP}
\esp\left[ X_N (\Lambda-H_N)^{-1} \right]  & = &     \sum_{i , j=1}^N\frac{\gamma^2_{i,j}}N  F_{i,j}  \esp\left[ (\Lambda-H_N)^{-1}  F_{i,j} (\Lambda-H_N)^{-1} \right].
\end{eqnarray}
Note that \eqref{eq:boundopnorm} ensures that the function and its derivatives are bounded, so we can correctly apply \eqref{eq:GaussIPP}. Moreover, since $\gamma_{ij}=\gamma_{ji}$, we have for any matrix $A$ 
	\eqa\label{Eq:Magic}
		  \sum_{i , j=1}^N  \frac{\gamma^2_{i,j}}N  F_{i,j} A F_{i,j} 		 &= &     \sum_{i =1}^N \frac{\gamma^2_{i,i}}N  E_{i,i} A E_{i,i} +    \frac{1}{2 }  \sum_{i > j} \frac{\gamma^2_{i,j}}N  (E_{i,j} +E_{j,i}) A (E_{i,j} +E_{j,i})\nonumber\\
	& & \ \ \ \ -    \frac{1}{2 }  \sum_{i < j} \frac{\gamma^2_{i,j}}N (E_{i,j} -E_{j,i}) A (E_{i,j} -E_{j,i})\nonumber\\
		 &= &     \sum_{i =1}^N \frac{\gamma^2_{i,i}}N E_{i,i} A E_{i,i} +     \sum_{i > j}  \frac{\gamma^2_{i,j}}N   (E_{i,j} A E_{j,i} +E_{j,i} A E_{i,j} )  \nonumber \\
	& =&   \sum_{i ,j=1}^N \frac{\gamma^2_{i,j}}N A_{jj}E_{ii} = \mcal R_N\big( \Delta(A)\big).
	\qea
Combined with \eqref{Eq:StepIPP}, the equality implies the expected result $$
\esp\left[ X_N (\Lambda-H_N)^{-1} \right] = \esp \left[ \mcal R_N\big( \Delta[(\Lambda-H_N)^{-1}]  \big) (\Lambda-H_N)^{-1}  \right].
$$
\end{proof}
Let us now introduce $E_N = E_N(\Lambda)$ defined as the  covariance   between the matrices $ \mcal R_N\big( G_{H_N}(\Lambda) \big) $ and $ (\Lambda- H_N)^{-1}$, namely
	\eqa\label{Eq:DefEN}
		E_N = \esp\Big[  \mcal R_N\big( G_{H_N}(\Lambda) \big)   (\Lambda- H_N)^{-1} \Big]- \esp\Big[  \mcal R_N\big( G_{H_N}(\Lambda) \big)  \Big] \times \esp\Big[  (\Lambda- H_N)^{-1} \Big].
	\qea
We can now state and prove the so-called Master equality introduced in Equation \eqref{Eq:AsympSubProp}.

\begin{Lem}[Master equality]\label{Lem:RecallASP} For any $\Lambda \in \DN^+$, we define the diagonal matrix 
\begin{equation}
\Omega_{H_N}(\Lambda) = \Lambda - \mcal R_N\Big(  \esp\big[G_{H_N}(\Lambda)\big] \Big). \label{eq:Omega}
\end{equation}
Then, we have $\Im m\, \Omega_{H_N}(\Lambda)> \Im m\, \Lambda$, and the following equality holds for all $\Lambda \in \DN^+$:
	\eqa\label{Eq:RecallASP}
		\esp\big[ (\Lambda-X_N-Y_N)^{-1}\big] =  \big( \Omega_{H_N}(\Lambda) -Y_N \big)^{-1} \times \big( \mbb I_N + E_N \big).
	\qea
\end{Lem}

\begin{proof}
Starting with the left hand side $ X_N (\Lambda- H_N)^{-1} $ in \eqref{eq:MasterEquality}, we want to obtain an expression involving only generalized resolvent. Recall that $H_N=X_N+Y_N$. If we where solely considering the random matrix $X_N$ and assume $Y_N=0$, we should write
	$$ X_N (\Lambda- X_N)^{-1} = \big( -(\Lambda - X_N) + \Lambda \big) (\Lambda- X_N)^{-1}  = - \mbb I_N + \Lambda  (\Lambda- H_N)^{-1}.$$
When $Y_N$ is non zero, we first fix a deterministic matrix $\Omega \in D_N(\mbb C^+)$ whose choice is determined later on, and multiplying on the left by $(\Omega-Y_N)^{-1}$ our expression under consideration: we have
	\eq
		\lefteqn{(\Omega- Y_N)^{-1}X_N (\Lambda- H_N)^{-1}}\\
		 & = & (\Omega- Y_N)^{-1} \big[  -(\Lambda-X_N-Y_N) + (\Omega - Y_N) +(\Lambda-\Omega)\big]  (\Lambda- H_N)^{-1}\\
		& = & -(\Omega- Y_N)^{-1}  + (\Lambda- H_N)^{-1} + (\Omega- Y_N)^{-1} (\Lambda-\Omega) (\Lambda- H_N)^{-1}.
	\qe
Moreover, since $Y_N$ is deterministic, the master equality \eqref{eq:MasterEquality} is equivalent to
	\eq
		 \esp\big[   (\Omega- Y_N)^{-1}X_N (\Lambda- H_N)^{-1}    \big] &= &\esp\big[   (\Omega- Y_N)^{-1}   \mcal R_N\big( G_{H_N}(\Lambda) \big)   (\Lambda- H_N)^{-1}  \big].
	\qe
Introducing the map $f:A \mapsto  \esp\big[   (\Omega- Y_N)^{-1} A (\Lambda- H_N)^{-1} \big]$ for any random matrix $A$, we obtain
	\eq
		  \esp\big[(\Lambda-X_N-Y_N)^{-1}\big]  &=& \big(\Omega - Y_N\big)^{-1} + f\Big(  \mcal R_N\big( G_{H_N}(\Lambda) \big) \Big) - f(\Lambda - \Omega).
	\qe
Since $f$ is linear, we have
	\eq
		\lefteqn{f\Big(  \mcal R_N\big( G_{H_N}(\Lambda) \big) \Big) - f(\Lambda - \Omega)}\\
		  & = & f\Big( \esp\big[ \mcal R_N\big( G_{H_N}(\Lambda) \big]  - \Lambda + \Omega\big) \Big) + f\Big(  \mcal R_N\big( G_{H_N}(\Lambda)\big) - \esp\big[ \mcal R_N\big( G_{H_N}(\Lambda)\big) \big]\Big).
	\qe 
We set $\Omega = \Omega_{H_N}(\Lambda) = \Lambda - \mcal R_N\big( \esp\big[G_{H_N}(\Lambda)\big] \big),$ so that the first term in the above equality vanishes. By Lemma \ref{Lem:RNProp}, we have that  $\Im m \big(\mcal R_N\big( G_{H_N}(\Lambda) \big) \big)\leq 0$, so $\Omega_{H_N}$ belongs to $\mrm D_N(\mbb C^+)$ and satisfies $\Im m\, \Omega_{H_N}(\Lambda)> \Im m\, \Lambda$. Hence we obtain the following expression
	\eq
		  \esp\big[(\Lambda-X_N-Y_N)^{-1}\big]  &=& \big(\Omega_{H_N}(\Lambda) - Y_N\big)^{-1}  +  f\Big(  \mcal R_N\big( G_{H_N}(\Lambda)\big) - \esp\big[ \mcal R_N\big( G_{H_N}(\Lambda)\big) \big]\Big).
	\qe
Since $Y_N$ is deterministic, the above identity completes the proof of Lemma \ref{Lem:RecallASP}.
\end{proof}

\subsection{The Master inequality} 

\subsubsection{Statement and use of Poincar\'e inequality}

We prove the following estimates on $E_N=E_N(\Lambda)$ defined in \eqref{Eq:DefEN} as the covariance   between the matrices $ \mcal R_N\big( G_{H_N}(\Lambda) \big) $ and $ (\Lambda- H_N)^{-1}$

\begin{Lem}[Master inequality]\label{MasterIneq}  Recall that we denote $\gamma_{\mrm{max}} = \max_{i,j}\gamma_{i,j}$. For any $\Lambda$ belonging to $\mrm D_N(\mbb C^+)$, we have 
	\eqa
	\big\|  E_N \big\| & \leq &  2\gamma_{\max}^3 N^{-1} \times  \|(\Im m \, \Lambda)^{-1}\|^3,\label{eq:MasterIneq}\\
	\big\|  \Delta \big[ E_N \big] \big\| &  \leq &  \gamma_{\max}^4 N^{-\frac 3 2} \times  \|(\Im m \, \Lambda)^{-1}\|^4. \label{eq:MasterIneqBis}
	\qea
\end{Lem}

%

By Lemma \ref{Lem:RecallASP} one has that $\|\Im m \,   \Omega^{-1}_{H_N}(\Lambda) \|\leq \| (\Im m \, \Lambda)^{-1}\|$. Hence as announced in the presentation of the proof of Section \ref{sec:organization}, Lemma \ref{MasterIneq} implies that the operator-valued subordination property \eqref{Eq:DiagAsympSubProp} approximatively holds up to an error term $\Delta[F_N(\Lambda)] = \Delta \big[ (\Omega_{H_N}-Y_N)^{-1}E_N \big]$ satisfying
$$
\big\| \Delta[F_N(\Lambda)] \big\| \leq \big\|(\Omega_{H_N}-Y_N)^{-1}E_N \big\| \leq  2\gamma_{\max}^3 N^{-1} \times  \| (\Im m \, \Lambda)^{-1}\|^4.
$$
Moreover, if $Y_N$ is diagonal then so is $(\Omega_{H_N}-Y_N)^{-1}$ and we have
$$
\big\| \Delta[F_N(\Lambda)] \big\|  =  \big\|(\Omega_{H_N}-Y_N)^{-1}\Delta\big[E_N\big] \big\| \leq \gamma_{\max}^4 N^{-\frac 3 2} \times  \|(\Im m \, \Lambda)^{-1}\|^5.
$$

The rest of the section is devoted to the proof of Lemma \ref{MasterIneq}. To abbreviate the notation, we define the random matrices (and functions of the diagonal parameter $\Lambda$)
\eq
		\begin{array}{ccccc}
		A_N & =   (\Lambda - H_N )^{-1},& \ \overset \circ A_N &= A_N - \esp[ A_N],\\
		D_N &=   \mcal R_N\big(\Delta( A_N ) \big),& \ \overset \circ D_N &= D_N - \esp[ D_N],
		\end{array}
	\qe
so that $ E_N= \esp\big[  \overset \circ D_N \overset \circ A_N \big]  $. 
We first use the Cauchy-Schwarz inequality for the norm
	\eqa \label{CSnorm} \big\|  \esp\big[\overset \circ D_N\overset \circ A\big] \big\| \leq \big\|  \esp\big[ \overset \circ D_N\overset \circ D_N^*\big] \big\|^{\frac 12}\big\|  \esp\big[ \overset \circ A_N^*\overset \circ A_N\big] \big\|^{\frac 12}.
	\qea
We recall the proof of \eqref{CSnorm} from  {\cite[Lemma 3.1 (a1)]{Jocic}}. By the singular value decomposition, $\|  \esp [\overset \circ D_N\overset \circ A_N]\| $ is the supremum of $\big| \big\langle  \esp [\overset \circ D_N\overset \circ A_N]x,y \big\rangle\big|$ over all unit vectors $x$ and $y$ in $\mbb C^N$. Moreover, we have
	\eq
		 \Big| \big\langle  \esp [\overset \circ D_N\overset \circ A_N]x,y \big\rangle\Big| & \leq & \esp\big[ \big| \langle \overset \circ D_N\overset \circ A_N x, y \rangle \big| \big] \leq \esp\big[\|\overset \circ D_N^*y \|_2  \times \| \overset \circ A_N x\|_2 \big]\\
		 & \leq & \big( \esp[ \| \overset \circ D_N^*y \|_2^2 ]  \times \esp[ \| \overset \circ A_Nx \|_2^2 ] \big)^{\frac 12}\\
		 & = &  \big(  \langle \esp[ \overset \circ D_N\overset \circ D_N^* ] y, y \rangle \times \langle \esp[  \overset \circ A_N^*\overset \circ A_N] x, x \rangle\big)^{\frac 12}\\
		 & \leq &  \big\|  \esp\big[ \overset \circ D_N\overset \circ D_N^*\big] \big\|^{\frac 12}\big\|  \esp\big[ \overset \circ A_N^*\overset \circ A_N\big] \big\|^{\frac 12}.
	\qe
Hence the inequality \eqref{CSnorm}. We hence deduce a first estimate
	\eqa		
	\big\|  E_N  \big\| &\leq&   \underset{k\in [N]}\sup \Big( \var D_N(k,k) \Big)^{\frac 12} \times \big\| \esp\big[ \overset \circ A_N^*\overset \circ A_N \big]\big\| ^{\frac 12}\label{Eq:CaseG}
	\qea
		 For the diagonal of $E_N$ we shall use the classical Cauchy-Schwarz inequality:
	\eqa
	\nonumber
		\nonumber \big\|  \Delta\big[  E_N \big] \big\| & = & \big\|  \overset \circ D_N \Delta\big[ \overset \circ A_N  \big] \big\|  =   \underset{k\in [N]}\max \Big(\Big| \esp\big[   \overset\circ D_N(k,k) \times  \overset \circ A_N(k,k) \Big| \Big) \\
		&\leq&   \underset{k\in [N]}\max  \bigg( \var\big( D_N(k,k)\big)  \var\big(  A_N(k,k) \big)		\bigg)^{\frac 12}.\label{Eq:CaseD}
	\qea

We now state the Gaussian Poincar\'e inequality (see e.g.\ Proposition 4.1 in \cite{HT05}), which allows use to estimate the variances $\var\big( A_N(k,k)\big) $ and $ \var\big(  D_N(k,k) \big)$ that appear in Inequalities \eqref{Eq:CaseG} and \eqref{Eq:CaseD}. 

\begin{Prop}[Gaussian Poincar\'e inequality] \label{prop:Poincare}
Let $f : \R^q \to \C$ be a continuously differentiable function, and $X_1,\ldots,X_q$ a sequence of independent centered real Gaussian variables with possibly different variances $\var(X_k) = \gamma_k^2$ for $1 \leq k \leq q$. Then, under the condition that $f$ and its first order derivatives are  polynomially bounded, one has that
$$
\var\left( f(X_1,\ldots,X_q) \right) \leq \esp\left( \| \Gamma^{1/2} \nabla f (X_1,\ldots,X_q)\|^2_2 \right)
$$
where $\Gamma = \diag(\gamma_1^2,\ldots,\gamma_q^2)$,  $\nabla f$ is the gradient of $f$, and $\| \cdot \|_2$ is the standard Euclidean norm of a vector with complex entries.
\end{Prop}

We write the matrices $A_N$ and $D_N$ as functions of the independent real Gaussian random variables defined in \eqref{eq:decompZN}, namely: recalling $H_N =X_N+Y_N$, we define for any real matrix $M=(m_{i,j})_{i,j}$ the diagonal matrices
	\eq
		 f_{1}(M) & = &  \Delta\Big[ \big(\Lambda - \sum_{i,j}m_{i,j} F_{i,j} - Y_N \big)^{-1}\Big] ,\\
		 f_{2}(M)& = & \mcal R_N\Big(\Delta\Big[ \Big(\Lambda - \sum_{i,j}m_{i,j} F_{i,j} - Y_N \Big)^{-1} \Big]\Big),
	\qe
where $(F_{i,j})_{i,j}$ is the basis of Hermitian matrices defined by \eqref{Eq:BaseF}, so that we have $ \Delta[ A_N]  = f_{1}\big((x'_{i,j})_{i,j}\big)$ and $D_N  = f_{2}\big((x'_{i,j})_{i,j}\big)$ for the coordinates $(x'_{i,j})_{i,j}$ of $X_N$ in the basis $(F_{i,j})_{i,j}$, see \eqref{eq:decompZN}. By the same computation as for the derivative in \eqref{eq:diff}, and by the linearity of the maps $\mcal R_N$ and $\Delta$, we obtain
	\eq
		\frac{\partial}{\partial x'_{i,j}} f_{1}\big((x'_{i,j})_{i,j}\big)   & = & \Delta\big[(\Lambda - H_N )^{-1}  F_{i,j} (\Lambda - H_N )^{-1}    \big]  = \Delta[A_NF_{i,j}A_N],\\
		\frac{\partial}{\partial x'_{i,j}} f_{2}\big((x'_{i,j})_{i,j}\big)   & = & \mcal R_N\Big(\Delta\big[  (\Lambda - H_N )^{-1}  F_{i,j} (\Lambda - H_N )^{-1}  \big] \Big)  = \mcal R_N\Big(\Delta\big[A_NF_{i,j}A_N) \big] \Big).
	\qe
For any $\ell \in [N]$, we apply Proposition \ref{prop:Poincare} (Gaussian Poincar\'e inequality) to the $\ell$-th diagonal entry of $A_N$ and of $D_N$, and hence get 
	\eqa
		\var\big(  A_N(\ell,\ell) \big)  & \leq & 
		\esp \bigg[ \sum_{i,j}  \frac{\gamma_{i,j}^2}N  \Big|  \big[ A_N F_{i,j} A_N  \big](\ell,\ell)\Big|^2\bigg]\nonumber\\
		  & \leq &  \frac{ \gamma_{\mrm{max}}^2}N  \esp \bigg[ \sum_{i,j}  \Big|  \big[ A_N F_{i,j} A_N  \big](\ell,\ell)\Big|^2\bigg],\label{Eq:PoincareEstim2}	\\
		 \var\big( D_N(\ell,\ell)\big) & \leq & 
		\esp \bigg[\sum_{i,j} \frac{\gamma_{i,j}^2}N \Big| \mcal R_N\Big(\Delta\big[  A_N  F_{i,j} A_N  \big] \Big)(\ell,\ell) \Big|^2\bigg]\nonumber\\
		& \leq & \frac{ \gamma_{\mrm{max}}^2}N
		\esp \bigg[\sum_{i,j}   \Big| \mcal R_N\Big(\Delta\big[  A_N  F_{i,j} A_N  \big] \Big)(\ell,\ell) \Big|^2\bigg]. \label{Eq:PoincareEstim1}
	\qea

\subsubsection{Estimation of the terms given by Poincar\'e inequality}

Hence  to obtain the required estimates for \eqref{Eq:CaseG} and \eqref{Eq:CaseD}, one should find  upper bounds for \eqref{Eq:PoincareEstim2} and \eqref{Eq:PoincareEstim1}. This is the purpose of this section, where we derive below the two estimates \eqref{Eq:CaseGsol2} and \eqref{Eq:CaseGsol3}, which allow to complete the proof of Lemma \ref{MasterIneq}.

We first write in a different way  the term $\sum_{i,j}    \big|  \big[ A_N F_{i,j} A_N  \big](\ell,\ell)\big|^2$. Recall that $(e_k)_{k\in [N]}$ denotes the canonical basis of $\mbb C^N$. Recalling that the entry $(\ell,\ell)$ of a matrix $M$ is equal to $e_\ell^* M e_\ell$, that the elementary matrix $E_{\ell,k}$ equals $e_\ell e_k^*$. Since $F_{ij}$ is a Hermitian matrix, we have
	\eq \Big|  \big[ A_N F_{i,j} A_N  \big](\ell,\ell)\Big|^2 & = &    \big[ A_N F_{i,j} A_N  \big](\ell,\ell) \times  \big[A_N^* F_{i,j} A_N^* \big](\ell,\ell)\\
	& = & e_\ell^* A_N F_{i,j} A_Ne_\ell \times  e_\ell^*   A_N^* F_{i,j} A_N^* e_\ell,
\qe
Note also that \eqref{Eq:Magic} implies the following equality, valid for any matrix $M$:
 	\eqa\label{Eq:Magic2}
	\sum_{i ,j=1}^N   F_{ij} MF_{ij}  =  \sum_{i ,j=1}^N  M_{jj}E_{ii} =  \sum_{i ,j=1}^N   e_i e_j^* M e_j e_i^*.
	\qea
Using \eqref{Eq:Magic2} for $M = A_Ne_\ell e_\ell^*   A_N^* $, we get
	\eq
	\lefteqn{ \sum_{i,j}    \Big|  \big[ A_N F_{i,j} A_N  \big](\ell,\ell)\Big|^2}\\
	 & = & e_\ell^* A_N \sum_{i,j}   \big (F_{i,j} M  F_{i,j} \big)    A_N^* e_\ell =   e_\ell^* A_N \sum_{i,j}  \big(    e_i e_j^* M e_j e_i^*  \big)A_N^* e_\ell\\
	& = & \sum_{i,j}     e_\ell^* A_N   e_i e_j^* A_Ne_\ell e_\ell^*   A_N^*  e_j e_i^* A_N^* e_\ell =   \sum_{i,j}    |A_N(\ell,i) |^2  |A_N(j,\ell)|^2\\
	& = &  \big[A_NA_N^*] (\ell,\ell)^2.
	\qe
Hence the application of Poincar\'e inequality \eqref{Eq:PoincareEstim2} yields 
	\eqa
		 \var\big(  A_N(\ell,\ell) \big)    & \leq  &   \gamma_{\mrm{max}}^2 N^{-1} \esp  \Big[  \big[A_NA_N^*] (\ell,\ell)^2\Big] \nonumber\\
		 & \leq &   \gamma_{\mrm{max}}^2 N^{-1} \esp  \big[    \|A_N\|^4 \big] \leq  \gamma_{\mrm{max}}^2 N^{-1}   \| (\Im m \, {\Lambda})^{-1}\|^4.\label{Eq:CaseGsol2}
	\qea
Similarly, we re-write the term $\sum_{i,j}   \big| \mcal R_N\big(\Delta\big[  A_N  F_{i,j} A_N  \big] \big)(\ell,\ell) \big|^2$. Since for any matrix $M$,
	$$\mcal R_N\big( \Delta M \big)(\ell,\ell) = \sum_{\ell'} \frac{\gamma^2_{\ell\ell'}}N   e_{\ell'}^* M e_{\ell'},$$
we obtain
	\eq
		\lefteqn{\Big| \mcal R_N\Big(\Delta\big[  A_N  F_{i,j} A_N  \big] \Big)(\ell,\ell) \Big|^2}\\
		& = &    \mcal R_N\Big(\Delta\big[  A_N  F_{i,j} A_N  \big] \Big)(\ell,\ell)  \mcal R_N\Big(\Delta\big[  A_N^*  F_{i,j} A_N^*  \big] \Big)(\ell,\ell) \\
		& = & \sum_{\ell', \ell''}   \gamma_{\ell,\ell'}^2   \gamma_{\ell,\ell''}^2 N^{-2}   e^*_{\ell'}     A_N   \Big( F_{i,j} A_N e_{\ell'}  e^*_{\ell''}     A_N^*  F_{i,j}\Big) A_N^* e_{\ell''}.
\qe
Using \eqref{Eq:Magic2} for $M =A_N e_{\ell'}  e^*_{\ell''}     A_N^*  $, we get the expressions and estimates
	\eq
		\lefteqn{\sum_{i,j}   \Big| \mcal R_N\Big(\Delta\big[  A_N  F_{i,j} A_N  \big] \Big)(\ell,\ell) \Big|^2}\\
		& = &  \sum_{i,j, \ell', \ell''}    \gamma_{\ell,\ell'}^2 \gamma_{\ell,\ell''}^2 N^{-2}   e^*_{\ell'}     A_N   e_i e_j^* A_N e_{\ell'}  e^*_{\ell''}     A_N^*  e_j e_i^*A_N^* e_{\ell''}\\
		& = & \sum_{i,j, \ell', \ell''}    \gamma_{\ell,\ell'}^2  \gamma_{\ell,\ell''}^2 N^{-2}  A_N(\ell',i) A_N(j,\ell')    A_N^*(\ell'',j) A_N^*(i,\ell'')\\
		& \leq & \gamma_{\mrm{max}}^4 N^{-2} \sum_{\ell', \ell''}   \Big| \sum_{i,j}  A_N(\ell',i) A_N(j,\ell')    A_N^*(\ell'',j) A_N^*(i,\ell'') \Big|\\
		& = & \gamma_{\mrm{max}}^4 N^{-2}\sum_{\ell', \ell''}   \Big|  (A_NA_N^*)(\ell',\ell'') \times (A^*A)(\ell'',\ell') \Big|.
	\qe
The Cauchy-Schwarz inequality for $   \sum_{\ell', \ell''} $ implies 
	\eq
		\lefteqn{\sum_{i,j}   \Big| \mcal R_N\Big(\Delta\big[  A_N  F_{i,j} A_N  \big] \Big)(\ell,\ell) \Big|^2}\\
		& \leq &  \gamma_{\mrm{max}}^4 N^{-2} \times \Big(  \sum_{\ell', \ell''}   \Big|  (A_NA_N^*)(\ell',\ell'')\Big|^2  \Big)^{\frac 12} \times \Big(    \sum_{\ell', \ell''} \Big|  (A^*A)(\ell'',\ell') \Big|^2 \Big)^{\frac 12}\\
		& = &   \gamma_{\mrm{max}}^4 N^{-2} \times  \Tr\big[(A_NA_N^*)^2\big]  \leq   \gamma_{\mrm{max}}^4 N^{-1}    \|(\Im m \, {\Lambda})^{-1}\|^4.
	\qe
Hence with \eqref{Eq:PoincareEstim1}, the above inequality gives
\eqa
		  \var\big( D_N(\ell,\ell)\big)  & \leq & \gamma_{\mrm{max}}^2 {N^{-1}}  \times \gamma_{\mrm{max}}^4 {N^{-1}}        \| (\Im m \, {\Lambda})^{-1}\|\\
		  & = & \gamma_{\mrm{max}}^6 {N^{-2}}       \| (\Im m \, {\Lambda})^{-1}\|^4\label{Eq:CaseGsol3}.
	\qea
	
Unfortunately, the Poincaré inequality does not conclude to an interesting estimate for $\big\| \esp\big[ \overset \circ A_N^*\overset \circ A_N \big]\big\|$ that will be roughly bounded by $4\| (\Im m \, {\Lambda})^{-1}\|^2$. Inserting the estimates \eqref{Eq:CaseGsol3}  in \eqref{Eq:CaseG} give
	\eq
		 \big\|  E_N   \big\|   &\leq &   \Big(  \gamma_{\mrm{max}}^6 {N^{-2}}      \|(\Im m \, {\Lambda})^{-1}\|^4 \times  4\| (\Im m \, {\Lambda})^{-1}\|^2\Big)^{\frac 12} \\
		 & =&   2\gamma_{\mrm{max}}^3{N^{-1}} \|(\Im m \, {\Lambda})\|^3 .
	\qe
Moreover, \eqref{Eq:CaseGsol2} and  \eqref{Eq:CaseG} implies
	\eq
		 \Big\| \Delta \big[ E_N  \big] \Big\|  &\leq &\Big(  \gamma_{\mrm{max}}^6 {N^{-2}}       \|(\Im m \, {\Lambda})^{-1}\|^4 \times  \gamma_{\mrm{max}}^2N^{-1}   \|(\Im m \, {\Lambda})^{-1}\|^4 \Big)^{\frac 12} \\
		 & =&    \gamma_{\mrm{max}}^4{N^{-\frac 32}} \|(\Im m \, {\Lambda})^{-1}\|^4,
	\qe	
which proves Inequality \eqref{eq:MasterIneq} and complete the proof of Lemma \ref{MasterIneq}.

\section{The fixed point equation}\label{Sec:FixedPtAnalysis}

In this section, we prove the existence and uniqueness of the deterministic equivalent $G_{H_N}^\square(\Lambda)$, and we derive an estimate for the difference $G_{H_N}^\square - \esp\big[ G_{H_N}\big]$.

\subsection{Fixed point formulation}
Recall that for any $\Lambda$ in $\mrm D_N(\mbb C^+)$, we denote $G_{Y_N}(\Lambda)= \Delta\big[ (\Lambda - Y_N)^{-1}\big]$. For any $\Lambda$ in $\mrm D_N(\mbb C^+)$ we consider the function 
\begin{equation} \label{eq:psiLambda}
\begin{array}{cccc}
		\psi_{ \Lambda} : &\mrm D_N(\mbb C)^- & \to & \mrm D_N(\mbb C)^-\\
			& G & \mapsto&G_{Y_N}  \big( \Lambda  -\mcal R_N(G)  \big) .
	\end{array}
\end{equation}
So $G_{H_N}^\square(\Lambda)$ is solution of Equation \eqref{FixPtEq1} of Theorem \ref{MainTh} if and only if $G_{H_N}^\square (  \Lambda )$ is a fixed point of $\psi_{{\Lambda}}$ for any $\Lambda$. Note that by Lemma \ref{Lem:RNProp}, when $G\in \mrm D_N(\mbb C)^- $ then $\mcal R_N(G) \leq 0$ and so $\Lambda  -\mcal R_N(G) \in \mrm D_N(\mbb C^+)$. Hence we can correctly evaluate the function $G_{Y_N}$ in this diagonal matrix and the expression defining $\psi_{\Lambda}$ makes sense. Moreover, by the last statement of Lemma \ref{Eq:PositivityLambdaPrime} we indeed have $\psi_{ \Lambda}( G) \in \mrm D_N(\mbb C)^-$. We shall use the next statement later. 

\begin{Lem}\label{Lem:LipsFP} For any $\Lambda \in \mrm D_N (\mbb C)^+$ the function $\psi_{\Lambda}$ is bounded and  Lipschitz for the operator norm: for any $G,G'\in \mrm D_N(\mbb C)^- $,
	\eq
		\| \psi_{\Lambda}(G)  \| & \leq  & \| (\Im m \, \Lambda)^{-1}\|,\\
		\|\psi_{\Lambda}(G) - \psi_{\Lambda}(G')\| & \leq  & \gamma^2_{\mrm{max}}  \big\|   (\Im m \, \Lambda)^{-1}     \big\|^2 \times \| G-G'\|.
	\qe
\end{Lem}

\begin{proof} We have by Lemmas \ref{Eq:PositivityLambdaPrime} and \ref{Lem:RNProp}
	$$\|\psi_{\Lambda}(G) \|= \Big\| G_{Y_N}  \big( \Lambda  -\mcal R_N(G)  \big) \Big \| \leq \Big\| \Im m \big( \Lambda  -\mcal R_N(G)  \big)^{-1} \Big\| \leq \| (\Im m \, \Lambda)^{-1}\|$$ 
	Moreover, by Lemma \ref{Lem:DeltaProp} and the fact that $\| \, \cdot \, \|$ is an algebra norm, for any $G, G'$ in  $D_N (\mbb C)^-$ we have	 
	\eq
		\lefteqn{\|\psi_{\Lambda}(G) - \psi_{\Lambda}(G')\|	}\\
		&=&	  \Big\| \Delta\Big(\big( \Lambda  - \mcal R_N( G)-Y_N \big)^{ -1}   \Big( \mcal R_N( G-G')\Big)    \big( \Lambda  - \mcal R_N(  G')-Y_N\big)^{ -1} \Big)\Big\|\\
			&\leq&   \big\| \big( \Lambda  - \mcal R_N( G)-Y_N \big)^{ -1} \big\|  \times \big\|    \mcal R_N( G-G')\big\| \times \big\|     \big( \Lambda  - \mcal R_N(  G')-Y_N\big)^{ -1}\big\|
	\qe
By  \eqref{eq:boundopnorm} we have $\big\| \big( \Lambda  - \mcal R_N( G)-Y_N \big)^{ -1} \big\| \leq \big\| \Im m \big( \Lambda  - \mcal R_N( G) \big)^{-1}\big\|$. But by Lemma \ref{Lem:RNProp} we have $\Im m \big(\Lambda  - \mcal R_N( G) \big) \geq \Im m \, \Lambda$ and so $\big\| \big( \Lambda  - \mcal R_N( G)-Y_N \big)^{ -1} \big\| \leq \| (\Im m \, \Lambda)^{-1}\|$. The same inequality holds for $G'$ instead of $G$. These inequalities together with the estimate $ \big\|    \mcal R_N( G-G')\big\|  \leq  \gamma^2_{\mrm{max}} \| G - G'\|$ of Lemma \ref{Lem:RNProp} yield
	\eq
		\|\psi_{\Lambda}(G) - \psi_{\Lambda}(G')\|	  
			& \leq & \gamma^2_{\mrm{max}}  \big\|  (\Im m \, \Lambda)^{-1}     \big\|^2 \times \| G-G'\|.
	\qe
\end{proof}

By Banach fixed-point theorem and Lemma \ref{Lem:LipsFP}, we get the existence and uniqueness a priori of the deterministic equivalent on a region on $\DN^+$. 

\begin{Cor}\label{Cor:Apriori} For any $\Lambda$ such that $\Im m \, \Lambda > \gamma_{max}\mbb I_N$ there exists a unique deterministic diagonal matrix $G_{H_N}^\square(\Lambda) \in \DN^{-}$ such that $G_{H_N}^\square(\Lambda) = \psi_{\Lambda}\big(G_{H_N}^\square(\Lambda)\big)$. 
\end{Cor}

\subsection{The deterministic equivalent}

\subsubsection{Setting of the problem} 
We want to extend the previous corollary for any $\Lambda$ in $\DN^+$. The difficulty is that $\psi_\Lambda$ is not contractive in general. Fortunately, it will be enough in our problem to consider uniqueness in the class of analytic function in several variables. Recall from \cite{Sch05}  that for any open set $\Omega$ of $\DN$, we say that a function $G : \Omega \to \MN$ is analytic on $\Omega$ whenever for any $k,\ell=1\etc N$ the function 
	$$(\lambda_1\etc \lambda_N) \mapsto G\big(\mrm{diag}(\lambda_1 \etc \lambda_N) \big)(k,\ell)$$
are analytic in each variable $\lambda_i$.  

\begin{Lem}\label{Lem:UniquAnalytic} There exists a unique deterministic analytic map $G_{H_N}^{\square}: \DN^+ \to \DN^-$ such that $G_{H_N}^\square(\Lambda) = \psi_{\Lambda}\big(G_{H_N}^\square(\Lambda)\big)$ for any $\Lambda \in \DN^+$. Moreover, for any $\Lambda, \Lambda'\in \DN^+$,
	\eqa\label{Eq:EstimDeterEq}
		 \big\|  G_{H_N}^\square(\Lambda) - G_{H_N}^\square(\Lambda')  \big\| \leq \| (\Im m \, \Lambda)^{-1}\|  \| (\Im m \, {\Lambda'})^{-1}\| \times \| \Lambda - \Lambda'\|.
	\qea
\end{Lem}

The remainder of the subsection is devoted to the proof of Lemma \ref{Lem:UniquAnalytic}. By \cite[Conclusion 1.2.1.2]{Sch05}, the analytic continuation principle holds for analytic maps in several variables. Hence, by Corollary \ref{Cor:Apriori} we know that there exists at most one analytic map $G$ solution of the fixed point problem on $\DN^+$, since all solutions must coincide in $\{ \Lambda \in \DN \; : \; \Im m \, \Lambda > \gamma_{max}\mbb I_N\}$. In Section \ref{sec:large} we prove the existence of such function $G_{H_N}^{\square}$ and that it satisfies the estimate \eqref{Eq:EstimDeterEq}. Then, Section \ref{sec:analiticity} will be dedicated to the proof of analyticity.

\subsubsection{Large random matrix model}\label{sec:large}

Given $N$, $\Gamma_N$ and $Y_N$ fixed we consider an intermediate sequence of Hermitian random matrices $H_{N,M}$ of size $NM$ by $NM$. Seeing a generic element $A$ of $\mrm M_{NM}(\mbb C)$ as an element of $\mrm M_{N}(\mbb C)\otimes \mrm M_{M}(\mbb C)$, we denote
	$$A = \big(A_{i,i',j,j'}\big)_{\substack{ i,j \in [N] \\ i',j' \in [M]}} = \sum_{\substack{ i,j \in [N] \\ i',j' \in [M]}} A_{i,i',j,j'} E_{i,j}\otimes E_{i',j'}.$$
Then, we consider two new random and deterministic matrices.
\begin{description}
	\item[-]  Let $X_{N,M}$ be a G.U.E. matrix with variance profile $\Gamma_{N,M} = \Gamma_N \otimes \one_{M, M}$, where $\one_{M, M}$ is the matrix whose all entries are one. Hence the variance profile is constant on blocks of size $M \times M$ and we can write 
		$$X_{N,M} =\sum_{\substack{ i,j \in [N] \\ i',j' \in [M]}}  \gamma_{i,j} N^{-\frac 12}M^{-\frac 12}x_{i,j,i',j'}E_{i,j}\otimes E_{i',j'},$$
where the $x_{i,j,i',j'}$ are complex Gaussian random variables, i.i.d.\ up to the Hermitian symmetry, centered and such that $x_{i,j,i',j'}$ has variance $1$. 
	\item[-]  We denote by $Y_{N,M} = Y_N  \otimes \mbb I_{M}$, where $\mbb I_{M}$ is the $M\times M$ identity matrix. It is a deterministic matrix diagonal by blocks of size $M \times M$, so that $Y_{N,M}  = \sum_{i,i',j} Y_N(i,j) E_{i,j}\otimes E_{i',i'}$.
\end{description}
We set $H_{N,M} = X_{N,M} + Y_{N,M}$. To avoid ambiguity, we denote by $\Lambda_N$ a generic element of $\DN^+$ and $\Lambda_{M,N}$ a generic element of $\mrm D_{NM}(\mbb C)^+$. We consider the diagonal of the generalized resolvent
	\eq
		\begin{array}{cccc} 
		G_{H_{N,M}}:& \mrm D_{NM}(\mbb C)^{+} & \to & \mrm D_{NM}(\mbb C)^{-}\\
		& \Lambda_{N,M} & \mapsto & \Delta\Big[ \big( \Lambda_{N,M} - H_{N,M} \big)^{-1}\Big],
		\end{array}
	\qe
 and the deterministic function
\begin{equation} \label{eq:GHM}
		\begin{array}{cccc} 
		G_{H_{N}}^{\square,M}: & \mrm D_{N}(\mbb C)^{+} & \to & \mrm D_{N}(\mbb C)^{-}\\
		& \Lambda_{N} & \mapsto &(id \otimes    \frac 1M \Tr )\Big[ \esp\big[ G_{H_{N,M}}( \Lambda_{N} \otimes \mbb I_{M})\big]\Big].
		\end{array}
\end{equation}

Note first that since $\|G_{H_{N}}^{\square,M}(\Lambda_N) \|\leq \| (\Im m \, \Lambda)^{-1}\|,$ we know that up to a subsequence $G_{H_{N}}^{\square,M}(\Lambda_N)$ has a limit $G_{H_{N}}^{\square}(\Lambda_N)$ as $M$ goes to infinity for any $\Lambda_N \in \DN^-$. Moreover, the same computation as Lemma \ref{Lem:LipsFP} yields
\begin{equation}
\|G_{H_{N}}^{\square,M}(\Lambda_N) - G_{H_{N}}^{\square,M}(\Lambda_N')\| \leq \| (\Im m \, \Lambda_N)^{-1} \|  \| (\Im m \, {\Lambda_N'})^{-1} \|  \times \| \Lambda_N - \Lambda_N'\|. \label{eq:LipM}
\end{equation}
Letting $M$ going to infinity along a subsequence, this implies that the estimate \eqref{Eq:EstimDeterEq} is valid for any accumulation point $G_{H_{N}}^{\square}(\Lambda_N)$ of $G_{H_{N}}^{\square,M}(\Lambda_N)$. 

We shall now prove that $G_{H_{N}}^{\square,M}$ converges when $M$ goes to infinity to a solution of the fixed point problem.  Thanks to Lemma \ref{MasterIneq}, we may now apply Equality \eqref{Eq:DiagAsympSubProp} to the random matrix $H_{N,M}$: for any $\Lambda_{N,M} \in \mrm D_{NM}(\mbb C)$, we have
	\eqa\label{Eq:SubordNM}
		\esp\big[ G_{H_{N,M}}(\Lambda_{N,M}) \big] = \psi_{\Lambda_{N,M}}\Big( \esp\big[ G_{H_{N,M}}(\Lambda_{N,M})\big] \Big) + \Theta_{M,N},
	\qea
where 
	$$ \psi_{\Lambda_{M,N}} (G) = G_{Y_{M,N}}\big(  \Lambda_{M,N} -\mcal R_{N,M}(G) \big)^{-1},$$
for any $G \in \mrm D_{NM}(\mbb C)$, and we have the estimates $\|\Theta_{M,N}\|  \leq  2\gamma_{\max}^3 N^{-1} \times  \| (\Im m \, \Lambda)^{-1}\|^5,$ and $\|\Theta_{M,N}\|  \leq  \gamma_{\max}^4 N^{-\frac 3 2} \times  \|(\Im m \, \Lambda)^{-1}\|^6$ if $Y_N$ is diagonal. Note that $\gamma_{\max}^2$ is indeed the maximum of the variances in the profile $\Gamma_{N,M}$. Moreover, because of the definition of $\Gamma_{N,M}$ the map $\mcal R_{N,M}$ is given by: for any $G \in \mrm D_{NM}(\mbb C)$, for any $i\in [N], i' \in [M]$
	\eq
		\mcal R_{N,M}(G)(i,i',i,i') & =  &   \sum_{j,j'} \Big( \frac{ \Gamma_{N}(i,j) \times \one_{M,M}(i',j') }{NM}\Big) \times G(j,j',j,j') \\
		& = &  \sum_{j \in [N]} \Big( \frac{ \Gamma_{N}(i,j)  }{N}\Big) \frac 1 M \sum_{j'\in [M]} G(j,j',j,j').
	\qe
	Since the above expression does not depends on $i'$, this proves that
	$$\mcal R_{N,M}(G) = \mcal R_{N}\Big(\big( id \otimes \frac 1 M \Tr \big) (G)\Big)\otimes \mbb I_M.$$
	
	Moreover note that if $\Lambda_{N,M} = \Lambda_N \otimes  \mbb I_M$ then $G_{Y_{N,M}}(\Lambda_{N,M}) = G_{Y_N}(\Lambda_N) \otimes  \mbb I_M$. Hence we get for any $G\in \mrm D_{NM}(\mbb C)$
	$$\psi_{\Lambda_N \otimes  \mbb I_M}(G)=\psi_{\Lambda_N}\Big(\big( id \otimes \frac 1 M \Tr \big) (G)\Big)\otimes  \mbb I_M.$$
Hence, applying $id \otimes \frac 1 M \Tr$ in \eqref{Eq:SubordNM} yields the following formula:
	$$ G_{H_{N}}^{\square,M}(\Lambda_N)= \psi_{\Lambda_N}\big( G_{H_{N}}^{\square,M}(\Lambda_N)\big)+ (id \otimes \frac 1 M \Tr) \big( \Theta_{M,N}\big).$$
Since $\|\Theta_{M,N} \|$ goes to zero as $M$ goes to infinity and by the continuity of $\psi_{\Lambda_N}$, letting $M$ go to infinity proves that any accumulation point $G_{H_{N}}^{\square}$ of $G_{H_{N}}^{\square,M}$ is solution of the fixed point problem that satisfies \eqref{Eq:EstimDeterEq} thanks to Inequality \eqref{eq:LipM}.

\subsubsection{Analiticity} \label{sec:analiticity}

Let us justify  that the quantities under consideration up to now are analytic functions.

\begin{Lem}\label{Lem:Analytic} For any Hermitian random matrix $M$, the function $\Lambda \mapsto  \esp\big[ (\Lambda-M)^{-1}\big]$ is analytic on $ \DN^+$. 
\end{Lem}

\begin{proof} Let us first assume that the matrices are deterministic. We prove the lemma by induction of the size $N$ of the matrices. Since for any $m\in \mbb R$ the map $\lambda \mapsto (\lambda - m)^{-1}$ is analytic on $\mbb C^+$, the lemma is true for $N=1$. From now we fix $N\geq 2$ and we assume the lemma is true for all deterministic Hermitian matrices of size $N-1$.

 We write $\Lambda= \mrm{diag}(\lambda_1 \etc \lambda_N)$ and recall that we denote $A_N = (\Lambda - M)^{-1}$. For any $k\in [N]$, let $A_{N-1}^{(k)}$ be the inverse of the $N-1$ by $N-1$ matrix obtained from $ (\Lambda - M)$ by removing the $k$-th line and column. Let $\mbf m^{(k)}$ be the vector of size $N-1$ obtained from the $k$-th column of $M$ by removing the $k$-th entry. Recall the Schur complement formula \cite[Appendix A.1.4]{BS}: 
	\eq
		(\Lambda - M)^{-1}(k,k)  = A_N(k,k) =&   \Big(    \lambda_k - M(k,k)- {\mbf m^{(k)}}^* A_{N-1}^{(k)}\mbf m^{(k)} \Big)^{-1}, \ \forall k \in [N] \\
		(\Lambda - M)^{-1}(k,\ell)  =&  - \big( A_{N-1}^{(k)} {\mbf m^{(k)}}^*\big)(\ell) \times  A_N(k,k), \ \forall k>\ell \in [N]\\
		(\Lambda - M)^{-1}(k,\ell)  =&  - \big(  {\mbf m^{(k)}} A_{N-1}^{(k)}\big)(\ell) \times  A_N(k,k), \ \forall k<\ell \in [N].
	\qe
By induction hypothesis, $\Lambda^{(k)}\mapsto A_{N-1}^{(k)}$ is analytic on $\mrm D_{N-1}(\mbb C)^+$. By Lemma \ref{Eq:PositivityLambdaPrime} we have
	$$\Im m \Big(  \lambda_k - M(k,k)- {\mbf m^{(k)}}^* A_{N-1}^{(k)}\mbf m^{(k)} \Big) = \Im m\, \lambda_k - {\mbf m^{(k)}}^* \big(  \Im m\, A_{N-1}^{(k)} \big) \mbf m^{(k)}  \geq  \Im m\, \lambda_k$$
Hence the maps $\Lambda \mapsto (\Lambda - M)^{-1}(k,k) $ are analytic in each variable for each $k=1\etc N$, and hence so are the maps $\Lambda\mapsto (\Lambda - M)^{-1}(k,\ell)$ for any $k,\ell$. 

Let now assume that $M$ is random. Each realization $\Lambda \mapsto  (\Lambda-M)^{-1}$ is analytic and the map is bounded. Hence $\Lambda \mapsto  \esp\big[ (\Lambda-M)^{-1}\big]$ is also analytic. 
\end{proof}

Hence the map $ G_{H_{N}}^{\square,M}$ defined by \eqref{eq:GHM} is indeed analytic. Since it is Lipschitz by Inequality \eqref{eq:LipM}, it follows that every accumulation point $ G_{H_{N}}^{\square}$ of the sequence is also analytic. This finishes the proof of Lemma \ref{Lem:UniquAnalytic}.

\subsection{Stability of the fixed point equation and proof of Lemma  \ref{StabilityAnalysis}}

Let $G_{H_N}^\square:\mrm D_N(\mbb C^+) \to \mrm D_N(\mbb C)^-$ be the deterministic equivalent, unique analytic solution of the fixed point problem
	$$G_{H_N}^\square(\Lambda)  = G_{Y_N}\Big(  \Lambda  -\mcal R_N\big(G_{H_N}^\square(\Lambda)\big)    \Big).$$
For reading convenience,  we recall that $\esp\big[ G_{H_N}(\Lambda) \big]= \esp\Big[ \Delta\big[(\Lambda - H_N)^{-1}\big]\Big] $ satisfies the approximate subordination property, namely
\begin{equation}
\esp\big[ G_{H_N}(\Lambda)\big]  = G_{Y_N}\Big(   \Lambda  -\mcal R_N\big(  \esp\big[G_{H_N}(\Lambda)\big]\big)    \Big) + \Theta_N(\Lambda), \label{eq:approxsub}
\end{equation}
where $\Theta_N(\Lambda) = \Delta \big[ (\Omega_{H_N}(\Lambda)-Y_N)^{-1}E_N(\Lambda) \big]$. The operator norm of $\Theta_N(\Lambda)$ satisfies by Lemma \ref{MasterIneq} $\| \Theta_N(\Lambda)\| \leq C_N$ where 
	\begin{equation}
		C_N= 2\gamma_{\max}^3 N^{-1} \times  \| (\Im m \, \Lambda)^{-1}\|^4 \label{eq:boundtheta}
	\end{equation}
 in general, and 
 	\begin{equation}
		C_N = \gamma_{\max}^4 N^{-\frac 3 2} \times  \|(\Im m \, \Lambda)^{-1}\|^5\label{eq:boundtheta2}
	\end{equation}
if $Y_N$ is diagonal. The purpose of this section is to prove Lemma \ref{StabilityAnalysis}, giving an estimate for the norm of the difference $ \big\| G_{H_N}^\square(\Lambda)  - \esp\big[ G_{H_N}(\Lambda)\big] \big\|$. 

We define two diagonal matrices by 
	\eq
		\tilde {G}_N(\Lambda)& = &   \esp\big[ G_{H_N}(\Lambda)\big] - \Theta_N( \Lambda),\\
	 \tilde {\Lambda} & = & \Lambda - \mcal R_N \big(  \Theta_N( \Lambda) \big) =  \Lambda - \mcal R_N \big( \esp\big[ G_{H_N}(\Lambda)\big]-\tilde{G}_N(\Lambda)\big).
	\qe
Provided we can justify that $\tilde{\Lambda}$ belongs to $\mrm{D}_N(\mbb C)^+$, we have
\begin{equation}
\tilde {G}_N(\Lambda) = G_{Y_N}\Big(   \Lambda  -\mcal R_N\big(  \esp\big[G_{H_N}(\Lambda)\big]\big)    \Big)=  G_{Y_N}\Big(  \tilde \Lambda  -\mcal R_N\big( \tilde G_{N}(\Lambda)\big)    \Big)=\psi_{\tilde{\Lambda}}\big( \tilde {G}_N(\Lambda)\big). \label{eq:tildeG}
\end{equation}
With $\mbb I_N$ denoting the identity matrix and using Lemma \ref{Lem:RNProp} and the bound \eqref{eq:boundtheta} for $\|\Theta_N(\Lambda)\|$, we have
	\eqa
		\Im m\, \tilde {\Lambda} & = & \Im m \, \Lambda - \Im m \, \mcal R_N \big(  \Theta_N( \Lambda) \big) \nonumber\geq  \Im m \, \Lambda - \| \mcal R_N \big(  \Theta_N( \Lambda) \big) \| \times \mbb I_N\\
			& \geq & \Im m \, \Lambda - \gamma_{\mrm{max}}^2 C_N \times \mbb I_N, \label{Eq:LambdaTilde}
	\qea
where $C_N$ is as in \eqref{eq:boundtheta}, or as in \eqref{eq:boundtheta2} when $Y_N$ is diagonal. Assuming that 
\begin{equation}
2\gamma_{\mrm{max}}^5 N^{-1} \|(\Im m \, \Lambda^{-1})\|^5 < 1, \label{eq:condLambda1}
\end{equation}
we indeed have $\tilde{\Lambda} \in \mrm{D}_N(\mbb C)^+$ and so $\tilde G_N(\Lambda)$ is solution of the fixed point problem for $\psi_{\tilde \Lambda}$. If $Y_N$ is diagonal, the same conclusion holds whenever
\begin{equation}
	\gamma_{\mrm{max}}^6 N^{-\frac 3 2} \| (\Im m \, \Lambda)^{-1}\|^6 < 1, \label{eq:condLambda1bis}.
\end{equation}

Hence, by Lemma \ref{Lem:UniquAnalytic} we obtain the equality $\tilde{G}_N(\Lambda) = G_{H_N}^\square( \tilde {\Lambda})$ and so, by Equalities \eqref{eq:approxsub} and \eqref{eq:tildeG}, we obtain
	\eqa\label{Eq:EqThetaAnalytic}
		\esp\big[ G_{H_N}(\Lambda)\big]  = G_{H_N}^\square( \tilde {\Lambda})+\Theta_N( \Lambda).
	\qea
By Lemma \ref{Lem:Analytic}, the map is $\Lambda \mapsto \esp\big[ G_{H_N}(\Lambda)\big]$ is analytic on $\DN^+$. Recall that $\Theta_N( \Lambda)= \Delta\big[ ( \Omega_{H_N}(\Lambda) - Y_N)^{-1} E_N \big]$ where $E_N$ is defined in \eqref{Eq:DefEN} and $\Omega_{H_N}(\Lambda)$ is defined in \eqref{eq:Omega}. One checks easily that the map $\Lambda \mapsto \Theta_N(\Lambda)$ is analytic on $\DN^+$, which implies that so are $\Lambda \mapsto \tilde {G}_N(\Lambda)$ and $\Lambda \mapsto \tilde {\Lambda} $. 
Hence Equality \eqref{Eq:EqThetaAnalytic} extends by analyticity for all $\Lambda>0$ such that $\tilde \Lambda>0$ and we get
\begin{equation}
\|\esp\big[ G_{H_N}(\Lambda)\big] - G_{H_N}^\square(\Lambda) \| \leq \| G_{H_N}^\square(\tilde{\Lambda}) - G_{H_N}^\square(\Lambda)\| + \| \Theta_N(\Lambda)\|. \label{eq:estim1}
\end{equation}
Moreover, with the same proof as for \eqref{Eq:EstimDeterEq}, we have the estimate
\begin{equation}
		\| G_{H_N}^\square(\tilde{\Lambda}) - G_{H_N}^\square(\Lambda)\| \leq  \| (\Im m \, \Lambda)^{-1} \| \  \| (\Im m \, \tilde{\Lambda})^{-1} \| \ \| \Lambda - \tilde {\Lambda}\|.  \label{eq:estim2}
\end{equation}
We have by Lemma \ref{Lem:RNProp},  $\| \Lambda - \tilde {\Lambda}\|  = \| \mcal R_N( \Theta_N(\Lambda) ) \| \leq \gamma^2_{\mrm max} \| \Theta_N(\Lambda)\|.$ Moreover, under the assumption 
\begin{equation}
\gamma_{\mrm{max}}^2 C_N \leq (1 - \delta) \Im m \, \Lambda, \mbox{ for some } 0 < \delta < 1, \label{eq:condLambda2ter}
\end{equation}
we obtain by \eqref{Eq:LambdaTilde} that
	\eq
		 \Im m \, \tilde{\Lambda} & \geq &  \delta \|(\Im m \, \Lambda)^{-1}\|^{-1}  \mbb I_N. 
	\qe
Hence provided that Condition \eqref{eq:condLambda2ter} holds, we have $ \| (\Im m \, \tilde{\Lambda})^{-1} \| \leq  \| (\Im m \, {\Lambda})^{-1} \|/\delta$. Combining \eqref{eq:estim1} with \eqref{eq:estim2} and \eqref{eq:boundtheta}, the previous estimates on $\| \Lambda - \tilde {\Lambda}\|$ and  $ \| (\Im m \, \tilde{\Lambda})^{-1} \| $ give
	\eq
		\|\esp\big[ G_{H_N}(\Lambda)\big] - G_{H_N}^\square(\Lambda) \| 
			&\leq &   \| (\Im m \, {\Lambda})^{-1} \|^2/\delta \  \gamma^2_{\mrm {max}} \| \Theta_N(\Lambda)\| +  \| \Theta_N(\Lambda)\|\\
			& \leq &  \big( 1 + \| (\Im m \, {\Lambda})^{-1} \|^2/\delta \  \gamma^2_{\mrm {max}} \big)  C_N,
	\qe
	which completes the proof of Lemma  \ref{StabilityAnalysis}.

\section{Analysis of the resolvent}\label{Sec:resolventAnalysis}

We now have all the ingredients to control the difference between the resolvent $(\lambda \mbb I_N - H_N)^{-1}$ and the deterministic equivalent $G_{H_N}^\square(\lambda \mbb I_N )$ which will finally complete the proof  of Theorem \ref{MainTh}.

\subsection{Expectation out of the diagonal}

 We first establish results allowing to show that the expectation of the resolvent is a diagonal matrix. Recall that a random matrix $A$ is unitarily invariant whenever $UAU^*$ has the same law as $A$ for any unitary matrix $U$. 

\begin{Lem}\label{Lem:ExpOutDiag} Let $A$ be a $N$ by $N$ unitarily invariant random matrix whose entries have finite moment of any orders and let $\Sigma \in \MN$
Then for any $n\geq 1$, the matrix $\esp\big[(\Sigma \circ A)^{n}\big]$ is diagonal.
\end{Lem}

\begin{proof}
 For any $i_1,i_{n+1}\in [N]$ we have
	\eqa\label{Eq:ExpOutDiag}
		\esp\big[ (\Sigma \circ A)^{n}\big](i_1,i_{n+1}) = \sum_{i_2\etc i_{n}=1}^N \Big( \prod_{k=1}^n\sigma(i_k,i_{k+1}) \Big) \times \esp\Big[ \prod_{k=1}^n A(i_k,i_{k+1})\Big],
	\qea
where $\sigma(i_k,i_{k+1})$ denotes the $(i_k,i_{k+1})$-th entry of the matrix $\Sigma$.

We shall prove that for any $i_2\etc i_n\in [N]$ then $\esp\Big[ \prod_{k=1}^n A(i_k,i_{k+1})\Big]=0$ when $i_1\neq i_{n+1}$. For this purpose we introduce a matrix function $A_t$ that depends on an implicit parameter $G$: for any anti-Hermitian matrix $G$, i.e. such that $G^*=-G$, and for any $t\in \mbb R$, we denote $A_t = e^{tG} A e^{-tG}$. Note that $A_0=A$ and the derivative of $A_t$ with respect to $t$ is $\partial_t A_t = GA_t - A_t G$. Moreover, the unitary invariance of $A$ implies that $A$ and $A_t$ have the same law. In particular for any $i_1\etc i_{n+1}\in [N]$ and any $t\in \mbb R$ we have
	$$\esp\Big[ \prod_{k=1}^n A(i_k,i_{k+1})\Big] = \esp\Big[ \prod_{k=1}^n A_t(i_k,i_{k+1})\Big].$$
We now differentiate the above equality with respect to $t$ and take $t=0$: using Leibniz formula,
	\eqa\label{Eq:ExpOutDiag2}
		0 & = & \esp\Big[ \partial_t\big( \prod_{k=1}^n A_t(i_k,i_{k+1})\big)_{|t=0}\Big] \\
		&= &  \sum_{k=1}^n  \esp\Big[ A(i_1,i_{2}) \cdots  A(i_{k-1},i_{k})\big( GA - AG \big) (i_{k},i_{k+1})A(i_{k+1},i_{k+2}) \cdots  A(i_{n},i_{n+1})\Big].\nonumber
	\qea
Recall that the above equality is a priori valid under the assumption that $G$ is anti-Hermitian. But the relation is linear in $G$, and the set of anti-Hermitian matrices spans $\MN$ as a vector space (any matrix $A$ can be written as a linear combination of Hermitian matrices $A = \Re e \, A + \mbf i \Im m \, A$ and a matrix $G$ is Hermitian whenever $\mbf i G$ is anti-Hermitian). We can hence specify the equality for the elementary matrix $G=E_{i_1,i_1}$. Note that for any $k\in [N]$, we have
	$$\big( GA - AG \big) (i_{k},i_{k+1}) 
	= \big(\delta_{i_1,i_k} - \delta_{i_1,i_{k+1}}\big) A(i_k,i_{k+1}).$$
Hence we obtain from Equation \eqref{Eq:ExpOutDiag2} a telescopic sum
	\eq 
		0  &  = &  \esp\big[ A(i_1,i_{2}) \cdots   A(i_{n},i_{n+1})\big]\times \sum_{k=1}^n  \big(\delta_{i_1,i_k} - \delta_{i_1,i_{k+1}}\big)\\
		& = & \esp\big[ A(i_1,i_{2}) \cdots   A(i_{n},i_{n+1})\big](1 - \delta_{i_1,i_{n+1}}).
	\qe
If $i_1\neq i_{n+1}$ then we get $\esp\big[ \prod_{k=1}^n A(i_k,i_{k+1})\big]=0$ for any $i_2\etc i_n$ and hence by Equation \eqref{Eq:ExpOutDiag} we deduce that $\esp\big[ (\Sigma \circ A)^{n}\big] $ is a diagonal matrix.

\end{proof}

\begin{Cor}\label{Cor:ExpOutDiag} Let $A$ and $\Sigma$ be as in Lemma \ref{Lem:ExpOutDiag}. Assume moreover that $ \Sigma \circ A$ is Hermitian. Then for any $\Lambda \in \DN^+$, the expectation of the generalized resolvent $\esp\big[ (\Lambda - \Sigma \circ A)^{-1} \big]$ is a diagonal matrix.
\end{Cor}

\begin{proof}
 We first assume that there is a constant $B>0$ such that almost surely one has $\| A\| \leq B$. For any matrix $M$ we denote by $\|M\|_F$ its Frobenius norm $\big\| M\big\|_F =  \big( \Tr \big[M^* M \big] \big)^{\frac{1}{2}}$ and we recall that $\| M \| \leq \| M \|_F \leq \sqrt N \| M\|$. Hence we have for $M=\Sigma \circ A$,
 	\eq
		\|  \Sigma \circ A \| &\leq & \Big( \Tr \Big[\big( \Sigma \circ A\big)^*\big( \Sigma \circ A\big) \Big] \Big)^{\frac 12}\\
		& = & \Big(  \sum_{i,j} |\sigma(i,j)|^2 |A(i,j)|^2  \Big)^{\frac 12}\\
			& \leq & \sigma_{\max} \times \| A \|_F \leq  \sigma_{\max} \sqrt N  \| A \|.
	\qe
	where $\sigma(i,j)$ denotes the $(i,j)$-th entry of the matrix $\Sigma$, and $\sigma_{\max} = \max_{i,j} |\sigma(i,j)|$.
For any $\Lambda \in \DN^+$ such that $\Im m \, \Lambda >  \sigma_{\max} \sqrt N  B \mbb I_N$, we have that $\| (\Sigma \circ A) \Lambda^{-1}\|<1$, and thus the following identity holds
	$$  (\Lambda - \Sigma \circ A)^{-1}  = \sum_{n\geq 0} \Lambda^{-1}   \big((\Sigma \circ A) \Lambda^{-1}\big)^n,$$
where the convergence of the sum is normal. In particular we can interchange summation and expectation, namely
	$$\esp\big[ (\Lambda - \Sigma \circ A)^{-1} \big] = \sum_{n\geq 0} \Lambda^{-1}   \esp\Big[\big((\Sigma \circ A) \Lambda^{-1}\big)^n\Big].$$
Moreover, we have the equality $(\Sigma \circ A) \Lambda^{-1} = \Sigma' \circ A$ where $\Sigma' = \Sigma\Lambda^{-1}$. Hence by Lemma \ref{Lem:ExpOutDiag} for any $\Lambda \in \DN^+$ and any $n\geq 1$ the matrix   $\esp\big[\big((\Sigma \circ A) \Lambda^{-1}\big)^n\big]$ is diagonal. So for any $\Lambda$ such that $\Im m \, \Lambda> \sigma_{\max} \sqrt N  B \mbb I_N$, the matrix $\esp\big[ G_{X_N}(\Lambda) \big]$ is also diagonal. This fact extends for any $\Lambda \in \DN^+$ by analytic continuation thanks to Lemma \ref{Lem:Analytic}.

To treat the general case, we use a classical spectral truncation argument. Let us denote by $\lambda_1\etc \lambda_N$ and $u_1\etc u_N$ the eigenvalues and the associated eigenvectors of $A$, so that we have $A = \sum_{i=1}^N \lambda_j u_i^* u_i$. For any $B>0$, we denote by $A^{(B)}$ the matrix $A^{(B)} = \sum_{i=1}^N \lambda_j \one\big( |\lambda_j|\leq B \big)u_i^* u_i$, where $\one$ denote the indicator function. Hence $A^{(B)}$ is uniformly bounded in operator norm by $B$. By the previous case, the matrix $\esp\big[ (\Lambda - \Sigma \circ A^{(B)})^{-1} \big] $ is diagonal. Moreover, we have
	\eq
		 \Big\|   (\Lambda - \Sigma \circ A)^{-1}    -   (\Lambda - \Sigma \circ A^{(B)})^{-1}   \Big\|  
		& \leq & \| (\Im m \, \Lambda^{-1})\|^2   \big\| \Sigma \circ \big( A^{(B)} - A \big) \big\|  \\
		& \leq &  \| (\Im m \, \Lambda^{-1})\|^2 \sigma_{\max} \sqrt N \| A^{(B)}-A \|.
	\qe
With $N$ fixed, we get that almost surely as $B$ tends to infinity the matrix $(\Lambda - \Sigma \circ A^{(B)})^{-1}$ converges to $(\Lambda - \Sigma \circ A)^{-1}$. Since the matrices are bounded in operator norm, the convergence holds in expectation. In particular, $\esp\big[ (\Lambda - \Sigma \circ A)^{-1}\big]$ is the limit of a diagonal matrix so it is diagonal. 
\end{proof}

\subsection{Concentration argument and proof of the main results}

We now complete the proof of Theorem \ref{MainTh} by combining Lemma \ref{StabilityAnalysis} with a concentration argument for the resolvent $(\Lambda - X_N - Y_N)^{-1}$ towards its expectation $\esp\big[ (\Lambda - H_N)^{-1} \big]$.

\begin{Lem}\label{Lem:ConcenResolvent} 
Let $\Lambda \in \DN^+$. Then, for  any pair of unit vectors $v,w$ (that is $\| v \|_2 = \| w \|_2 = 1$), one has that, for all $t> 0$,
\begin{equation}
\Pr \big( \big| v^\ast( \left(\Lambda - H_N)^{-1} - \esp \big[ (\Lambda - H_N)^{-1}\big] \right)w \big| \geq t \big) \leq 4 \exp \left( -N \frac{ t^2 \| (\Im m \, \Lambda)^{-1}\|^{-4}}{2 \gamma_{\max}^2} \right), \label{eq:concen0}
\end{equation}
where  $\gamma_{\max}^2$ is the maximum of the variances in the profile $\Gamma_{N}$. Moreover, for any $\lambda \in \mbb C^+$ and all $t> 0$,
\begin{equation}
\Pr \big( \big| g_{H_N} (\lambda)   - \esp \big[ g_{H_N} (\lambda)  \big]  \big| \geq t \big) \leq 4 \exp \left( -N^2 \frac{ t^2 | \Im m \, \lambda|^4}{8 \gamma_{\max}^2} \right), \label{eq:concen1}
\end{equation}
where $g_{H_N}$ is the Stieltjes transform of $H_N $.
\end{Lem}

To prove this result, we use the following result of Gaussian concentration inequality for Lipschitz functions (see e.g.\ \cite[Theorem 5.6]{MR3185193}).

\begin{Th}\label{Concentration} Let $X=(X_1\etc X_n)$ be a vector of $n$ independent standard normal random variables. Let $f: \mbb R^n \to \mbb R$ be a $L$-Lipschitz function for the Euclidean norm of $\mbb R^n$. Then for all $t>0$ we have
	\eqa\label{Eq:Concentration}
		\mbb P\Big( f(X) - \esp\big[ f(X) \big] \geq t \Big) \leq e^{-t^2/(2L^2)}.
	\qea
\end{Th}

\begin{proof}
We let $\HN \subset \MN$ be the subset of Hermitian matrices. For two unit vectors $v,w$  and  any $\Lambda$ in $\mrm D_N(\mbb C^+)$ we consider the function 
	$$\begin{array}{cccc}
		\phi_{ \Lambda}^{v,w} : &\mrm \HN & \to &  \mbb C \\
			& A & \mapsto& v^*(\Lambda - A - Y_N)^{-1}w.
	\end{array}$$
In order to use Theorem \ref{Concentration} for the real and the imaginary parts of $\phi_{\Lambda}^{u,w}$, we shall estimate its Lipschitz constant. We denote by $\big\| A\big\|_F =  \big( \Tr \big[A^* A \big] \big)^{\frac{1}{2}}$ the Frobenius norm of a matrix $A$. We use the isomorphism between $\MN$ endowed with $\| \, \cdot \, \|_F$ and $\mbb R^{N^2}$ endowed as the Euclidean norm. Using Cauchy-Schwarz's inequality and Lemma \ref{Eq:PositivityLambdaPrime}, it follows that, for any $A, A'$ in  $\HN$,	 
\begin{eqnarray}
		| \phi_{ \Lambda}^{v,w}(A) - \phi_{ \Lambda}^{v,w}(A') |	&=&	  \Big| v^*  \big( \Lambda  - A-Y_N \big)^{ -1}   \Big( A-A'\Big)    \big( \Lambda  - A'-Y_N\big)^{ -1} w \Big| \nonumber \\
			&\leq&   \big\| v^*  \big( \Lambda  - A-Y_N \big)^{ -1} \big\|_2  \big\| \Big( A-A'\Big)    \big( \Lambda  - A'-Y_N\big)^{ -1} w  \big\|_2  \nonumber \\
			&\leq&   \big\|   \big( \Lambda  - A-Y_N \big)^{ -1} \big\| \big\| A-A' \big\| \big\| \big( \Lambda  - A'-Y_N\big)^{ -1}  \big\| \nonumber  \\
			& \leq &  \| (\Im m \, \Lambda)^{-1}\|^2 \big\| A-A' \big\|   \leq   \| (\Im m \, \Lambda)^{-1}\|^2 \big\| A-A' \big\|_F. \label{ineq:GaussLip}
\end{eqnarray}
Since $A$ can be decomposed in the basis \eqref{Eq:BaseF} of Hermitian matrices as
$
A = \sum_{i,j=1}^N \tilde{a}_{i,j}F_{i,j},
$
where the $\tilde{a}_{i,j}$ are reals, we have that $\phi_{ \Lambda}^{v,w}$ is a $L$-Lipschitz function (with Lipschitz constant $L = \| (\Im m \, \Lambda)^{-1}\|^2$) of the $N \times N$ matrix $\tilde{A} = (\tilde{a}_{i,j})$ having real entries. Now, recall the decomposition \eqref{eq:decompZN} of $X_N$ into the basis \eqref{Eq:BaseF}
\begin{equation}
X_N =\sum_{i,j=1}^N \frac{\gamma_{i,j}}{\sqrt{N}} \tilde{x}_{i,j}F_{i,j}, \label{eq:decompXN}
\end{equation}
where the $\tilde{x}_{i,j}$ are i.i.d. real Gaussian random variables that are centered and of variance one. By Inequality \eqref{ineq:GaussLip} and after a change of variable, we get that the function $(\tilde{x}_{i,j})_{i,j}\mapsto \phi_{ \Lambda}^{v,w}(X_N)$ is a $L$-Lipschitz function (that is complex-valued) with Lipschitz constant
$$
L = \| (\Im m \, \Lambda)^{-1}\|^2 \frac{\gamma_{\max}}{\sqrt{N}}.
$$

Hence \eqref{eq:concen0} follows by Gaussian concentration inequality, namely Equation \eqref{Eq:Concentration}.
Now, let $\lambda \in \mbb C^+$ and recall that the Stieltjes transform  $H_N = X_N + Y_N$ is the map
$$
g_{H_N} (\lambda) =  \frac 1 N \Tr \big[ G_{H_N}(\lambda \mbb I_N) \big] =  \frac 1 N \Tr \big[ (\lambda \mbb I_N - H_N)^{-1}\big], 
$$
and let us consider the mapping  $A \mapsto \phi_{\lambda}(A) :=   \Tr \big[ (\lambda \mbb I_N -A - Y_N)^{-1}\big]$ for $A \in \HN$. For any $A, A'$ in  $\HN$, we remark that
\begin{eqnarray}
\big| \phi_{\lambda}(A) - \phi_{\lambda}(A') \big| & = & \big|  \Tr \big[ \big( \lambda \mbb I_N   - A-Y_N \big)^{ -1}    ( A-A' )    \big( \lambda \mbb I_N   - A'-Y_N\big)^{ -1} \big] \big|  \nonumber \\
& = & \big|  \Tr \big[ \big( \lambda \mbb I_N   - A-Y_N \big)^{ -1}    \big( \lambda \mbb I_N   - A'-Y_N\big)^{ -1}    ( A-A' )  \big] \big| \label{eq:boundTrace}
\end{eqnarray}
To obtain an appropriate upper bound for \eqref{eq:boundTrace}, we use the following inequalities (see e.g.\ Lemma II.2 in \cite{Lasserre}): denoting by $\sigma_1(A) \leq \cdots \leq \sigma_N(A)$ the singular values of $A$,
\begin{equation}
 \sum_{i=1}^{N} \sigma_{n-i+1}(\Re e\, B)  \sigma_{i}(C) \leq \Re e \left( \Tr \big[ B C\big] \right) \leq \sum_{i=1}^{N} \sigma_{i}(\Re e\, B)  \sigma_{i}(C) \label{eq:Lasserre}
\end{equation}
which hold for any $B \in \MN$ and  $C \in \HN$. For two such matrices, Inequality \eqref{eq:Lasserre} combined with Cauchy-Schwarz's inequality implies that
$$
| \Re e \left( \Tr \big[ B C\big] \right) | \leq   |\sigma_{1}(\Re e\, B)| \sum_{i=1}^{N} |\sigma_{i}(C)|  \leq  \| \Re e\, B\| \sqrt{N}  \sqrt{ \sum_{i=1}^{N} |\sigma_{i}(C)|^2}.
$$
Since $\| C \|_F^2 =  \sum_{i=1}^{N} |\sigma_{i}(C)|^2$ and $  \| \Re e\, B\| \leq \| B \|$ (by the same argument as for the imaginary part in the proof of Lemma \ref{Lem:NormEstim}), one finally obtains that
$$
| \Re e \left( \Tr \big[ B C\big] \right) | \leq  \sqrt{N} \| B \| \| C \|_F.
$$
Using the fact $\Im m \, A  =  - \Re e (\mbf i A )$,  one obtains by similar arguments that
$$
| \Im m \left( \Tr \big[ B C\big] \right) | \leq  \sqrt{N} \| B \| \| C \|_F,
$$
which finally yields
$$
| \Tr \big[ B C\big]   | \leq | \Re e \left( \Tr \big[ B C\big] \right) | + | \Im m \left( \Tr \big[ B C\big] \right) | \leq 2  \sqrt{N} \| B \| \| C \|_F
$$
Hence, combining the above inequality  with  \eqref{eq:boundTrace} and Lemma \ref{Eq:PositivityLambdaPrime}, it follows that
$$
\big| \phi_{\lambda}(A) - \phi_{\lambda}(A') \big| \leq 2  \sqrt{N} | \Im m \, \lambda |^{-2}   \big\| A-A' \big\|_F
$$
Therefore, thanks to the decomposition \eqref{eq:decompXN} for $X_N$, the mapping $\tilde{X}_N \mapsto \frac 1 N \Tr \big[ (\lambda \mbb I_N - X_N-Y_N)^{-1}\big] =  \frac 1 N   \phi_{\lambda}(X_N)  $ is a $L$-Lipschitz function   with Lipschitz constant
$$
L = 2 | \Im m \, \lambda |^{-2} \frac{\gamma_{\max}}{N}.
$$
 
Therefore, using again Gaussian concentration for Lipschitz functions, one obtains Inequality \eqref{eq:concen1}, which completes the proof of Lemma \ref{Lem:ConcenResolvent}.
\end{proof}

Now, let us fix $0 < \delta < 1$, and consider $\Lambda \in \DN^+$ satisfying $2\gamma^5_{\mrm{max}} N^{-1}\| (\Im m \, \Lambda)^{-1}\|^5 \leq 1-\delta$ so that   $ \| \esp \big[ G_{H_N}(\Lambda) \big]   - G^\square_{H_N}(\Lambda)\|  \leq t_N^{(1)}$, by Lemma \ref{StabilityAnalysis}, where
	\eqa\label{Def:TildeT1}
	t_N^{(1)} :=    \Big( 1 +  \frac{\gamma^2_{\mrm {max}}}{\delta}   \| (\Im m \, \Lambda)^{-1} \|^2 \Big)  \frac{2 \gamma_{\mrm{max}}^3\| (\Im m \, \Lambda)^{-1}\|^4}N.
	\qea
Since $G_{H_N}(\Lambda)$ and its expectation are diagonal matrices one has that the operator norm of their difference  satisfies
$$
\big\|G_{H_N}(\Lambda) - \esp \big[ G_{H_N}(\Lambda)  \big] \big\| = \max_{1 \leq i \leq N} \big|(\Lambda - H_N)^{-1}[i,i] -  \esp \big[  (\Lambda - H_N)^{-1}[i,i] \big]  \big|.
$$
Thus, combining Inequality \eqref{eq:concen0}  with a union bound yields the following concentration inequality: for all $t> 0$,
$$
\Pr \big( \big\|G_{H_N}(\Lambda) - \esp \big[ G_{H_N}(\Lambda)  \big] \big\| \geq t \big) \leq 4 N \exp \left( -N \frac{ t^2 \| (\Im m \, \Lambda)^{-1}\|^{-4}}{2 \gamma_{\max}^2} \right).
$$
Hence, taking $t = t_N^{(2)} :=\sqrt{2} \gamma_{\max} \sqrt{   d \log(N) }  \| (\Im m \, \Lambda)^{-1}\|^{2}  N^{-1/2} $ (for some $d > 1$), one finally obtains that
$$
\Pr \big( \big\| G_{H_N}(\Lambda)   - G^\square_{H_N}(\Lambda)  \big\| \geq t_N^{(2)} + t_N^{(1)} \big) \leq 4 N^{1-d},
$$
which proves Inequality \eqref{eq:MainThStieltjes},  and completes the proof of Theorem \ref{MainTh} in the general case. The case diagonal where $Y_N$ is diagonal  follows similarly.

Then, to derive the proof of Corollary \ref{MainCor}, we use the concentration inequality \eqref{eq:concen1} for the Stieltjes transform $g_{H_N} (\lambda) $. Since $g_{H_N}^\square (\lambda) =  \frac 1 N \Tr \big[ G_{H_N}^\square(\lambda \mbb I_N) \big]$ and given that $| \frac 1 N \Tr \big[ A \big]| \leq \| A \|$ for any diagonal matrix $A \in \MN$, we obtain  from Lemma \ref{StabilityAnalysis}  that,  for $\lambda$ satisfying Condition \eqref{eq:condlambda2bis},
\begin{equation}
\big| \esp \big[ g_{H_N} (\lambda) \big] - g_{H_N}^\square (\lambda) \big|  \leq  \tilde{t}_N^{(1)}, \quad \mbox{where} \quad \tilde{t}_N^{(1)} :=  \Big( 1 +  \frac{\gamma^2_{\mrm {max}}}{\delta | \Im m \, \lambda|^{2}}    \Big) \frac{2 \gamma_{\mrm{max}}^3}{ N   (\Im m \, \lambda)^4}. \label{eq:concen2}
\end{equation}
Therefore, taking $t = \tilde{t}_N^{(2)} := \frac{\sqrt{2} \gamma_{\max} \sqrt{2d \log(N)}}{  | \Im m \, \lambda|^2 N}   $ (for some $d > 0$) one obtains, by combining Inequalities \eqref{eq:concen1} and \eqref{eq:concen2}, that
$$
\Pr \big( \big| g_{H_N} (\lambda)   - g_{H_N}^\square (\lambda) \big| \geq \tilde{t}_N^{(2)} + \tilde{t}_N^{(1)} \big) \leq N^{-d},
$$
which proves Inequality \eqref{eq:MainThStieltjes2}. The case diagonal where $Y_N$ is diagonal also follows similarly, and this completes the proof of Corollary \ref{MainCor}.

We now prove Corollary \ref{Cor:Beta1}, assuming hence that $Y_N$ is a diagonal matrix. We take $\Lambda = \lambda \mbb I_N$  with satisfying $\Im m \, \lambda \geq \gamma_{\mrm{max}} N^{-1/4} (1-\delta)^{-1/6}$. Since $Y_N$ is supposed to be Hermitian, it is a diagonal matrix with real entries. Therefore,  by Corollary \ref{Cor:ExpOutDiag}, one has that $\esp\big[ (\lambda \mbb I_N - H_N)^{-1} \big]$ is  diagonal. Then, we remark that
\begin{equation}
 \| \beta_k(\lambda) - \beta_k^\square(\lambda) \| \leq \|    U_{N,k}^* \left( (\lambda \mbb I_N - H_N)^{-1} -G_{H_N}^\square( \lambda \mbb I_N) \right)  U_{N,k} \| \| \Theta_k \| \label{ineqconcenbeta0}
\end{equation}
Now, with $\hat{t}_N^{(1)} :=  \Big( 1 +  \frac{\gamma^2_{\mrm {max}}}{\delta | \Im m \, \lambda|^{2}}    \Big) \frac{ \gamma_{\mrm{max}}^4}{ N^{3/2}   (\Im m \, \lambda)^5}$, we have 
	\eqa
		\lefteqn{  \big\|    U_{N,k}^* \left(  \esp\big[ G_{H_N}( \lambda \mbb I_N)  \big] -G_{H_N}^\square( \lambda \mbb I_N) \right)  U_{N,k} \big\|}\nonumber\\
		& \leq & \big\|      \esp\big[ G_{H_N}( \lambda \mbb I_N)  \big] -G_{H_N}^\square( \lambda \mbb I_N)  \big\| \leq \hat{t}_N^{(1)}  \label{ineqconcenbeta1}.
	\qea
Moreover, if we denote by $u_1,\ldots,u_k$ the columns of the matrix $U_{N,k}$, we have
	\eq
		\lefteqn{ \big\| U_{N,k}^* \left( (\lambda \mbb I_N - H_N)^{-1} -\esp\big[ (\lambda \mbb I_N - H_N)^{-1} \big] \right)  U_{N,k}\big \| }\\
		 &\leq & k \max_{1 \leq \ell, \ell' \leq k }    u_{\ell}^* \left( (\lambda \mbb I_N - H_N)^{-1} -\esp\big[ (\lambda \mbb I_N - H_N)^{-1} \right) u_{\ell'}.
	\qe
Thus, by combining Inequality \eqref{eq:concen0} with the fact that $\esp\big[ (\lambda \mbb I_N - H_N)^{-1} \big]  =  \esp\big[ G_{H_N}( \lambda \mbb I_N)  \big]$ (since $ \esp\big[ (\lambda \mbb I_N - H_N)^{-1} \big] $ is diagonal), it follows  by a union bound argument that, for all $t > 0$, 
\begin{equation}
\Pr \left(  \frac{1}{k}\| U_{N,k}^* \left( (\lambda \mbb I_N - H_N)^{-1} -\esp\big[ G_{H_N}( \lambda \mbb I_N)  \big]\right)  U_{N,k} \|  \geq t \right) \leq 4 k^2 \exp \left( -N \frac{ t^2 | \Im m \, \lambda|^4}{2 \gamma_{\max}^2} \right). \label{ineqconcenbeta2}
\end{equation}
Therefore, combining Inequalities \eqref{ineqconcenbeta0}, \eqref{ineqconcenbeta1} and \eqref{ineqconcenbeta2} we obtain that 
$$
\Pr \left( \| \beta_k(\lambda) - \beta_k^\square(\lambda) \|  \geq  \| \Theta_k \|( k \bar{t}_N^{(2)} + \hat{t}_N^{(1)}) \right) \leq 4 k^2 N^{-d}, 
$$
with $\bar{t}_N^{(2)} := \sqrt{2} \gamma_{\max} \frac{ \sqrt{d \log(N)}}{  | \Im m \, \lambda|^2} N^{-1/2}$, which finally yields Inequality  \eqref{Goal}.

Let us now prove Corollary \ref{Cor:Beta2}, where $Y_N$ is not necessarily diagonal. In particular $ \esp\big[ (\lambda \mbb I_N - H_N)^{-1} \big] $ is not necessarily a diagonal matrix. Thus, we shall use the deterministic equivalent  $\tilde{\beta}_k^\square(\lambda)$ defined by \eqref{eq:defbetasquaregen} to approximate $ \beta_k(\lambda)$. First, as previously, we remark that
\begin{equation}
 \| \beta_k(\lambda) - \tilde{\beta}_k^\square(\lambda) \| \leq \|    U_{N,k}^* \left( (\lambda \mbb I_N - H_N)^{-1} -\big(\Omega_{H_N}^\square(\lambda \mbb I_N)-Y_N \big)^{-1}  \right)  U_{N,k} \| \| \Theta_k \| \label{ineqconcentildebeta0},
\end{equation}
where $\Omega_{H_N}^\square(\lambda \mbb I_N) =  \lambda \mbb I_N - \mcal R_N\big(  G_{H_N}^\square(\lambda \mbb I_N)  \big)$.
By the same arguments  used to derive inequality \eqref{ineqconcenbeta2}, it follows that
\begin{equation}
\Pr \left(  \frac{1}{k}\| U_{N,k}^* \left( (\lambda \mbb I_N - H_N)^{-1} -\esp\big[ (\lambda \mbb I_N - H_N)^{-1} \big] \right)  U_{N,k} \|  \geq t \right) \leq 4 k^2 \exp \left( -N \frac{ t^2 | \Im m \, \lambda|^4}{2 \gamma_{\max}^2} \right). \label{ineqconcentildebeta2}
\end{equation}

Hence, the only difference with the setting where $Y_N$ is diagonal is the control of the term 
$
 \| U_{N,k}^* \left(  \esp\big[ (\lambda \mbb I_N - H_N)^{-1} \big]  - \big(\Omega_{H_N}^\square(\lambda \mbb I_N)-Y_N \big)^{-1} \right)  U_{N,k} \|
$ which is obviously bounded by $\|  \esp\big[ (\lambda \mbb I_N - H_N)^{-1} \big]  - \big(\Omega_{H_N}^\square(\lambda \mbb I_N)-Y_N \big)^{-1} \|$. We now consider the decomposition 
\begin{eqnarray*}
\lefteqn{ \|  \esp\big[ (\lambda \mbb I_N - H_N)^{-1} \big]  - \big(\Omega_{H_N}^\square(\lambda \mbb I_N)-Y_N \big)^{-1} \| }\\& \leq & \|  \esp\big[ (\lambda \mbb I_N - H_N)^{-1} \big]  - \big(\Omega_{H_N}(\lambda \mbb I_N)-Y_N \big)^{-1} \|  \\
& & + \| \big(\Omega_{H_N}(\lambda \mbb I_N)-Y_N \big)^{-1}  -  \big(\Omega_{H_N}^\square(\lambda \mbb I_N)-Y_N \big)^{-1} \|
\end{eqnarray*}
where $ \big(\Omega_{H_N}(\lambda \mbb I_N)-Y_N \big)^{-1}  = \lambda \mbb I_N - \mcal R_N\Big(  \esp\big[G_{H_N}(\lambda \mbb I_N)\big] \Big)$. By combining Lemma \ref{Eq:PositivityLambdaPrime}, Lemma \ref{Lem:RecallASP} and Lemma \ref{MasterIneq}, we obtain that
\begin{equation}
 \|  \esp\big[ (\lambda \mbb I_N - H_N)^{-1} \big]  - \big(\Omega_{H_N}(\lambda \mbb I_N)-Y_N \big)^{-1} \|   \leq   \frac{2 \gamma_{\mrm{max}}^3}{ N   (\Im m \, \lambda)^4}.  \label{ineqconcentildebeta3}
\end{equation}
Now, using the equality
$
\big(\Omega_{H_N}(\lambda \mbb I_N)-Y_N \big)^{-1}  -  \big(\Omega_{H_N}^\square(\lambda \mbb I_N)-Y_N \big)^{-1}  =  \big(\Omega_{H_N}^\square(\lambda \mbb I_N)-Y_N \big)^{-1}  \left(\Omega_{H_N}^\square(\lambda \mbb I_N) -  \Omega_{H_N}(\lambda \mbb I_N) \right) \big(\Omega_{H_N}(\lambda \mbb I_N)-Y_N \big)^{-1}
$,
one has that
\begin{eqnarray*}
	\lefteqn{ \| \big(\Omega_{H_N}(\lambda \mbb I_N)-Y_N \big)^{-1}  -  \big(\Omega_{H_N}^\square(\lambda \mbb I_N)-Y_N \big)^{-1} \|}\\ & \leq & \| \big(\Omega_{H_N}^\square(\lambda \mbb I_N)-Y_N \big)^{-1} \|  \|  \big(\Omega_{H_N}(\lambda \mbb I_N)-Y_N \big)^{-1} \| \\
 & & \times \| \mcal R_N\big(  G_{H_N}^\square(\lambda \mbb I_N)  \big) -  \mcal R_N\Big(  \esp\big[G_{H_N}(\lambda \mbb I_N)\big] \Big)\|
\end{eqnarray*}
By Lemma \ref{Eq:PositivityLambdaPrime}  and Lemma \ref{Lem:RecallASP} one obtains that $ \|  \big(\Omega_{H_N}(\lambda \mbb I_N)-Y_N \big)^{-1} \| \leq   (\Im m \, \lambda)^{-1}$. Using again Lemma \ref{Eq:PositivityLambdaPrime}  one has that $\| \big(\Omega_{H_N}^\square(\lambda \mbb I_N)-Y_N \big)^{-1} \| \leq   \| \big( \Im m\, \Omega_{H_N}^\square(\lambda \mbb I_N)  \big)^{-1} \| $. Since $\Omega_{H_N}^\square(\lambda \mbb I_N) =  \lambda \mbb I_N - \mcal R_N\big(  G_{H_N}^\square(\lambda \mbb I_N)  \big)$ and given that $ \Im m\, G_{H_N}^\square(\lambda \mbb I_N) < 0$ by Lemma \ref{Lem:UniquAnalytic}, it follows from  Lemma \ref{Lem:RNProp} that $\Im m \big( \mcal R_N\big(  G_{H_N}^\square(\lambda \mbb I_N)  \big) \big)\leq 0$. Consequently, one has that $\Im m\, \Omega_{H_N}^\square(\lambda \mbb I_N)  > \Im m\, \lambda \mbb I_N$, and this finally yields $ \| \big(\Omega_{H_N}^\square(\lambda \mbb I_N)-Y_N \big)^{-1} \| \leq   (\Im m \, \lambda)^{-1}$.

Then,  using Lemma \ref{Lem:RNProp}, we remark that 
$$
\| \mcal R_N\big(  G_{H_N}^\square(\lambda \mbb I_N)  \big) -  \mcal R_N\Big(  \esp\big[G_{H_N}(\lambda \mbb I_N)\big] \Big)\| \leq \gamma_{\mrm{max}}^2  \| G_{H_N}^\square(\lambda \mbb I_N)  -  \esp\big[G_{H_N}(\lambda \mbb I_N)\big] \|,
$$
and therefore,  if $\Im m \, \lambda \geq  \gamma_{\mrm{max}}  \left(\frac{2}{N(1-\delta)}\right)^{1/5}$, then Lemma \ref{StabilityAnalysis} implies that
$$
\| \mcal R_N\big(  G_{H_N}^\square(\lambda \mbb I_N)  \big) -  \mcal R_N\Big(  \esp\big[G_{H_N}(\lambda \mbb I_N)\big] \Big)\| \leq  \Big( 1 +  \frac{\gamma^2_{\mrm {max}}}{\delta | \Im m \, \lambda|^{2}}    \Big) \frac{2 \gamma_{\mrm{max}}^5}{ N   (\Im m \, \lambda)^4}.
$$
Therefore, we  obtain that
\begin{equation}
 \| \big(\Omega_{H_N}(\lambda \mbb I_N)-Y_N \big)^{-1}  -  \big(\Omega_{H_N}^\square(\lambda \mbb I_N)-Y_N \big)^{-1} \|  \leq     \Big( 1 +  \frac{\gamma^2_{\mrm {max}}}{\delta | \Im m \, \lambda|^{2}}    \Big) \frac{2 \gamma_{\mrm{max}}^5}{ N   (\Im m \, \lambda)^6},  \label{ineqconcentildebeta4}
\end{equation}
Hence, the use of Inequalities    \eqref{ineqconcentildebeta3} and \eqref{ineqconcentildebeta4}  yields that
\begin{equation}
\|  \esp\big[ (\lambda \mbb I_N - H_N)^{-1} \big]  - \big(\Omega_{H_N}^\square(\lambda \mbb I_N)-Y_N \big)^{-1} \| \leq  \frac{2 \gamma_{\mrm{max}}^3}{ N   (\Im m \, \lambda)^4} +  \Big( 1 +  \frac{\gamma^2_{\mrm {max}}}{\delta | \Im m \, \lambda|^{2}}    \Big) \frac{2 \gamma_{\mrm{max}}^5}{ N   (\Im m \, \lambda)^6}.  \label{ineqconcentildebeta5} 
\end{equation}
Finally, combining  Inequalities \eqref{ineqconcentildebeta0}, \eqref{ineqconcentildebeta2} and \eqref{ineqconcentildebeta5} yields Inequality \eqref{Goal2}, which completes the proofs of the results stated in Section  \ref{Sec:MainStat}.

\subsection{Convergence of the fixed point algorithm}

A key step in these numerical experiments in the  numerical approximation of the solution of the fixed point equation \eqref{FixPtEq1} through an iterative algorithm whose convergence is first briefly discussed.

\begin{Lem} For any diagonal function $G^{(0)}_N: \DN^+ \to \DN^-$, we consider the sequence of diagonal functions $\big(G^{(n)}_N\big)_{n\geq0}$ given by $G^{(n+1)}_N(\Lambda) = \psi_{\Lambda}\big( G^{(n)}_N(\Lambda) \big)$ for any $\Lambda \in  \DN^+$. Then, as $n$ goes to infinity, the sequence $G^{(n)}_N(\Lambda)$ converges to $G^\square_{H_N}(\Lambda)$, where $G^\square_{H_N}$ denotes the deterministic equivalent characterized in Lemma \ref{Lem:UniquAnalytic}.
\end{Lem}

\begin{proof} Recall that by Lemma \ref{Lem:LipsFP} the sequence $\big(G^{(n)}_N(\Lambda)\big)_{n\geq0}$ is bounded, and so up to a subsequence it converges to some diagonal matrix $G_N(\Lambda)$. If $\Im m \, \Lambda> \gamma_{\max}\mbb I_N$, by Corollary \ref{Cor:Apriori} then $G_N(\Lambda)= G^\square_{H_N}(\Lambda)$ by contractivity of the fixed point problem. Moreover, for each $n$, the function $G^{(n)}$ is an analytic function of the variable $\Lambda$, uniformly bounded for $\Im m \, \Lambda>\eps$ for any $\eps$. By dominated convergence, the limit $G_N$ up to any subsequence is analytic. By analytic continuation, all limit coincide and are equal to $G^\square_{H_N}$. 
\end{proof}

\bibliographystyle{siam}
\bibliography{BiblioSpikesVarianceProfile}

\end{document}